\documentclass[11pt, oneside]{amsart}
\usepackage{amsfonts, amstext, amsmath, amsthm, amscd, amssymb}
\usepackage{manfnt}
\usepackage{mathrsfs}
\usepackage{url}
\usepackage[dvipsnames]{xcolor}
\usepackage[paper=a4paper, text={138mm,208mm},centering, marginparwidth=2.5cm]{geometry}

\definecolor{darkblue}{rgb}{0,0,0.6}
\usepackage{graphicx, pinlabel, color}
\usepackage[all,graph]{xy}
\usepackage{tikz-cd}
\usepackage[breaklinks, pdftex, ocgcolorlinks,colorlinks=true, citecolor=darkblue, filecolor=darkblue, linkcolor=darkblue, urlcolor=black]{hyperref}
\usepackage{enumerate}
\usepackage{mathtools}
\usepackage[capitalize,noabbrev]{cleveref}
\usepackage{comment}
\usepackage{faktor} 

\usepackage{tikz} 
\usetikzlibrary{calc}

\newtheorem{theorem}[equation]{Theorem}
\newtheorem{lemma}[equation]{Lemma}
\newtheorem{proposition}[equation]{Proposition}
\newtheorem{corollary}[equation]{Corollary}
\newtheorem{conjecture}[equation]{Conjecture}

\newtheorem*{claim*}{Claim}

\theoremstyle{definition}
\newtheorem{example}[equation]{Example}
\newtheorem{remark}[equation]{Remark}
\newtheorem{definition}[equation]{Definition}

\numberwithin{equation}{section}

\newcommand{\bp}{\begin{pmatrix}}
\newcommand{\ep}{\end{pmatrix}}
\newcommand{\be}{\begin{equation}}
\newcommand{\ee}{\end{equation}}
\newcommand{\ba}{\begin{array}}
\newcommand{\ea}{\end{array}}

\makeatletter
\newenvironment{step}[1][\proofname]{\par
  \normalfont \topsep6\p@\@plus6\p@\relax
  \trivlist
  \item[\hskip\labelsep
        \itshape
    #1\@addpunct{.}]\ignorespaces
}{
  \endtrivlist\@endpefalse
}

\AddToHook{env/proposition/begin}{\crefalias{equation}{proposition}}
\AddToHook{env/theorem/begin}{\crefalias{equation}{theorem}}
\AddToHook{env/corollary/begin}{\crefalias{equation}{corollary}}
\AddToHook{env/definition/begin}{\crefalias{equation}{definition}}
\AddToHook{env/lemma/begin}{\crefalias{equation}{lemma}}
\AddToHook{env/question/begin}{\crefalias{equation}{question}}
\AddToHook{env/example/begin}{\crefalias{equation}{example}}
\AddToHook{env/conjecture/begin}{\crefalias{equation}{conjecture}}
\AddToHook{env/remark/begin}{\crefalias{equation}{remark}}

\crefname{theorem}{Theorem}{Theorems}
\crefname{proposition}{Proposition}{Propositions}
\crefname{corollary}{Corollary}{Corollaries}
\crefname{definition}{Definition}{Definitions}
\crefname{lemma}{Lemma}{Lemmas}
\crefname{question}{Question}{Questions}
\crefname{example}{Example}{Examples}
\crefname{conjecture}{Conjecture}{Conjectures}
\crefname{remark}{Remark}{Remarks}

\def\Z{\mathbb Z}
\def\R{\mathbb R}
\def\Q{\mathbb Q}

\def\C{\mathbb C}

\def\wt{\widetilde}
\def\p{\partial}
\def\sm{\setminus}
\def\S{\Sigma}

\def\ol{\overline}

\def\RP{\mathbb{RP}}

\def\G{\Gamma}
\def\xra{\xrightarrow}

\DeclareMathOperator\Int{Int}

\DeclareMathOperator\coker{coker}
\DeclareMathOperator\im{Im}

\DeclareMathOperator\Id{Id}
\DeclareMathOperator\GL{GL}
\DeclareMathOperator\SO{SO}

\DeclareMathOperator\Homeo{Homeo}

\DeclareMathOperator\Aut{Aut}
\DeclareMathOperator\hAut{hAut}
\DeclareMathOperator\id{Id}

\DeclareMathOperator\PL{PL}
\DeclareMathOperator\Diff{Diff}

\DeclareMathOperator\PDiff{PDiff}
\DeclareMathOperator\BDiff{BDiff}
\DeclareMathOperator{\Out}{Out}
\DeclareMathOperator{\Inn}{Inn}
\DeclareMathOperator{\Emb}{Emb}
\DeclareMathOperator{\Isom}{Isom}
\DeclareMathOperator{\Image}{Im}
\DeclareMathOperator{\Norm}{Norm}
\DeclareMathOperator{\Twist}{Twist}
\DeclareMathOperator{\Twistns}{Twist_{\text{ns}}}

\def\H{\mathbb H}
\DeclareMathOperator{\SL}{SL}
\DeclareMathOperator{\PSL}{PSL}
\DeclareMathOperator{\vol}{vol}

\DeclareMathOperator{\isom}{Isom}

\begin{document}

\title{A survey on mapping class groups of 3-manifolds}

\author{Philipp Bader}
\email{p.bader.1@research.glasgow.ac.uk}

\author{Rachael Boyd}
\email{rachael.boyd@glasgow.ac.uk}

\author{Giulia Carfora}
\email{g.carfora.1@research.gla.ac.uk}

\author{Gabriel Corrigan}
\email{g.corrigan.1@research.gla.ac.uk}

\author{Daniel Galvin}
\email{daniel.galvin@austin.utexas.edu}

\author{Csaba Nagy}
\email{nagy@mpim-bonn.mpg.de}
\address{Max-Planck-Institut f\"{u}r Mathematik, Bonn, Germany}

\author{John Nicholson}
\email{john.nicholson@glasgow.ac.uk}

\author{Weizhe Niu}
\email{weizheniu@mail.tsinghua.edu.cn}
\address{Yau Mathematical Sciences Center, Tshinhua University, China}

\author{Isacco Nonino}
\email{2754452n@student.gla.ac.uk}

\author{Mark Pencovitch}
\email{m.pencovitch.1@research.gla.ac.uk}

\author{Mark Powell}
\email{mark.powell@glasgow.ac.uk}

\address{School of Mathematics and Statistics, University of Glasgow, United Kingdom}

\begin{abstract}
    We survey computations and tools concerning the mapping class group of a compact, oriented, connected 3-manifold $M$. We provide a guide to the literature and sketch proofs for various families of irreducible and geometric 3-manifolds. We also consider JSJ and prime decompositions of 3-manifolds, and consequences for their mapping class groups.
\end{abstract}
\maketitle

\setcounter{tocdepth}{1}
\tableofcontents

\section{Introduction}
Famously, as a result of Thurston's~\cite{Thurston-82,Thurston-Levy-book} and Perelman's~\cite{Perelman:2002-1,Perelman:2003-1,Perelman:2003-2} work, compact, orientable 3-manifolds are considered to have been classified. Apart from lens spaces (which were classified early on~\cite{Threlfall-Seifert,Reidemeister1935HomotopieringeUL,Whitehead-incidence,Moise}) such 3-manifolds are determined up to diffeomorphism by their fundamental group; see e.g.~\cite[Thm.~2.1.3]{AFW2015} for a detailed statement.  Alternatively, one take the prime decomposition of Kneser~\cite{Kneser1929} and Milnor~\cite{Milnor62}, then cut along JSJ tori~\cite{JacoShalen1976,Johannson-book} and classify the pieces: they are Seifert fibred or hyperbolic. Geometrisation provides a classification of the six Seifert fibred geometries. 

The question of determining mapping class groups of closed, orientable 3-manifolds is very well understood. However, while much has been written on the classification problem, heretofore knowledge on the computation of mapping class groups of 3-manifolds has been contained in multiple sources. We aim to provide a guide to the literature, and, by stating the key results with context, we hope to equip readers with the tools to understand the mapping class groups of closed, orientable 3-manifolds that they encounter. We will also give some information for compact 3-manifolds with nonempty boundary. We do not consider nonorientable 3-manifolds.

The results stated in our survey are not new, except for in \cref{sec: JSJ}, in which we give some methods for obtaining the mapping class group of an irreducible 3-manifold from the mapping class groups of its JSJ pieces. We are not aware of this material having made a prior appearance in the literature.

\subsection{Symmetries of $M$} \label{ss:sym-M}

Let~$M$ be a compact, oriented 3-manifold, always assumed henceforth, without further comment, to be smooth and connected. We can consider smooth, piecewise linear (PL), and continuous homeomorphisms~$M\to M$, and these form topological groups which fit together in the following sequence:
\begin{equation}\label{eq:3-term-SES}
\Diff(M)\to \PL(M) \to \Homeo(M).
\end{equation}
Here we consider~$\Diff(M)$ as a topological group, equipped with the $C^\infty$ or Whitney topology, and
we consider $\Homeo(M)$ equipped with the $C^0$ or compact-open topology.
The space $\PL(M)$ is defined as the geometric realisation of the simplicial space of $\PL$-homeomorphisms of $M$.

Note that the left hand map of \eqref{eq:3-term-SES} exists but it is not obvious how to define it. To do so, one must introduce the notion of a piecewise differentiable map. These form a group~$\PDiff(M)$ and there are maps~$\Diff(M)\to \PDiff(M)$ and $\PL(M) \overset{\simeq}{\to} \PDiff(M)$ induced by inclusion. The map~$\Diff(M)\to \PL(M)$ is then the composite of the first map with the homotopy inverse of the second map.

The resulting map $\Diff(M)\to \PL(M)$ is a homotopy equivalence by smoothing theory~\cite{Hirsch-Mazur}, and Hatcher's proof of the Smale conjecture \cite{Hatcher1983}.
Hence we will ignore $\PL(M)$ for the remainder of this discussion.

The map $\Diff(M) \to \Homeo(M)$ is a weak equivalence by work of Cerf~\cite{Cerf1959} (see \cite[Essay V]{KirbySiebenmann1977} for more details).

If $\partial M \neq \emptyset$ then we denote by $\Diff_\partial(M)$ the subgroup of $\Diff(M)$ consisting of diffeomorphisms which restrict to the identity on $\partial M$. In addition for $K\subseteq M$ a submanifold we introduce notation for the subgroup of diffeomorphisms fixing $K$ pointwise and setwise:
\begin{align*}
    \Diff_K(M) &= \{ f\in \Diff(M) \mid f|_K=\id_K\}\\
    \Diff(M,K) &= \{f\in \Diff(M) \mid f(K)=K  \}.
\end{align*}

Recall  $f$ is isotopic to $g$ if there exists a smooth map~$F \colon M \times I \to M$ such that $F_0=f$, $F_1=g$ and $F_t\colon M \to M \in \Diff(M)$ for all $t \in I$. We denote by $\sim$ and $\simeq$ the equivalence relations induced by isotopy and homotopy of maps respectively.

\begin{definition}
    The \emph{mapping class group} of $M$ is:
    \[
        \pi_0\Diff(M) = \Diff(M)/\sim.
    \]
\end{definition}
Let $\Diff_0(M)$ be the connected component of~$\Diff(M)$  that contains the identity (i.e.~all diffeomorphisms of~$M$ which are isotopic to the identity). There is a short exact sequence of groups:
\begin{equation*}
    \Diff_0(M) \to \Diff(M) \to \pi_0\Diff(M).
\end{equation*}
We can also consider the \emph{orientation-preserving} (o.p.) diffeomorphisms $\Diff^+(M)$ of $M$, and the corresponding o.p.\ mapping class group \[\pi_0 \Diff^+(M) := \Diff^+(M)/\sim.\]

Replacing $\Diff(M)$ with $\Homeo(M)$ leads to the analogous notion of \emph{topological mapping class group} $\pi_0 \Homeo(M)$.
There are two more groups of symmetries that we will study. These are as follows.
\begin{itemize}
    \item $\hAut(M)\coloneqq \{f \colon M\to M \mid f \text{ is a homotopy equivalence} \}$, the monoid of homotopy self-equivalences of~$M$. {This will be viewed as a space equipped with the $C^0$ topology, and with respect to this we have a continuous inclusion map $\Homeo(M)\hookrightarrow \hAut(M)$.}
    \item $\Isom(M)\coloneqq {\{f \colon M\to M \mid f \text{ is an isometry} \}}$ where, if $M$ has a Riemannian metric $g$, then $f\in \Diff(M)$ is an \emph{isometry} if $f_*(g)=g$. This has the consequence that $f$ is \emph{distance preserving} with respect to $g$. It follows $\Isom(M)\hookrightarrow \Diff(M)$ {and we take $\Isom(M)$ to have the subspace topology with respect to this inclusion}.
\end{itemize}

Each of these of course also have o.p.\ versions, $\Homeo^+(M)$, $\hAut^+(M)$, and $\Isom^+(M)$.
Putting together all the inclusions, we get the following sequence
\begin{equation}\label{eq:inclusion-sequence}
    \Isom(M)\hookrightarrow \Diff(M) \rightarrow \Homeo(M)\hookrightarrow \hAut(M).
\end{equation}
The first and third maps are subspace inclusions; the second is not.
The second map in \eqref{eq:inclusion-sequence} is a homotopy equivalence for every 3-manifold $M$.
The first and third maps are not homotopy equivalences in general, and indeed need not induce isomorphisms on $\pi_0$.

Taking connected components in \eqref{eq:inclusion-sequence}, we obtain a sequence
\begin{equation}\label{eq:pi_0-inclusion-sequence}
    \underbrace{\pi_0 \Isom(M)}_{\Isom(M)/\sim} \to \underbrace{\pi_0 \Diff(M)}_{\Diff(M) / \sim} \xrightarrow{\cong} \underbrace{\pi_0 \Homeo(M)}_{\Homeo(M)/\sim} \to \underbrace{\pi_0 \hAut(M)}_{\hAut(M)/\simeq}.
\end{equation}
In order from left to right, these sets correspond to the set of isometries up to isotopy through isometries, diffeomorphisms up to smooth isotopy, homeomorphisms up to topological isotopy, and self homotopy equivalences up to homotopy.
The first map is not always an isomorphism, e.g.\ it fals to be surjective for $S^1 \times S^2$~\cite{Gluck1962} and for $T^3$. The second map is an isomorphism, since it is induced by a homotopy equivalence. The third map is not an isomorphism in general.  First, it is not injective for certain reducible 3-manifolds, due to Friedman--Witt~\cite{FriedmanWittHomotopyNotIsotopy86}.
The third map is also not surjective in general; we explain why for $S^1 \times S^2$ in \cref{example:hom-equiv-not-homotopic-to-homeo}.

On the other hand, when $M$ is closed and prime, the third map is injective.
This follows from \cref{thm:kernel-rho} below together with the fact that neither the reflections of $S^3$ and $\#^k \RP^3$ nor the Gluck twist on $S^1 \times S^2$ are not homotopic to the identity (\cref{sec-S^1xS^2}).
If $M$ is moreover aspherical, then the third map is an isomorphism~\cite[Theorem~2.1.2]{AFW2015}.

\begin{example}\label{example:hom-equiv-not-homotopic-to-homeo}
Here is an example of a homotopy self-equivalence of $S^1 \times S^2$ that is not homotopic to a homeomorphism. Pinch off a 3-sphere, and then map it to $\{p\} \times S^2$ via the Hopf map $\eta$, to obtain a composition $f \colon S^1 \times S^2 \to (S^1 \times S^2) \vee S^3 \xrightarrow{\Id \vee \eta} S^1 \times S^2$. This acts as the identity on all homotopy groups $\pi_i(S^1 \times S^2)$, and so is a homotopy equivalence. For $i \geq 2$, to see this note that the inclusion $S^2 \to \{q\} \times S^2 \to S^1 \times S^2$ induces an isomorphism $\pi_i(S^2) \cong \pi_i(S^1 \times S^2)$. To see this is not homotopic to a homeomorphism, consider the inverse image of two points in $\{p\} \times S^2$. This consists of those two points again, disjoint union a Hopf link. If $f$ were homotopic to a homeomorphism, again by taking transverse inverse images, the Hopf link would bound disjoint discs in $S^1 \times S^2 \times [0,1]$.
\end{example}

\subsection{The map to outer automorphisms of \texorpdfstring{$\pi_1(M)$}{the fundamental group}}

Recall that $\Out(\pi_1(M))$ is the group of \emph{outer automorphisms} of $\pi_1(M)$, given by automorphisms modulo conjugation. We will  make use of the map
\begin{equation} \label{eqn-MGC to out}
 \rho \colon   \pi_0\Diff(M) \to \Out(\pi_1(M)).
\end{equation}
obtained by choosing a basepoint fixing representative of the isotopy class (which we may do since $M$ is connected), and then taking the induced map on $\pi_1(M)$.  This map is not necessarily surjective nor injective, but we will see situations in which it is either, and others when it is  an isomorphism.
To define the map in \eqref{eqn-MGC to out} more precisely, the long exact sequence of homotopy groups associated to the fibration
\[
\Diff_p(M)\to \Diff(M) \to \Emb(\{x\},M) \simeq M,
\]
where $\{x\}$ is the one point space, and~$p$ the basepoint of~$M$, yields the top row in the diagram
\[
\begin{tikzcd}
\pi_1(M) \ar[r] \ar[d] & \pi_0 \Diff_{p}(M) \ar[r] \ar[d] & \pi_0\Diff(M) \ar[r] \ar[d,dashed,"\rho"] & \{p\} \\
\Inn (\pi_1 (M)) \ar[r] &\Aut(\pi_1(M))  \ar[r] & \Out(\pi_1(M)). &
\end{tikzcd}
\]
From a choice of lift of $[f] \in \pi_0\Diff(M)$ to $\pi_0 \Diff_{p}(M)$, we obtain $f_* \in \Aut(\pi_1(M))$. Two lifts are related by point-pushing along an element $g \in \pi_1(M)$, whose effect on $f_*$ is precisely composition with the inner automorphism $x \mapsto gxg^{-1}$. Factoring out by this indeterminacy we obtain the dashed map $\pi_0\Diff(M) \to \Out(\pi_1(M))$, as desired.

In \cref{sec-reducible}, building on the work in previous sections, we will explain the following theorem. The result is proven in Hatcher--Wahl~\cite[Proposition~2.1]{Hatcher_Wahl_MCG3Manifold}, but relies on the work of many authors, in particular Waldhausen~\cite{Waldhausen1968}, Gabai~\cite{Gabai1997}, Boileau--Otal~\cite{BoileauOtalHeegardSplittings91}, and McCullough~\cite{McCulloughTopAlgAuts3mfds90,McCullough2002}.

\begin{theorem}[\cite{Hatcher_Wahl_MCG3Manifold}]\label{thm:kernel-rho}
Let $M$ be a closed, orientable 3-manifold that is not $S^3$ nor a connected sum $\#^k\RP^3$. Then $\ker \rho$ is generated by generalised Dehn twists along embedded spheres in $M$. If $M$ is irreducible, then $\rho$ is injective.
\end{theorem}

\begin{remark}\leavevmode
    \begin{enumerate}
        \item We have to exclude $M = S^3$ and $M = \#^k \RP^3$ because they admit (orientation-reversing) reflections that act as the identity on $\pi_1(M)$.
        \item The generalised Dehn twists in \cref{thm:kernel-rho} are called \emph{sphere twists} in this context, and are defined in \cref{def:sphere twist}. For reducible 3-manifolds, it is not straightforward to determine which sphere twists along separating spheres are trivial in the mapping class group.
One strategy is to consider circle actions on the prime summands.
\item It can also be challenging to determine the image of $\rho$, in general.  It is surjective for closed, Haken 3-manifolds (\cref{sec-Haken}), but is not always so.
    \end{enumerate}
\end{remark}

\begin{example}
    For $M = N \# (S^1 \times S^2)$, where $N$ is aspherical, the sphere twist about the nonseparating $S^2$ in $S^1 \times S^2$ is the Gluck twist, and is nontrivial, as we will show in \cref{sec-S^1xS^2}. The sphere twist about the separating 2-sphere is trivial, because $S^1 \times S^2 \sm \mathring{D}^3$ admits a circle action whose restriction to the boundary 2-sphere is rotation by $2\pi$.
\end{example}

\begin{example}\label{example:lens-space-switching}
    Let $L(p,q_1)$ and $L(p,q_2)$ be lens spaces that are not diffeomorphic, and let $M := L(p,q_1) \# L(p,q_2)$. The fundamental group is $\Z/p *\Z/p$. The automorphism of $\pi_1(M)$ that switches the factors is not represented by any diffeomorphism, because by uniqueness of prime decompositions, that would entail a diffeomorphism between $L(p,q_1)$ and $L(p,q_2)$.
\end{example}

Despite the general nature of \cref{thm:kernel-rho}, 3-manifolds come with interesting decompositions and often with extra structure. This leads to alternative ways to understand their mapping class groups, which could in some cases be more tractable than $\Out(\pi_1(M))$, or can reduce the computation to smaller pieces, for which either $\Out(\pi_1(M))$ is more computable, or that are geometric, and we need to understand the isometries instead.

\subsection{Geometrisation}
Suppose that $M$ is oriented and irreducible, i.e.~every embedded 2-sphere in $M$ bounds a 3-ball (or equivalently, by the sphere theorem, $\pi_2(M)=0$).
A Riemannian metric $g$ on $M$ is \emph{locally homogenous} if for all $x, y \in M$, there exist neighbourhoods $x\in U_x$, $y\in U_y$, and an isometry $U_x\to U_y$. If $M$ admits such a $g$, then it follows that $\Isom(\widetilde{M})$ acts transitively on~$\widetilde{M}$ with the metric inherited from $(M,g)$. We say that $M$ ``admits a geometric structure modelled on $\widetilde{M}$''.

Moreover, the possible geometric structures are classified. There are eight of them: six are Seifert fibred and there are also hyperbolic and Sol manifolds. Of these eight, the Sol manifolds are the only ones with nontrivial JSJ decomposition.

\begin{definition}
We say that $M$ is \emph{geometric} if its interior admits one of the eight geometric structures.
\end{definition}

Thurston conjectured that each oriented irreducible manifold can be decomposed along a collection of tori in such a way that each of the resulting pieces is geometric and of finite volume. Thurston proved this for Haken manifolds \cite{Thurston86} and the general case was proved by Perelman~\cite{Perelman:2002-1,Perelman:2003-1,Perelman:2003-2}.

A key results toward the proof of the geometrisation theorem is  the JSJ decomposition theorem, proved independently by Jaco--Shalen~\cite{JacoShalen1976} and Johannson \cite{Johannson1979b}, which gives an essentially unique collection of embedded tori in $M$ that decompose $M$ into Seifert fibred and atoroidal pieces (atoroidal means that the only embedded tori are boundary parallel). Moreover, Hatcher
showed that if $M$ is not a $T^2$ bundle over $S^1$ with Anosov monodromy, then the path component of the space of submanifolds of $M$ containing the isotopy class of the JSJ tori is a contractible space \cite[Thm 1(a)]{Hatcher1999}; see also \cite[Corollary 5.21]{BoydBregmanSteinebrunnerModuliSpacesFinite24}.

Combining this with Perelman's elliptisation and hyperbolisation theorems (see e.g.~\cite[Thm 1.7.3, Thm 1.7.5]{AFW2015}), which elucidate the potential geometries on the atoroidal pieces, we arrive at the geometrisation theorem (see e.g.~\cite[Thm 1.7.6]{AFW2015}).

\begin{theorem}[Geometrisation theorem]\label{thm: geometrisation}
After cutting~$M$ along a canonical system of tori and annuli $($JSJ decomposition$)$, each piece admits a geometric structure, that is either Seifert fibred or hyperbolic.
\end{theorem}

These theorems are discussed in more detail in \cref{sec: JSJ}.

\subsection{The Smale conjecture}

The original Smale conjecture was the following statement, proved by Hatcher in the 1980s.

\begin{theorem}[Smale conjecture, \cite{Hatcher1983}]\label{thm-hatcher smale}
The inclusion of $\Isom(S^3)$ into $\Diff(S^3)$ induces a homotopy equivalence $\Isom(S^3) \simeq \Diff(S^3)$, and thus $\Diff(S^3)  \simeq \operatorname{O}(4)$.
\end{theorem}

One can ask whether the Smale conjecture generalises.

\begin{conjecture}[Generalised Smale conjecture]\label{conj - generalised smale}
For a geometric 3-manifold $M$, the inclusion $\Isom(M)\hookrightarrow \Diff(M)$ is a homotopy equivalence.
\end{conjecture}

There are many cases where the generalised Smale conjecture holds, and  we will remark on several of these in the individual sections of this survey. However it does not hold in general: for example Hatcher showed \cite{Hatcher1981} that $\Diff(S^1 \times S^2)\simeq \rm{O}(3) \times \rm{O}(2) \times \Omega \rm{O}(3)$, and $\Isom(S^1 \times S^2)\simeq \rm{O}(3) \times \rm{O}(2)$. Another example is given by $T^3$, which satisfies $\pi_0\Diff(T^3)\cong \Out(\Z^3)\cong \GL(3, \Z)$ by \cite{Waldhausen1968}. However for any flat metric $\pi_0\Isom(T^3)$ is finite, so the generalised Smale conjecture does not hold for $T^3$.

We can instead consider the following diagram
\[
\xymatrix @R-0.25cm{
    \Diff_0(M) \ar[r] &  \Diff(M) \ar[r] & \pi_0\Diff(M)\\
    \Isom_0(M) \ar[u]\ar[r] &  \Isom(M)\ar[u] \ar[r] & \pi_0\Isom(M)\ar[u]
}
\]
where $\Isom_0(M)$ is the connected component of the isometry group $\Isom(M)$ containing the identity. Restricting our attention to these connected components, a version of the Smale conjecture does hold in general.

\begin{theorem}[Weak generalised Smale conjecture] \label{thm-weak Smale}
    If $M$ is geometric, then $\Isom_0(M)\hookrightarrow \Diff_0(M)$ is a homotopy equivalence. 
\end{theorem}

Note that $M$ is permitted nonempty boundary in \cref{thm-weak Smale}.
The theorem is due to many people, in addition to Hatcher~\cite{Hatcher1983}, in particular it uses the work of Hatcher~\cite{Hatcher1976}, Ivanov~\cite{Ivanov1976}, Hong--Kalliongis--McCullough--Rubinstein~\cite{HKMR2012}, Gabai~\cite{Gabai2001}, McCullough-Soma~\cite{McCulloghSoma2013}, and Bamler--Kleiner~\cite{BamlerKleiner2019,BamlerKleiner2023a,BamlerKleiner2023b}.

See \cite[Section 4.1]{BoydBregmanSteinebrunnerModuliSpacesFinite24} for a more detailed overview of the literature and a computation of the homotopy types appearing in the weak generalised Smale conjecture.

The upshot of this theorem is that for geometric 3-manifolds we can understand $\Diff(M)$ in terms of the mapping class group~$\pi_0\Diff(M)$ and the connected component $\Diff_0(M) \simeq \Isom_0(M)$. Sometimes the map on the right of the above diagram, $\pi_0\Isom(M)\to \pi_0\Diff(M)$, can also be shown to be an isomorphism.

\subsection{Overview of this survey}
The following subsections introduce each section of our survey and the results within.

\subsubsection{The 3-manifold $S^3$}
By work of Cerf, $\pi_0\Diff(S^3)\cong \Z/2$ \cite{Cerf1968} (\cref{thm:MCGS3}). This is generated by a reflection, so the orientation-preserving subgroup $\pi_0\Diff^+(S^3)$  is trivial. We also know the Smale conjecture $\Diff(S^3)\simeq \Isom(S^3) \simeq \rm{O}(3)$ by Hatcher \cite{Hatcher1983}. In \cref{section-S^3} we give an overview of Cerf's proof, show that it implies his celebrated $\Gamma_4=0$ theorem (\cref{theorem:Cerf}), and comment on Hatcher's proof~\cite{Hatcher1983}  of the Smale conjecture (\cref{thm-hatcher smale}).

\subsubsection{The 3-manifold $S^1 \times S^2$}
By work of Gluck \cite{Gluck1961, Gluck1962}, $\pi_0\Diff(S^1 \times S^2)\cong (\Z/2)^3$. In \cref{sec-S^1xS^2} we state this theorem (\cref{thm:S^1xS^2_MCG}) and give an overview of the proof.
The orientation-preserving mapping class group $\pi_0\Diff^+(S^1 \times S^2)$ is the index two subgroup $(\Z/2)^2$ generated by the Gluck twist and simultaneous reflection in both factors.
From later work of Hatcher \cite{Hatcher1981}, we also know the homotopy type of the diffeomorphism group: $\Diff(S^1 \times S^2)\simeq \rm{O}(2)\times \rm{O}(3)\times \Omega \rm{O}(3)$.

\subsubsection{Lens spaces}
The mapping class groups $\pi_0\Diff(L(p,q))$ of lens spaces were computed by Bonahon \cite{Bonahon1983}. In \cref{sec-lens spaces} we present his results, specifically in \cref{thm-MCG lens spaces}. This computation relies on work of Schubert \cite{Schubert1956}, who showed that up to isotopy there is a unique Heegaard torus  splitting each lens space into two solid tori (\cref{thm: UniqueTorus}).
We describe three generators $\tau$, $\sigma_+$, and $\sigma_-$ (which might be trivial, depending on $p$ and $q$ -- see \cref{thm-MCG lens spaces}). The orientation-preserving mapping class group $\pi_0\Diff^+(L(p,q))$ is in each case the subgroup generated by $\tau$ and $\sigma_+$.

The generalised Smale \cref{conj - generalised smale} is also true for lens spaces by work of Hong--McCullough--Rubinstein~\cite{HongMcCulloughRubinstein04}, which was ultimately published as part of \cite{HKMR2012}.  A new proof was more recently obtained by Ketover--Liokumovich~\cite[Thm~2.8]{KetoverLiokumovich2023}.
The group of path components of the isometry groups of lens spaces are therefore in bijection with the mapping class groups, and these are shown in Tables 1 and 2 of \cite{HKMR2012}.

\subsubsection{Elliptic 3-manifolds}
\begin{definition}\label{defn-elliptic}
  A closed 3-manifold  $M$ is \emph{elliptic} if it admits spherical geometry (the universal cover $\widetilde{M}$ is isometric to $S^3$ with the standard spherical metric). This is true if and only if $\pi_1(M)$ is finite \cite[Thm 1.7.3]{AFW2015}.
\end{definition}

Elliptic 3-manifolds represent one of the six Seifert fibred geometries. Note that lens spaces are elliptic. Our main reference for elliptic 3-manifolds is the book of Hong--Kalliongis--McCullough--Rubinstein \cite{HKMR2012}, where they prove the generalised Smale conjecture \ref{conj - generalised smale} in detail for the elliptic case. Tables of $\pi_0\Diff(M)\cong \pi_0\Isom(M)$ appear in their introduction, and we will reproduce them here.  Note that, apart from lens spaces, elliptic 3-manifolds do not admit orientation-reversing mapping classes, and so $\pi_0\Diff(M) \cong \pi_0\Diff^+(M)$ in this case.

In \cref{sec-elliptic} we cover the classification of elliptic manifolds and the computation of their isometry groups. Since the generalised Smale conjecture is known, this gives the mapping class groups (see \cref{thm: MCG elliptic manifolds}).

\subsubsection{Haken 3-manifolds}

For \cref{sec-Haken}, the 3-manifolds $M$ under consideration need not be closed, but we continue to restrict to compact, orientable 3-manifolds.

\begin{definition}\label{defn-Haken}
  An irreducible 3-manifold  $M$ is \emph{Haken} if $M$ contains an incompressible surface $\Sigma \subseteq M$, i.e.~$\pi_1(\Sigma) \hookrightarrow \pi_1(M)$ is injective.
\end{definition}

Every 3-manifold with nonempty boundary is Haken.
Note that the definition of Haken does not interact nicely with the eight geometries -- in particular a Haken manifold may have nontrivial JSJ decomposition, and need not be geometric. In older literature, Haken 3-manifolds are called \emph{sufficiently large}.

For $M$ closed and Haken, Waldhausen showed that $\pi_0\Diff(M) \to \pi_0\operatorname{hAut}(M)$ is an isomorphism \cite{Waldhausen1968} (see also \cite{Scott1972}). Since $M$ is aspherical (\cref{lemma:kpi1}), it turns out that the latter group coincides with $\operatorname{Out}(\pi_1(M))$, so that \[\rho \colon \pi_0\Diff(M) \xrightarrow{\cong} \Out(\pi_1(M))\] is an isomorphism.
It may not be straightforward to decide which outer automorphisms correspond to the orientation-preserving subgroup $\pi_0\Diff^+(M)$.

Moreover, Hatcher \cite{Hatcher1976} and Ivanov \cite{Ivanov1976}, showed that, for $M$ Haken, $\Diff(M) \to \operatorname{hAut}(M)$ is a homotopy equivalence. Later, Johannson \cite{Johannson1979} showed that `simple' Haken 3-manifolds (those with trivial JSJ decomposition) have finite mapping class groups.

In \cref{sec-Haken} we state these results, and their generalisations when $\partial M \neq \emptyset$. We give an overview of the proof of surjectivity in Waldhausen's \cref{thm:Waldhausen}, which necessitates introducing Haken's notion of \emph{hierarchies}.

\subsubsection{Hyperbolic 3-manifolds}

\cref{sec-hyperbolic} focuses on closed, orientable 3-manifolds. This is not a meaningful restriction, since if $\partial M \neq \emptyset$ then $M$ is Haken, and the results of \cref{sec-Haken} apply.

\begin{definition}\label{defn-hyperbolic}
  A closed, orientable $3$-manifold $M$ is \emph{hyperbolic} if it admits hyperbolic geometry ($\widetilde{M}$ is isometric to $\mathbb{H}^3$ with the standard hyperbolic metric).
\end{definition}

For finite volume hyperbolic manifolds, Mostow rigidity \cite{Mostow1968} (\cref{thm-Mostow rigidity}) tells us that
\[ \pi_0\Isom(M) \to \pi_0\hAut(M)\cong \Out(\pi_1(M))
\]
where the first map is a surjective map, which factors through~$\pi_0\Diff(M)$.

For Haken hyperbolic 3-manifolds, Waldhausen's results \cite{Waldhausen1968} compute the mapping class groups. For general hyperbolic 3-manifolds, Gabai \cite{Gabai1997, Gabai2001} showed that $\Diff_0(M)\simeq \{\ast\}$, and Gabai--Meyerhoff--Thurston \cite{GabaiMeyerhoffThurston2003} showed that $\pi_0\Diff(M) \hookrightarrow \pi_0\hAut(M)$ is injective (\cref{thm:GMT}). It follows that $\pi_0\Diff(M)\cong \Out(\pi_1(M))$. Gabai further showed \cite{Gabai2001} that the strong generalised Smale conjecture holds for hyperbolic manifolds (\cref{thm-strong Smale for hyperbolic}).

In \cref{sec-hyperbolic} we discuss properties of hyperbolic 3-manifolds, and give an overview of the above results and the steps in their proofs.

\subsubsection{Seifert fibred 3-manifolds}

In \cref{sec-Seifert fibred} we discuss mapping class groups of compact, orientable Seifert fibred 3-manifolds, which necessarily have toroidal boundary.

\begin{definition}\label{defn-Seifert fibered}
    A compact, orientable \(3\)-manifold \(M\) with toroidal boundary is called \textit{Seifert fibred} if it is a circle bundle over a \(2\)-dimensional orbifold, such that each fibre has a tubular neighbourhood that is equivalent to a standard fibred torus for interior fibres, or standard fibred half-torus $S^1 \times D^2_+$ for boundary fibres.
\end{definition}

Recall that for $M$ a geometric manifold, six of the eight geometries are Seifert fibred, with the exception of hyperbolic and Sol manifolds.  We discuss the definition of Seifert fibred 3-manifolds in \cref{subsec:defn-of-SFS}, their construction and their description in terms of Seifert symbolds in \cref{subsec:Seifert-symbols}, and the relationship between Seifert fibred and Haken 3-manifolds in \cref{subsec:SFS-vs-Haken}.

In \cref{subsec:fibre-pres-diffeos} we discuss the space of fibre-preserving diffeomorphisms of a Seifert fibred 3-manifold. If $M$ is not a lens space, and not one of a list of exceptional cases, every diffeomorphism is isotopic to a fibre-preserving one.  In the closed, Haken case, any two fibre-preserving diffeomorphisms that are isotopic, are in fact isotopic through fibre-preserving diffeomorphisms.

For the closed, Haken case, we know the mapping class group is isomorphic to $\Out(\pi_1(M))$, and there is a refined version discussed in \cref{sec-Haken} for 3-manifolds with boundary. The reduction to fibre-preserving diffeomorphisms yields some tools for understanding $\Out(\pi_1(M))$ in terms of automorphisms of the orbifold fundamental group of the base, which we describe in \cref{subsec:Haken-SFS}.

In the non-Haken case one only has to consider Seifert fibrations with base orbifold $S^2$ and three exceptional fibres. Each such 3-manifold $M$ that is not Haken is either elliptic, and so $\pi_0\Diff(M)$ is $\pi_0 \Isom(M)$ by \cref{sec-elliptic}, or has $\mathbb{H} \times \R$ or $\SL(2,\R)$ geometry, in both cases with $\pi_0 \Isom(M) \cong \pi_0\Diff(M) \cong \Out(\pi_1(M))$. So the mapping class group is in principle computable.  In \cref{subsec:SFS-non-Haken-case} we present the outcome of explicit calculations for an infinite family of such 3-manifolds.

\subsubsection{JSJ decompositions and mapping class groups}

In \cref{sec: JSJ}, the 3-manifold $M$ is assumed to be compact, orientable, and irreducible, with boundary a (possibly empty) disjoint union of tori, and we restrict to the orientation-preserving mapping class group $\pi_0\Diff^+(M)$.
As discussed above, the {JSJ} decomposition of Jaco--Shalen~\cite{JacoShalen1976} and Johannson \cite{Johannson1979b} decomposes a 3-manifold into Seifert fibred and atoroidal pieces (see e.g.~\cite{Jaco1980}).
In \cref{sec: JSJ} we will first introduce the JSJ theorem and geometrisation theorem in further detail.

Then we derive some relationships between the orientation-preserving mapping class groups of the JSJ components of a manifold $M$ and $\pi_0\Diff^+(M)$.   To our knowledge these relationships have not been recorded elsewhere in the literature.

Assuming a nontrivial JSJ decomposition, hyperbolic JSJ components are Haken and so have their mapping class groups determined by the induced action on the fundamental group (\cref{sec-Haken,sec-hyperbolic}). The Seifert fibred pieces can also be Haken; either way their mapping class groups are studied in \cref{sec-elliptic,sec-Seifert fibred}.

\subsubsection{Reducible 3-manifolds}
Recall a compact 3-manifold~$M$ is \emph{prime} if whenever we decompose~${M}$ as a nontrivial connected sum~${M}=M_1\# M_2$, then at least one of~$M_1$ or~$M_2$ is diffeomorphic to~$S^3$. Recall that $S^1\times S^2$ is the unique orientable 3-manifold which is prime but not irreducible. 
By work of Kneser \cite{Kneser1929}, every compact, oriented 3-manifold admits a connected sum decomposition $M=P_1\#\cdots\# P_n$ where the \emph{prime factors} $P_i$ are prime and not diffeomorphic to~$S^3$. Milnor later showed that the oriented prime factors appearing in this decomposition are uniquely determined up to reordering~\cite{Milnor62}.

Up to this point, we will have considered geometric 3-manifolds, irreducible 3-manifolds, and $S^1\times S^2$. In \cref{sec-reducible} we turn to considering reducible 3-manifolds ($\neq S^1 \times S^2$), i.e.~3-manifolds with nontrivial prime decomposition.

We explain in \cref{subsec-reducible-in-terms-summands} what is known about the mapping class group of a reducible 3-manifold as a function of the mapping class groups of its prime summands.

Then we state \cref{prop:generators of kernel of action of MCG on Out(pi1)}, which extends \cref{thm:kernel-rho} to the case of nonempty boundary. This applies to reducible 3-manifolds $M$, and allows one to in principle compute the mapping class group, for $M$ closed, if one can determine the fate of the sphere twists and the image of $\rho$.

In the particular case of $M\cong (S^1\times S^2)^{\#n}$ we survey work of \cite{Laudenbach1973, Laudenbach1974, BrendleBroaddusPutman2023} which focuses on the Laudenbach sequence arising from study of the map to $\Out(\pi_1(M))$ in \eqref{eqn-MGC to out}. In particular Brendle--Broaddus--Putman show that this map admits a splitting (\cref{cor-BBP splitting}). In fact their proof can be generalised to show that, for $M$ closed and oriented, $\pi_0\Diff(M)$ decomposes as a semi-direct product, albeit with one of the factors somewhat mysterious, and we present the result in this level of generality (\cref{thm: generalised version of BBP23 main thm}).

\subsubsection{Finiteness properties}

In \cref{section:finiteness-properties} we consider results on finiteness of mapping class groups of 3-manifolds. For Haken manifolds, McCullough \cite{McCullough1991} showed that the mapping class groups are finitely presented, and investigated other finiteness properties. We also cover finiteness results for $\BDiff(M)$.

Then we outline a proof of a result of Hatcher--McCullough \cite{HatcherMcCullough1990} which states that the mapping class group $\pi_0\Diff(M)$ of a compact, orientable 3-manifold, possibly with nonempty boundary, is finitely presented if the mapping class groups of its irreducible prime components are (\cref{thm-finite presentation}).

\subsubsection{Further topics}

We finish by briefly surveying some other interesting results on mapping class groups of 3-manifolds. There are many topics which could be included here and we present three.

In \cref{subsec-HS of MCGs} we give an overview of Hatcher and Wahl's results on homological stability for mapping class groups of 3-manifolds \cite{Hatcher_Wahl_MCG3Manifold}.

Finally in \cref{subsec-Nielsen} we collect known results on Nielsen realisation for mapping class groups of 3-manifolds, which is a topic that was popular in the 1970s and 80s in the irreducible setting and has recently seen progress for reducible 3-manifolds.

 \subsection{Guide to computing mapping class groups}

We give a guide to the methods one can try for computing the mapping class group of a closed, orientable 3-manifold, noting that there is no guarantee of success in all situations.

\begin{itemize}
    \item If $M$ is $S^3$ or $S^1 \times S^2$, see \cref{section-S^3,sec-S^1xS^2} respectively: $\pi_0 \Diff(S^3) \cong \Z/2$ and $\pi_0 \Diff(S^1 \times S^2) \cong (\Z/2)^3$.
    \item Compute the prime decomposition of $M$.
    \item Appealing to \cref{thm:kernel-rho} may yield useful information. Compute the image of $\rho \colon \pi_0 \Diff(M) \to \Out(\pi_1(M))$, and determine which sphere twists are nontrivial.
    \item If that is insufficient, we can consider additional structure on $M$, and piece the mapping class group together from that of its JSJ pieces.  For each prime summand, try the following.
    \begin{enumerate}[(a)]
        \item Decide whether it is Haken. If so, see \cref{sec-Haken}.
        \item Decide whether it is Seifert fibred. If so, \cref{sec-Seifert fibred} may provide useful information.
        \item Compute the JSJ decomposition.
    \end{enumerate}
\item For each JSJ piece, consider the following.
\begin{enumerate}[(a)]
    \item Is it hyperbolic? If so, see \cref{sec-Haken}.
    \item Is it elliptic? If so, see \cref{sec-lens spaces,sec-elliptic}.
    \item If neither hold, see \cref{sec-Seifert fibred}.
\end{enumerate}
Caveat: there may be an intermediate Seifert fibred 3-manifold that is a union of JSJ pieces, that is more convenient to work with.
\item Assuming we now know the mapping class groups of the pieces, use \cref{sec: JSJ} to assemble the mapping class group of each irreducible summand from that of its pieces.
\item Use \cref{sec-reducible} to assemble the mapping class group of $M$ from that of its summands. By \cref{subsec-reducible-in-terms-summands}, the mapping class group is generated by slide diffeomorphisms, flips, Gluck twists, transpositions of prime factors, and the mapping class groups of the factors rel.\ a 3-disc.
\end{itemize}

\subsection*{Conventions}
Throughout the article, $M$ will denote a compact, oriented 3-manifold.
For $K \subseteq M$ a submanifold, or for $K = \partial M$, 
we write $\Diff_K(M)$ for the diffeomorphisms of $M$ that fix $K$ setwise, and $\Diff(M,K)$ for the diffeomorphisms that fix $M$ setwise.

We exclude $S^3$ and $S^1 \times S^2$ from the list of lens spaces in this article.

\subsection*{Acknowledgements}
This survey is the result of a seminar at the University of Glasgow organised by Rachael Boyd and Mark Powell. The webpage for the seminar is \url{https://www.maths.gla.ac.uk/~mpowell/3-manifolds-seminar.html}. The speakers in the seminar are now the authors of this survey. The authors would like to collectively thank the other members of the geometry and topology group at Glasgow who participated in the seminar. In particular we would like to thank Riccardo Giannini who contributed to an earlier version of these notes.

Boyd would like to thank Corey Bregman and Bena Tshishiku for helpful conversations, especially with regards to navigating the literature. Corrigan would like to thank Tara Brendle and Simeon Hellsten for helpful conversations and clarifications.

Finally we would like to thank the organisers of the Georgia Topology Conference 2025 for encouraging us to publish this survey in their proceedings.

Bader was supported by the Additional Funding Programme for Mathematical Sciences, delivered by EPSRC (EP/V521917/1) and the Heilbronn Institute for Mathematical Research.
Boyd was supported by EPSRC Fellowship No.~EP/V043323/2.
Nagy was supported by EPSRC New Investigator grant EP/T028335/2.
Nonino was supported by EPSRC Studentship No. ~EP/W524359/1.
Powell was partially supported by EPSRC New Investigator grant EP/T028335/2.

\section{The mapping class group of \texorpdfstring{$S^3$}{the 3-sphere}}\label{section-S^3}

In this section we provide an overview of the study of $\Diff({S^3})$.
We will show that the orientation preserving mapping class group of $S^3$ is trivial, leading to the following theorem.

\begin{theorem}[\cite{Cerf1968}]\label{thm:MCGS3}
The mapping class group of $S^3$ is $\pi_0\Diff(S^3)=\mathbb{Z}/2$.
\end{theorem}

\begin{remark}\label{rmk:triviality_orientation_preserving}~
  Note that two diffeomorphisms of $S^3$ are homotopic only if they are both orientation preserving or orientation reversing. Hence \cref{thm:MCGS3} implies that every orientation-preserving diffeomorphism of $S^3$ is isotopic to the identity, i.e.\ $\pi_0 \Diff^+(S^3) = \{0\}$.
\end{remark}

\begin{remark}
A proof that $\pi_0 \Homeo(S^3) \cong \Z/2$ was also given by Fisher~\cite{Fisher}. This also proves \cref{thm:MCGS3} when combined with the isomorphism $\pi_0 \Diff (S^3) \cong \pi_0 \Homeo (S^3)$.
\end{remark}

Next, we discuss the implications of this result for the potential construction of  homotopy 4-spheres.
Then we outline the proof of \cref{thm:MCGS3}.
Lastly, we will discuss \cref{thm-hatcher smale} \cite{Hatcher1983}. The statement $\Diff({S^3}) \simeq \rm{O}(4)$ gives a complete answer to the original problem of understanding the homotopy type of $\Diff({S^3})$.

\subsection{{Triviality of {${\pi_0\Diff^+(S^3)}$}}}
 We discuss how to prove \cref{thm:MCGS3} through a series of equivalent reformulations of the problem.
We start with the following key lemma.

\begin{lemma}\label{lem:splitting_hokmotopy_groups}
We have the following homotopy equivalence
$$\Diff^+({S^n}) \simeq \SO(n+1) \times \Diff^+_{ \partial S^{n-1}}(D^n),$$
which implies that
    \begin{equation}\label{eq:second}
    \pi_i\Diff^+({S^n}) \cong \pi_i\Diff_{S^{n-1}}^+{(D^{n})} \times \pi_i\SO(n+1)
\end{equation}
for $i \ge 0$.
\end{lemma}

\begin{proof}
    We give an overview of the proof. A good reference for more details is \cite{Antonelly-Burghelea-Kahn}.
    First, for $p\in S^n$, the evaluation map $${\rm ev}_p\colon \Diff^+ (S^n) \to S^n, \phi \mapsto\phi(p)$$
defines a transitive action $\Diff^+(S^n) \times S^{n} \to S^n$. The stabiliser of a point $p \in S^n$ is $\Diff^+(S^n,p)$, the diffeomorphisms fixing $p$. This action defines a locally trivial fibration \cite{PalaisLocalTriviality60}. We can say more: this fibration admits a global section $S^n \to \Diff^+(S^n)$ with image in the special orthogonal group $\SO(n+1) \subseteq \Diff^+(S^n)$. We can restrict the action of $\Diff^+(S^n)$ on $S^n$ to $\SO(n+1)$, to obtain a transitive action by rotations on $S^n$ that also determines a locally trivial fibration $\SO(n) \to \SO(n+1) \to S^n$.   In fact, $S^n \cong \SO(n+1)/\SO(n)$.
Consider the derivative map at the point $p$,
\[ {\rm D}_p\colon \Diff^+(S^n,p) \to \GL(T_pS^n)\cong \GL(n,\mathbb R) \xrightarrow{\simeq} \SO(n).\]
 The kernel of this map is homotopy equivalent to $\Diff^+_{\partial D^n}(D^n)$ or, equivalently, to $\Diff_{S^{n-1}}^+(D^n)$. Thus we have a fibration
$$ \Diff^+_{S^{n-1}}(D^n) \to \Diff^+(S^n,p)\xrightarrow{{\rm D}_p}\SO(n)$$ that admits a global section.

Define a map $\Diff^+(S^n) \to \SO(n+1)$, as follows. Let $F \in \Diff^+(S^n)$. Let $e_1,\dots,e_{n+1}$ be the standard basis of $\R^{n+1}$. Note that we can consider $e_1$ as a point $p$ of $S^n$, and $e_2,\dots,e_{n+1}$ as a basis for the tangent space $T_{e_1}S^n$, using the identification $\langle e_1 \rangle^\bot \cong T_{e_1} S^n$. Similarly we can consider $dF(e_i) \in T_{F(e_1)}S^n$ as an element of $\R^{n+1}$ orthogonal to $F(e_1)$.  We define
\begin{align*}
   \Diff^+(S^n) &\to \GL(n+1,\R) \\
   F &\mapsto [F(e_1),dF(e_2),\ldots dF(e_{n+1})].
\end{align*}
The required map  $\Diff^+(S^n) \to \SO(n+1)$ arises from post-composing this with the homotopy equivalence $\GL(n+1,\R) \to \SO(n+1)$ from Gram-Schmidt.

We can fit everything together into the following commutative diagram.
\[\begin{tikzcd}[cramped]
	{\Diff_{S^{n-1}}^+(D^n)} \\
	{\Diff^+(S^n,p)} & {\Diff^+(S^n)} & {S^n} \ar[d,equals] \\
	{\SO(n)} & {\SO(n+1)} \ar[r] & S^n
	\arrow[hook, from=1-1, to=2-1]
	\arrow[hook, from=2-1, to=2-2]
	\arrow["{{\rm D}_p}", two heads, from=2-1, to=3-1]
	\arrow["{{\rm ev}_p}", two heads, from=2-2, to=2-3]
	\arrow[two heads, from=2-2, to=3-2]
	\arrow[hook, from=3-1, to=3-2]
\end{tikzcd}\]
One has to show that the central square is a pullback. It follows that we get a  fibration $$\Diff_{S^{n-1}}^+(D^n) \to \Diff^+(S^n) \to \SO(n+1),$$ where the inclusion of the special orthogonal group into $\Diff^+(S^n)$ gives a global section. Thus we have the homotopy equivalence $$\Diff^+({S^n}) \simeq \SO(n+1) \times \Diff^+_{ S^{n-1}}(D^n)$$ and the homotopy groups split as in \eqref{eq:second}.
\end{proof}

\begin{proof}[Proof of \cref{thm:MCGS3}]

Since we want to show $\pi_0\Diff^+({S^3})=0$ and we know that $\pi_0(\SO(4))=0$, by  \cref{lem:splitting_hokmotopy_groups} it suffices to prove  that $\pi_0\Diff_{S^2}^+({D^3})=0$.
The restriction map gives a locally trivial fibration
\begin{equation*}
    \Diff^+(D^3) \to \Diff^+(S^2)
\end{equation*}
whose fibre is $\Diff_{S^2}^+({D^3})$.
The long exact sequence in homotopy groups of the fibration yields an exact sequence
\begin{equation}\label{eq:third}
    \pi_1\Diff^+(D^3) \to \pi_1\Diff^+(S^2) \to \pi_0\Diff_{S^2}^+({D^3}) \to \pi_0\Diff^+(D^3) \to \pi_0\Diff^+(S^2) \to 0.
\end{equation}
By Smale \cite{Smale_diffeos_S^2} we know $\Diff^+(S^2)\simeq \SO(3)$ and the composition
\begin{equation*}\label{eq:fourth}
    \pi_i(\SO(3)) \to \pi_i\Diff^+(D^3) \to \pi_i\Diff^+(S^2)
\end{equation*}
is an isomorphism for all $i \ge 0$.
This implies that
\begin{enumerate}
    \item $\pi_0\Diff^+(S^2)=0$ and
    \item $ \pi_i\Diff^+(D^3) \to \pi_i\Diff^+(S^2)$ is surjective for all $i \ge 1$.
\end{enumerate}
Combining \eqref{eq:third} with the above implications we obtain an isomorphism
\begin{equation*}\label{eq:fifth}
  \pi_0\Diff_{S^2}^+({D^3}) \xrightarrow{\cong} \pi_0\Diff^+(D^3).
\end{equation*}
Triviality of $\pi_0\Diff_{S^2}^+({D^3})$ is hence equivalent to triviality of $\pi_0\Diff^+({D^3})$.

Consider the embedding spaces $\Emb(D^3, \R^3)$ and $\Emb(S^2,\R^3)$, and the identity component $\Diff^+_0(D^3)$.
The group $\Diff^+(D^3)$ acts on $\Emb(D^3,\R^3)$ on the right using pre-composition, i.e.~$j \phi= j \circ \phi$ for $\phi \in \Diff^+(D^3)$ and $j \in \Emb(D^3,\R^3)$. There is also a locally trivial fibration
\begin{equation}\label{eq:seventh}
    \Emb(D^3,\R^3)/\Diff_0^+(D^3)\to \Emb(D^3,\R^3)/\Diff^+(D^3)
\end{equation}
whose fibre is $\Diff^+(D^3)/\Diff_0^+(D^3)$. Note that $\Diff^+(D^3)/\Diff_0^+(D^3)$ is isomorphic to the orientation preserving mapping class group of $D^3$. Since $\Diff^+(D^3)$ is locally path-connected the quotient space is discrete. Since the fibration in \eqref{eq:seventh} is locally trivial with discrete fibre, it is a \emph{covering} of $\Emb(D^3,\R^3)/\Diff^+(D^3)$.
The strategy now is to show that this cover is trivial.

\begin{claim*}
    If the covering of \eqref{eq:seventh}  is trivial, then $\Diff^+(D^3)/\Diff_0^+(D^3)$ is contractible, and hence $\pi_0\Diff^+(D^3) = \{0\}$.
\end{claim*}
\begin{proof}[Proof of Claim]
    $\Emb(D^3,\R^3)$ is connected and hence $\Emb(D^3,\R^3)/\Diff_0^+(D^3)$ is connected as well, being its quotient. If the map in \eqref{eq:seventh} is a trivial cover then $$\Emb(D^3,\R^3)/\Diff_0^+(D^3) \cong \left((\Diff^+(D^3)/\Diff_0^+(D^3)\right) \times  \left(\Emb(D^3,\R^3)/\Diff^+(D^3)\right),$$ so we must have $\Diff^+(D^3)/\Diff_0^+(D^3)$ is connected. Since $\Diff^+(D^3)/\Diff_0^+(D^3)$ is discrete we conclude it is trivial.
\end{proof}

To summarise what we obtained so far: the triviality of $$\Emb(D^3,\R^3)/\Diff_0^+(D^3)\to \Emb(D^3,\R^3)/\Diff^+(D^3)$$ implies $\pi_0\Diff^+({D^3})=0$ which in turn implies $\pi_0\Diff^+(S^3)=0$.

The action of $\Diff^+(D^3)$ on $\Emb(D^3,\R^3)$ gives rise to a locally trivial fibration~\cite{PalaisLocalTriviality60, BoydBregmanSteinebrunnerModuliSpacesFinite24}
\begin{equation*}\label{eq:sixth}
    \Diff^+(D^3) \to \Emb(D^3,\R^3) \to \Emb(D^3,\R^3)/ \Diff^+(D^3)
\end{equation*}
where the final term is the space of 3-discs in $\R^3$.  We have a similar fibration with $S^2$ in place of $D^3$.
We can now consider the following commutative diagram, where the two horizontal rows are fibrations.

\[\begin{tikzcd}[cramped]
	 {\Diff^+(D^3)} & {\Emb(D^3,\R^3)} & {\Emb(D^3,\R^3)/\Diff^+(D^3)}  \\
 {\Diff^+(S^2)} & {\Emb(S^2,\R^3)} & {\Emb(S^2,\R^3)/\Diff^+(S^2)}
	\arrow[from=1-1, to=1-2]
	\arrow[from=1-1, to=2-1]
	\arrow[from=1-2, to=1-3]
	\arrow[from=1-2, to=2-2]
	\arrow[from=1-3, to=2-3]
	\arrow[from=2-1, to=2-2]
	\arrow[from=2-2, to=2-3]
\end{tikzcd}\]

The first two vertical maps are given by restriction, and the third one is defined to make the diagram commute.
Cerf \cite[Lemme 4]{Cerf62} notes that the right-most vertical map in the diagram is a homeomorphism if and only if the (weak) smooth Schoenflies conjecture for $S^2$ holds. The latter was proven by Alexander \cite{Alexander} and Morse-Baiada \cite{Morse-Baiada}, and reproven by Cerf in \cite{Cerf1968}.

We can therefore replace the base space $\Emb(D^3,\R^3)/\Diff^+(D^3)$ in \eqref{eq:seventh}, to obtain a covering $$\Emb(D^3,\R^3)/\Diff_0^+(D^3)\to \Emb(S^2,\R^3)/\Diff^+(S^2).$$
We need to show this is a trivial covering. To do this, Cerf finds a continuous section $p$ of the covering map. His argument is rather long and complicated and uses a stratification of the base space $\Emb(S^2,\R^3)/\Diff^+(S^2)$ involving Thom's Transversality theorems and the auxiliary space of differentiable functions $S^2 \to \R^3$ \cite{Cerf1968}.
Broadly speaking, Cerf first constructs a section over the space $\left(\Emb(S^2,\R^3)/\Diff^+(S^2)\right)^0$, which consists of those $S^2$-submanifolds for which the height function (given as the third coordinate in $\R^3$) is a nondegenerate $C^2$
function with all critical points at different levels. Then he extends it over the stratum $ \left(\Emb(S^2,\R^3)/\Diff^+(S^2)\right)^1$, which consists of $S^2$-submanifolds where the height function is either nondegenerate with exactly two critical points on the same level as some other critical point, or has exactly one degenerate critical point -- where the function looks locally like $x^3 +y^3$ -- and all other critical points are nondegenerate and at different levels. Lastly, Cerf argues that the section built over $\left(\Emb(S^2,\R^3)/\Diff^+(S^2)\right)^0 \cup \left(\Emb(S^2,\R^3)/\Diff^+(S^2)\right)^1$ extends over the whole space $\Emb(S^2,\R^3)/\Diff^+(S^2)$.
We will not discuss the specifics of the proof here. For the technical step of constructing the section over these two strata, we refer the reader to \cite{Cerf1968}.
\end{proof}

\subsection{$\Gamma_4=0$}
One of the consequences of \cref{thm:MCGS3} is another of Cerf's theorems, on $\Gamma_4$.

\begin{definition}\label{def:Gamma_n}
  Let $\alpha_n \colon \Diff^+({D^n}) \to \Diff^+({S^{n-1}})$ be the natural map given by the restriction to the boundary.
Define $\Gamma_n := \coker{\alpha_n}$.
\end{definition}

\begin{theorem}[Cerf, \cite{Cerf1968}]\label{theorem:Cerf}
    $\Gamma_4=0$.
\end{theorem}

\begin{proof}
   Let $\phi \in \Diff^+(S^3)$, and consider a collar $S^{3} \times I \hookrightarrow D^4$ of the boundary of $D^4$. By \cref{thm:MCGS3} and \cref{rmk:triviality_orientation_preserving}, $\phi$ is isotopic to the identity.  Use the isotopy $H \colon S^{3} \times I \to S^{3}$ to construct a diffeomorphism of the collar $\ol{H} \colon S^{3} \times I \to S^{3} \times I$; $\ol{H}(x,t) =(H(x,t),t)$. This restricts to $\phi$ on $S^{3} \times \{0\} = \partial D^n$, and to $\Id_{S^3}$ on $S^{3} \times \{1\}$, the interior boundary of the collar. Extend by the identity on the rest of $D^4$, to see that $\phi = 0 \in \Gamma_4$, as desired.
\end{proof}

\begin{remark}
    By work of Kervaire-Milnor~\cite{Kervaire_Milnor} on the groups $\theta_n$ of homotopy $n$-spheres up to $h$-cobordism and the (smooth input, topological output) generalised Poincar\'e Conjecture $(n \ge 5)$ \cite{Smale_poincare}, we know that $\Gamma_n \cong \theta_n$, and the latter corresponds to the group of oriented smoothings of $S^n$ (again, under the assumption $n \ge 5$).
The identification $\Gamma_n \cong \{\text{smoothings of } S^n\}/\text{isotopy}$ is given by sending $[F]$ to the smoothing $\mathcal{S}_{[F]}$ of $S^{n}$ obtained by gluing two $n$-discs along $S^{n-1}$ via a representative $F$.

 It is known that $\Gamma_5 = \Gamma_6=0$~\cite{Kervaire_Milnor}, which means that there are no exotic spheres in dimension five and six, whereas $\Gamma_7\cong \mathbb{Z}/28$, leading to exotic 7-spheres. On the other hand, \cref{theorem:Cerf} does not imply that there are no exotic spheres in dimension $4$.  It tells us that there are no exotic spheres that arise from the above construction, i.e.~ every manifold obtained from gluing two $4$-discs along a diffeomorphism of $S^3$ is diffeomorphic to $S^4$, because every such diffeomorphism is extendable. There might still be other ways to obtain an exotic 4-sphere.
\end{remark}

\begin{remark}
    It is worth mentioning that there is also a proof of \cref{theorem:Cerf} by Eliashberg using tools from contact geometry \cite{Eliashberg:20years} (see also \cite{GeigesZehmisch2010}).
\end{remark}

\subsection{Hatcher's Theorem}
As mentioned in the introduction of this section, the full scope of Hatcher's work is to determine the homotopy type of $\Diff(S^3)$. In \cite{Hatcher1983} Hatcher proved the Smale conjecture (\cref{thm-hatcher smale}) which states that $ \Diff{S^3} \simeq \rm{O}(4)$.

\begin{remark}
    We point out that by passing to connected components, this immediately recovers \cref{thm:MCGS3}.
\end{remark}

While Theorem \ref{thm-hatcher smale} is stated in a more concise and direct way, Hatcher's original formulation was rather different. For the sake of completion, we will state the original theorem.

\begin{theorem}[\cite{Hatcher1983}]\label{thm:hatcher_original}
    Let $g_t \colon S^2 \to \R^3$ be a smooth family of $C^{\infty}$ embeddings, $t \in S^k$. Then this extends to $\widehat{g}_t \colon D^3 \to \R^3$ for all $k \ge 0$.
\end{theorem}

We show that the two formulations are indeed equivalent. First we know that
\begin{equation}
 \Diff({S^n}) \simeq \rm{O}(n+1) \times \Diff_{ \partial D^n}(D^n)
\end{equation}
for all $n$. This follows from a similar proof as the orientation preserving case. The statement $\Diff({S^3}) \simeq \rm{O}(4)$ is hence equivalent to $\Diff_{S^2}(D^3) \simeq *$.
It therefore suffices to show that this last statement is equivalent to the one in Hatcher's original work.
Theorem \ref{thm:hatcher_original} says that the natural map $\rho \colon \Emb(D^3,\R^3) \to \Emb(S^2,\R^3)$ is surjective on $\pi_k$ for all $k \ge 0$. Consider the following commutative diagram
\begin{center}
\begin{tikzcd}
{\Emb(D^n,\R^n)} \arrow[rd, "\simeq"] \arrow[rr, "\rho"] &            & {\Emb(S^{n-1},\R^n)} \arrow[ld] \\
                                                        & {\GL(n,\R)} &
\end{tikzcd}
\end{center}
Here the lower left map is given by evaluating the derivative at a point (giving us a homotopy equivalence), and the lower right map is evaluating the derivative at $e_1$ and adjoining the normal vector of the image of $e_1$. From the diagram we see that $\rho$ is injective on $\pi_k$ for all $k$. Moreover, $\rho$ is a fibration whose fibre is $\Diff_{S^{n-1}}(D^n)$. This implies that $\Diff_{S^2}(D^3)$ is contractible if and only if $\rho$ is also surjective on all $\pi_k$. This shows that the two statements are indeed equivalent.
We will not prove Hatcher's result here. However, by explaining this equivalence of statements we hope the proof in \cite{Hatcher1983} becomes slightly more accessible.

Finally, note that Bamler and Kleiner provided an independent proof of Hatcher's theorem, using techniques from Ricci flow~\cite{BamlerKleiner2019}.

\section{The mapping class group of \texorpdfstring{$S^1\times S^2$}{the product of the circle and the sphere}}\label{sec-S^1xS^2}

In this section we will compute the mapping class group of $S^1\times S^2$.  This computation is due to Gluck \cite{Gluck1961,Gluck1962}.

\subsection{Statement of the theorem and overview of the proof}

First, let us describe some self-diffeomorphisms of the manifold in question, $S^1\times S^2$.  Let $a\colon S^2\to S^2$ be the antipodal map and let $s\colon S^1\to S^1$ be the conjugation map, i.e.~if we parametrise $S^1$ as the unit complex numbers, then $s$ is the map sending $z\in \C$ to its complex conjugate $z^*$.  In an abuse of notation, we will then define maps $a,s\colon S^1\times S^2\to S^1\times S^2$ as $\Id_{S^1}\times a$ and $s\times \Id_{S^2}$, respectively.  We define a final map, the so-called \emph{Gluck twist}, as \begin{align*}
    T\colon S^1\times S^2\to& S^1\times S^2, \\
    (\theta,x) \to& (\theta, r_\theta(x))
\end{align*}
where $r_\theta\colon S^2\to S^2$ is the map given by a positive rotation by angle $\theta$ about the vertical axis (from south pole to north pole).

We can now state the theorem, i.e.~the computation of the mapping class group of $S^1\times S^2$.

\begin{theorem}[\cite{Gluck1961,Gluck1962}]\label{thm:S^1xS^2_MCG}
    The mapping class group of $S^1\times S^2$ is
    \[
    \pi_0\Diff(S^1\times S^2)\cong \Z/2\times \Z/2 \times \Z/2,
    \]
    where the three generators of the $\Z/2$-factors are $a,s$ and $T$.
    The orientation-preserving subgroup is
     \[
    \pi_0\Diff^+(S^1\times S^2)\cong \Z/2\times \Z/2,
    \]
    generated by $a\circ s$ and $T$.
\end{theorem}

We now lay out the strategy of the proof of this theorem.  First, note that we have a homomorphism
\[
\varphi\colon \pi_0\Diff(S^1\times S^2) \to \Z/2 \times \Z/2
\]
which sends a representative diffeomorphism to its induced map on $H_1(S^1\times S^2)\cong \Z$ paired with its induced map on $H_2(S^1\times S^2)\cong \Z$ (identifying $\Aut(\Z) \cong \Z/2$).  Given a diffeomorphism, both of these induced maps are the $\pm 1$ maps on these homology groups and, since isotopic diffeomorphism must induce the same maps, this gives us the well defined homomorphism $\varphi$ above.

We will reduce proving \Cref{thm:S^1xS^2_MCG} to proving the following theorem.

\begin{theorem}\label{thm:ker_S^1xS^2}
    We have that $\ker\varphi\cong \Z/2$, generated by $T$.
\end{theorem}

The proof of the above theorem will constitute most of the section.  First, however, we will prove \Cref{thm:S^1xS^2_MCG}, assuming that \Cref{thm:ker_S^1xS^2} holds.

\begin{proof}[Proof of \Cref{thm:S^1xS^2_MCG}]
    We have the following short exact sequence.
    \[
    0\to \ker\varphi \to \pi_0\Diff(S^1\times S^2) \xrightarrow{\varphi} \Z/2\times \Z/2 \to 0
    \]
    where the last map is surjective since the maps $a$, $s$ and $a\circ s$ induce all of the non-trivial elements in $\Z/2\times \Z/2$.  In fact, this map is split, since $a$ and $s$ give us a splitting $\Z/2\times \Z/2\to \pi_0\Diff(S^1\times S^2)$. For this one must observe that $a$ and $s$ are indeed both order two, and that they commute.

    Since $\ker\varphi\cong \Z/2$ by \Cref{thm:ker_S^1xS^2}, we conclude that $\pi_0\Diff(S^1\times S^2)\cong \Z/2\times\Z/2\times\Z/2$ as required (note that since $\ker\varphi\cong \Z/2$ it supports no interesting actions and hence we obtain a direct product structure).
\end{proof}

\Cref{thm:ker_S^1xS^2} trivially follows from the following two propositions.

\begin{proposition}\label{prop:S^1xS^2prop1}
    Let $f\in \ker\varphi$ be a representative diffeomorphism.  Then $f$ is isotopic to $\Id$ or $T$.
\end{proposition}

\begin{proposition}\label{prop:S^1xS^2prop2}
    The Gluck twist $T$ represents an order two element in the mapping class group of $S^1\times S^2$.  In particular, it is not isotopic to the identity.
\end{proposition}

 We will now prove these propositions.

\subsection{Proof of \Cref{prop:S^1xS^2prop1}}

The proof will be split up into a number of steps.  The aim is to pick a representative diffeomorphism in $\ker\varphi$ and then step by step build an isotopy from it to either the identity map or the Gluck twist.

First we define some notation.  Let $S:=\{1\}\times S^2\subseteq S^1\times S^2$, let $N\in S^2$ be the north-pole, and let $\alpha:=S^1\times\{N\}$.  Let $C$ denote a tubular neighbourhood of $\alpha$, and let $\partial C$ be the boundary torus (see \cref{fig:s1xs2_decomposition}).

\begin{figure}
    \centering
    \includegraphics[width=0.6\linewidth]{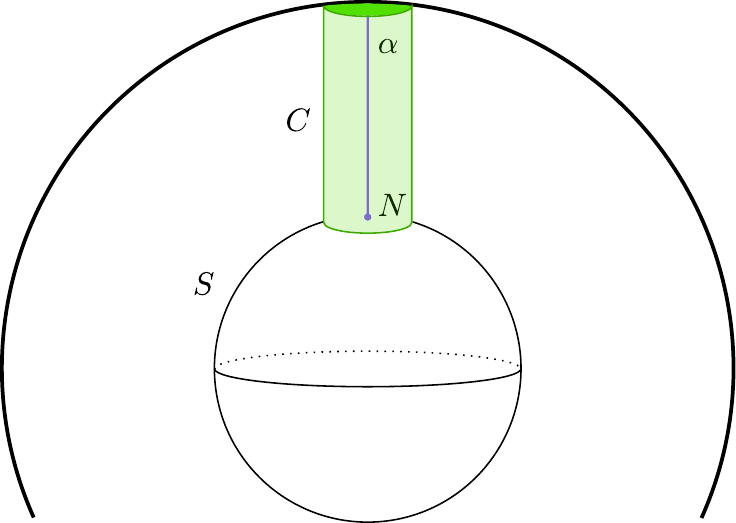}
    \caption{The manifold $S^1\times S^2$ (the outer and inner sphere both denote $S=\{1\}\times S^2$).}
    \label{fig:s1xs2_decomposition}
\end{figure}

Pick $f\colon S^1\times S^2$, a diffeomorphism representing an isotopy class in $\ker\varphi$.  Then the steps are as follows.
\begin{enumerate}[(i)]
    \item Isotope $f$ to a map which fixes $S$ pointwise.
    \item Isotope this new map to one which fixes $S$ and $\alpha$ pointwise.
    \item Isotope this new map to one which fixes $S$ pointwise and fixes $\partial C$ setwise.
    \item Isotope this new map to one which fixes $S$ pointwise and whose restriction to $\partial C$ is the $n$-fold twist map.
    \item Isotope this new map to one which fixes $S$ pointwise and whose restriction to $\partial C$ is either the identity, or the $1$-fold twist map.
    \item Isotope this new map to either the identity map or the Gluck twist $T$.
\end{enumerate}

\begin{step}[Step (i)]
    First, we can isotope $f$ so that $f(S)$ does not intersect $S$.  Now, consider the two regions of $S^1\times S^2$ bounded by $f(S)$ and $S$.  By an argument analogous to that of proving that connected sum is well-defined (Gluck works in the topological category, and so has to cite the annulus theorem for 3-manifolds, but we work in the smooth category where this is a standard exercise), it follows that both of these regions are diffeomorphic to $I\times S^2$.  It follows that $f$ can be further isotoped such that $f(S)=\{-1\}\times S^2$, and then a rotational isotopy gives that we can further isotope $f$ such that $f(S)=S$.

    Now $f$ restricts to a map $f\vert_S\colon S^2\to S^2$ which is degree one since we assumed that $f\in\ker\varphi$.  Hence, by Munkres~\cite{Munkres} or Smale~\cite{Smale_diffeos_S^2}, $f\vert_S$ is isotopic to the identity, and so $f$ is isotopic to a map that fixes $S$ pointwise.
\end{step}

\begin{step}[Step (ii)]
    This step actually has a different interpretation, which we will explain now.  Let $K\colon S^1\hookrightarrow S^1\times S^2$ be a knot in $S^1\times S^2$.  We have an invariant associated to $K$, denoted $g(K)$, called the \emph{geometric winding number}, which is defined as the minimal possible transverse intersections of $K$ with $S$.  We then have the following lemma.

    \begin{lemma}[Lightbulb trick]
        Let $K$ be a knot in $S^1\times S^2$ with $g(K)=1$.  Then $K$ is unknotted, i.e.~isotopic to $\alpha\subseteq S^1\times S^2$.
    \end{lemma}

    Given the above lemma, the step is easy to complete.  We can see that $f(\alpha)$ is a knot in $S^1\times S^2$, and clearly (since intersections between submanifolds are preserved under diffeomorphisms) $g(f(\alpha))=1$, so we can use the lemma to say that $f(\alpha)$ is isotopic to $\alpha$, meaning that $f$ can isotoped such that $f(\alpha)=\alpha$.
\end{step}

\begin{step}[Step (iii)]
    Let $C$ denote a small (open) tubular neighbourhood of $\alpha$ and let $C'$ denote a larger (closed) tubular neighbourhood such that $U:=f(C)\subset C'$ (one can always pick $C$ smaller to obtain this).  Consider the submanifold $C' \sm C$.  The boundary $\partial U$ is a torus, and so Dehn's lemma tells us that $U$ is diffeomorphic to a solid torus (an unknot exterior) if and only if $\pi_1(U)\cong \Z$.  By shrinking $C$ (and hence $f(C)$) we see that $U$ deforms onto $C'\sm\alpha$, which satisfies $\pi_1(C'\sm\alpha)=\Z$.

    Now the solid torus $U$ can be used to guide an isotopy of $h(C)$ to $C'$, and then a further shrinking isotopy to $C$ finishes the proof.
\end{step}

\begin{step}[Step (iv)]
    Fix a framing of the tubular neighbourhood $C$.  We can now consider the restriction $f\colon C\to C$ and look at the difference in the framings between $\alpha$ and $f(\alpha)$.  These framings are classified by maps $\S^1\to \SO(2)$, and hence are classified by integers.  Let $w(f)\in \Z$ denote this framing difference.  Gluck proves the following lemma.
    \begin{lemma}
        Two diffeomorphisms $g,h\colon C\to C$ are isotopic relative to $\partial C$ if and only if $w(g)=w(h)$.
    \end{lemma}
    Since the proof is somewhat technical, we omit it.  We note that Gluck remarks the proof is essentially an involved application of the 2-dimensional Schoenflies theorem.

    Let $n:=w(f)$.  Then the above lemma implies that $f$ and the $n$-fold twist map are isotopic relative to $\partial C$, which completes the step.
\end{step}

\begin{step}[Step (v)]
    Parameterise $\partial C$ as $S^1\times S^1$ where the first $S^1$ factor corresponds to the $S^1$-factor in $S^1\times S^2$ and the second corresponds to the meridian $S^1$-factor.  Then
    \begin{align*}
    f\vert_{\partial C}\colon S^1\times S^1\to& S^1\times S^1,\\
    (x,y)\mapsto& (x,yx^n)
    \end{align*}
    is the $n$-fold twist map.  We can change $n$ by an even number $2k$ by performing $k$ \emph{belt moves}.  This refers to an isotopy where we take both $\alpha$ and a push-off and pull them around $S$, introducing two new twists (it is called so since it can be easily recreated by taking a belt and pulling the middle-portion around one of the ends).  This defines an isotopy of $f$ to a map whose restriction to $\partial C$ is either the identity map (the $0$-fold twist map) or the $1$-fold twist map.
\end{step}

\begin{step}[Step (vi)]
    This is the final step, and the one where we use that $\pi_0\Diff_{\partial}(D^3)$ is trivial -- \Cref{theorem:Cerf}.  Assume that $f$ restricts to the identity $\partial C$.  Then  $S^1\times S^2\sm \partial C$ consists of two $3$-cells, and $f$ cannot interchange them.  Then since $\pi_0\Diff(D^3,\partial)$ is trivial, there exists an isotopy of $f$ restricted to each of these $3$-cells to the identity map that fixes the boundary throughout the isotopy.  Hence, $f$ is isotopic to the identity.

    If we now assume that $f$ restricts to the $1$-fold twist map on $\partial C$, then $f\circ T^{-1}$ restricts to the identity map on $\partial C$, and hence by the previous paragraph we have that $f\circ T^{-1}$ is isotopic to the identity.  Hence $f$ is isotopic to $T$.
\end{step}

This completes the proof of \Cref{prop:S^1xS^2prop1}.
\qed

\subsection{Proof of \Cref{prop:S^1xS^2prop2}}
We give a different proof to that of Gluck.  His proof showed that $T$ was not homotopic to the identity by calculating a certain Hopf invariant.  We instead consider the induced map on \emph{spin structures}.

\begin{definition}
    Let $Y$ be a compact, orientable 3-manifold and let $K$ be a choice of 2-skeleton for $Y$ with 1-skeleton $K_1$.  A \emph{spin structure} on $(Y,K)$ is a choice of trivialisation of $TY_{K_1}$ that extends over $K$, considered up to homotopy of trivialisations.
\end{definition}

\begin{remark}
    Spin structures on $Y$ are a torsor over $H^1(Y;\Z/2)$.  This is most naturally seen by using a different definition of spin structures in terms of homotopy classes of lifts of the tangent bundle.  See \cite[IV]{Kirby1989}.  In particular, this means that $S^1\times S^2$ has two spin structures, restricting to the two spin structures on the core circle $\alpha$.
\end{remark}

Note that during this whole argument there has been a natural choice for the 2-skeleton of $Y$, i.e.~take $K:=\alpha \cup S$ (and so $K_1=\alpha$).  Then the diffeomorphism $T$ already fixes $K$ pointwise and hence the derivative $dT\colon TY\to TY$ induces a map on spin structures.  Since the diffeomorphism may be assumed to be the identity on a small neighbourhood of $S$, the trivialisation of $TY\vert_{S}$ is actually fixed.  However, the trivialisation of $TY\vert_{\alpha}$ is changed by one full rotation around $\alpha$.  The circle admits two distinct spin structures, the bounding and Lie structures, that differ by one full rotation.  Hence $T$ acts non-trivially on the spin structures of $Y$.

Since the identity map clearly induces the trivial map on spin structures, and the induced map on spin structures is invariant under isotopies, it follows that $T$ cannot be isotopic to the identity.  Since $T^2$ is isotopic to the identity (by a belt move) this means that $T$ is order two in the mapping class group of $S^1\times S^2$.

This completes the proof of \cref{prop:S^1xS^2prop2}, and hence of \cref{thm:ker_S^1xS^2}.

\section{Mapping class groups of lens spaces}\label{sec-lens spaces}
In this section we will compute the mapping class group of lens spaces, as presented by Bonahon in \cite{Bonahon1983}. We begine by recalling the construction of lens spaces.

\subsection{Defining $L(p,q)$}
This subsection provides two definitions of a lens space. There are many equivalent definitions possible---the one we will work with in this survey is the second definition of this subsection. Note that we are focusing on 3-dimensional lens spaces, but these definitions can be generalised to work with $n$-dimensional lens spaces. Throughout this section let $p \in \mathbb N$, $q \in \mathbb Z$ be such that $\gcd(p,q)=1$.

The lens space $L(p,q)$ can be defined as the quotient of $S^{3}$ by a free action of $\mathbb{Z}/p$. Here the 3-sphere is modelled as \[S^{3}=\{(z_{1},z_{2}) \in \mathbb C^{2} \mid  \lvert z_{1} \rvert^{2} + \lvert z_{2} \rvert^{2} =1 \}\] and we consider the rotation
\[\rho \colon S^{3} \to S^{3}; \;\;\; \rho(z_{1},z_{2}) = \big( e^{\frac{2\pi i}{p}} \cdot z_{1}, e^{\frac{2\pi q i}{p}} \cdot z_{2}\big).\]
It can easily be verified that for $m \in \mathbb{Z}/ p$ the function
\[m \cdot (z_{1},z_{2}) := \rho^m(z_1,z_2) = \big( e^{\frac{2\pi i m }{p}} \cdot z_{1}, e^{\frac{2\pi q i m }{p}} \cdot z_{2} \big)\]
determines a group action of $\Z/p$ on $S^3$. We check that the action is free. Whenever $m \cdot (z_{1},z_{2}) = (z_{1}, z_{2})$ we have $e^{\frac{2\pi i m }{p}}z_{1} = z_{1}$ and $e^{\frac{2\pi q i m }{p}}z_{2} = z_{2}$. We consider two cases, $e^{\frac{2\pi i m }{p}} = 1$ or $z_1=0$:
\begin{itemize}
    \item If $e^{\frac{2\pi i m }{p}} = 1$, then $p\mid m \implies m = 0 \in \mathbb{Z}/p$.
    \item If $z_{1} =0$, then $\lvert z_{2} \rvert =1 \implies e^{\frac{2\pi q i m }{p}} =1 \implies p\mid mq \implies p\mid m$ (because $\gcd(p,q)=1$) and thus $m = 0 \in \mathbb{Z}/p$.
\end{itemize}
Hence only the identity element of $\Z/p$ fixes a point on $S^{3}$ as a consequence of the fact that $p$ and $q$ are coprime. Therefore the $\mathbb{Z}/p$ action is free.
Note that when $p=2$, $\rho$ is the antipodal map and $L(2,1) \cong \mathbb{R}P^{3}$.

Next we describe an alternative construction of lens spaces, which is the one we will refer to for the remainder of the section. Let $V_{1} \cong V_{2} \cong S^{1} \times D^{2}$ be abstract solid tori. Define $L(p,q) := V_{1} \cup_{\theta} V_{2}$, where $\theta \colon \partial V_{1} \to \partial V_{2}$ is an orientation-reversing diffeomorphism given by $\theta(u,v) = (u^{r}v^{p}, u^{s}v^{q})$, where $r, s \in \mathbb{Z}$ with $qr - ps = -1$.
Here we model
\[S^1 = \{z \in \C \mid |z|=1\} \text{ and } D^2 = \{z \in \C \mid |z| \leq 1\},\]
respectively.
The diffeomorphism class of the lens space does not if we modify $\theta$ by an isotopy.  Thus for this definition, $\theta$ may be considered as a class in $\pi_{0}\Diff(S^{1}\times S^{1}) \cong \operatorname{GL}_{2}(\mathbb Z)$ (for the computation of the mapping class group of $S^1 \times S^1$ see e.g.\ \cite[Theorem 2.5]{FarbMargalit2012}). Using this interpretation, $\theta$ is represented by the matrix
\[\begin{pmatrix}
    r & p \\
    s & q
\end{pmatrix}\]
that maps a meridian
$\begin{psmallmatrix}
    0 \\ 1
\end{psmallmatrix}$
of the torus $S^1 \times S^1$ to
$\begin{psmallmatrix}
    p\\
    q
\end{psmallmatrix}$.\\

In fact, the diffeomorphism type of a lens space only depends on $\theta$ modulo diffeomorphisms which extend over $V_1$ and $V_2$. A lens space defined by $\theta$ is diffeomorphic to a lens space defined by $\theta^\prime$ if and only if $\theta^\prime = h_2 \circ \theta \circ h_1$, with each $h_i$ a self-diffeomorphism of $\partial V_i$ that extends over $V_i$.
This condition is equivalent to saying that we may always substitute the parameters $r, s, p, q$ of $\theta$ with $r^{\prime}$, $s^{\prime}$, $p$, $q^{\prime}$ such that $q^{\prime} r^{\prime} - p s^{\prime} = - 1$ and $q^{\prime} \equiv q \mod{p}$ to obtain a diffeomorphic lens space.

\begin{remark}
    For the remainder of this section we will assume $p \geq 2$.
    In some of the literature, we often find that $L(1,0) \cong S^{3}$ and $L(0,1) \cong S^{1} \times S^{2}$ whenever the definition chosen allows it. The computation of the mapping class group of these two manifolds is different from the general one for lens spaces we shall present, and has already been discussed in previous sections.
\end{remark}

\subsection{Essentially unique genus one Heegaard splittings}

The main result at the core of the computation is the following theorem.

\begin{theorem}\label{thm: UniqueTorus}
    Up to isotopy, the lens space $L(p,q)$ contains a unique torus $T$ separating it into two solid tori.
\end{theorem}

The importance of this theorem does not come from the existence of such a torus, as it is part of the definition of a lens space, but concerns its uniqueness. Moreover, this result was already known before Bonahon, as it had been proved by Schubert in \cite{Schubert1956}. This was done in the context of viewing lens spaces as the double branched cover of $S^{3}$ over a two-bridge knot, and considering an involution $\tau$ on said cover; he then showed that up to $\tau$-equivariant isotopy, $L(p,q)$ contains a unique torus preserved by $\tau$ and separating $L(p,q)$ into two solid tori.

This result was largely ignored for some time, until Bonahon presented a new proof and used it to compute $\pi_{0}\Diff( L(p,q))$. Using Theorem \ref{thm: UniqueTorus}, it is also possible to derive the classification of lens spaces which had been proved by Reidemeister in \cite{Reidemeister1935HomotopieringeUL}.
Furthermore, it is worth noting that the mapping class groups of lens space were computed independently by Hodgson--Rubinstein~\cite{HodgsonRubinstein1983}.

Let us briefly mention the methods behind the proof of Theorem \ref{thm: UniqueTorus}.

\begin{proof}[Sketch proof of \cref{thm: UniqueTorus}]
    The idea is to consider two tori $T$ and $T^{\prime}$ in $L(p,q)$, each separating the lens space into two solid tori. We want to show that $T$ and $T^{\prime}$ are isotopic.

Since $T$ separates $L(p,q)$ into two solid tori $V_{1}$ and $V_{2}$, we consider a map $i\colon D \to L(p,q)$, where $D$ is a disc, such that $i(\partial D)$ is the core of $V_{2}$, $i\vert_{\partial D} \colon \partial D \to i(\partial D)$ is a $p$-sheeted cover, and $i\vert_{\mathring{D}}$ is an embedding avoiding $i(\partial D)$. Bonahon calls $i(D)$ a generalised projective plane and notes that $T$ is isotopic to a the boundary of a tubular neighbourhood of $i(\partial D)$.

Next we work with $T^{\prime}$ and consider a Morse function $f\colon L(p,q) \to \mathbb{R}$ with one critical point of each index $0$, $1$, $2$, and $3$, such that $T^{\prime}$ is isotopic to a level surface between critical points of index 1 and 2. Such a function $f$ can be constructed starting with a suitable Morse function on the standard solid torus.

The idea is then to isotope $i(D)$ so that its singular curve $i(\partial D)$ is as simple as possible with respect to the Morse function $f$. Finally, prove that if $i(D)$ is in a level surface of $f$, then $T$ is isotopic to $T^{\prime}$.
\end{proof}

A complete discussion of the proof of Theorem \ref{thm: UniqueTorus} can be found in \cite{Bonahon1983, Hatcher2001}.

\subsection{Consequences of Theorem \ref{thm: UniqueTorus}}
For us, the most interesting consequences of Theorem \ref{thm: UniqueTorus} concern the mapping class group of lens spaces. In fact, a simple consequence is that any diffeomorphism of $L(p,q)$ is isotopic to a diffeomorphism that preserves the torus $T = \partial V_{1} = \partial V_{2}$.
Whether a diffeomorphism of $T$ extends to the whole lens space depends on the gluing map $\theta$. If such a diffeomorphism extends, then it either preserves both $V_{1}$ and $V_{2}$ or it exchanges them.

Among such diffeomorphisms we always have the involution $\tau$ preserving both $V_{1}$ and $V_{2}$, parametrised on $V_{1}\cong V_{2} \cong S^{1} \times D^{2}$ by $\tau(u,v) = (\overline{u}, \overline{v})$.
In general, there are no diffeomorphisms of $L(p,q)$ exchanging the solid tori $V_{1}$ and $V_{2}$, except when $q^{2} \equiv \pm 1 \mod{p}$.
When $q^{2} \equiv + 1 \mod{p}$ there exists an involution $\sigma_{+}$ of degree $+1$  which can be described by the map $(u, v) \in V_{1} \leftrightarrow (u,v) \in V_{2}$.
Similarly, when $q^{2} \equiv - 1 \mod{p}$, there exists a diffeomorphism $\sigma_{-}$ of degree $-1$ and order $4$, which can be expressed by $(u,v) \in V_{1} \mapsto (\overline{u}, v) \in V_{2}$ and $(u,v) \in V_{2} \mapsto (u, \overline{v}) \in V_{1}$.
Note that $\tau$ commutes with $\sigma_{+}$ and $\sigma_{-}$, and that $\sigma_{-}^{2}= \tau$.\\

It turns out that it is enough for us to work with these three diffeomorphisms.

\begin{proposition}\label{prop:groupgen}
    Any diffeomorphism of $L(p,q)$ is isotopic to an element of the subgroup of $\Diff(L(p,q))$ generated by $\tau$, and possibly $\sigma_{+}$ and $\sigma_{-}$ $($if they are defined$)$.
\end{proposition}

Observe that it is necessary to exclude $L(0,1) \cong S^{1} \times S^{2}$, as the Gluck twist along $\{1\} \times S^{2}$ is not isotopic to a composition of $\tau$, $\sigma_{+}$, and $\sigma_{-}$.

\begin{proof}
    Let $\varphi$ be a diffeomorphism of $L(p,q)$, which by Theorem \ref{thm: UniqueTorus}, can be isotoped so that $\varphi(T) = T = \partial V_{1}= \partial V_{2}$. After possibly composing with $\sigma_{+}$ and $\sigma_{-}$, we can furthermore suppose $\varphi(V_{1})= V_{1}$ and $\varphi(V_{2}) = V_{2}$.
    Let $m_{i}$ be the generator of the kernel of $H_{1}(T) \to H_{1}(V_{i})$, for $i=1,2$, respectively. Since $p\neq 0,1$, we have that the $m_{i}$ are distinct and that $m_{1} \cdot m_{2} \neq \pm 1$. Since $\varphi\vert_{T}$ is a diffeomorphism of~$T$, we deduce that $\varphi_{*}(m_{i}) = \pm m_{i}$ in $H_{1}(T)$, and thus $\varphi_{*}$ is multiplication by $\pm 1$ on $H_{1}(T)$.
    After possibly composing $\varphi$ with $\tau$, we can suppose $\varphi_{*}$ is the identity on $T$. In particular, this implies that $\varphi$ itself is isotopic to the identity on $T$ . We can now isotope $\varphi$ to fix meridian discs of $V_{1}$ and $V_{2}$ and then apply \Cref{theorem:Cerf} to extend the isotopy to the remainder of the solid torus, which is a 3-ball.  Thus, we have shown that after possibly composing $\varphi$ with $\sigma_{\pm}$ we can isotope it to the identity.
\end{proof}

The following lemma will be of use in the computation of $\pi_{0} \Diff( L(p,q))$.

\begin{lemma}\label{lemma: isomaps}
    If $q \equiv \pm 1\mod{p}$, there exists an isotopy of $L(p,q)$ exchanging $V_{1}$ and $V_{2}$. Furthermore, if $p=2$, there exists an isotopy of $L(2,1)\cong \mathbb{R}P^{3}$ whose output coincides with $\tau$ on $T= \partial V_{1} = \partial V_{2}$.
\end{lemma}
\begin{proof}
    The solid torus $V_{1}$ is isotopic to $U(C_{1})$, which denotes the tubular neighbourhood of $C_{1}=S^{1} \times \{0\}$, the core of $V_{1}$. If $q \equiv \pm 1 \mod{p}$, it is possible to choose a parametrisation of $V_{1}\cong V_{2} \cong S^1 \times D^2$ such that $r = \mp 1$. Now, $C_{1}$ is isotopic to $C= S^{1} \times \{1\}$ in $\partial V_{1}$. We may imagine this as `pushing' $C_{1}$ towards the boundary of $V_{1}$. The curve $C$ can also be visualised as the curve with parametrisation $z \in S^{1} \to (z^{r}, z^{s})$ in $\partial V_{2}$.
    Since $r = \mp 1$, $C$ wraps around the longitude of $V_{2}$ once, and is thus isotopic to $C_{2}$, the core of $V_{2}$. Hence we have shown that the core curve of $V_{1}$ can be isotoped into the core curve of $V_{2}$, thus by considering a tubular neighbourhood of such a curve we obtain an isotopy exchanging $V_{1}$ and $V_{2}$.

    If $p=2$, then $L(2,1) \cong \mathbb{R}P^{3}$ and $C_{1}$ is isotopic to $\mathbb{R}P^{1} \subseteq \mathbb{R}P^{2} \subseteq \mathbb{R}P^{3}$. There exists an isotopy of $\mathbb{R}P^{2}$ which reverses the orientation of $\mathbb{R}P^{1}$ and which can be extended to $\mathbb{R}P^{3}$.
    After composing this isotopy of $\mathbb{R}P^{3}$ with an isotopy identifying $V_{1}$ and the tubular neighbourhood of $\mathbb{R}P^1$,  we obtain the desired result.
\end{proof}
Note that $q \equiv \pm 1\mod{p}$ implies $q^2 \equiv 1\mod{p}$. In particular, this means that $\sigma_{+}$ exists.
In fact, by keeping track of what happens in the $D^2$ coordinates of $V_1$ and $V_2$ during the isotopies of \cref{lemma: isomaps}, Bonahon obtained the following relationships:
\begin{itemize}
    \item If $q \equiv 1 \mod{p}$, $\sigma_{+}$ is isotopic to $\tau$;
    \item If $q \equiv -1 \mod{p}$, $\sigma_{+}$ is isotopic to the identity;
    \item If $p=2$, both $\tau$ are isotopic to the identity.
\end{itemize}

We can now already see the following.

\begin{proposition}\label{prop:mcg-L21}
  The mapping class group of $L(2,1) \cong \RP^3$ is  $\pi_{0}\Diff(L(2,1)) \cong \Z/2$, generated by $\sigma_-$. Since $\sigma_-$ is orientation-reversing, $\pi_{0}\Diff^+(L(2,1)) = \{0\}$.
\end{proposition}

\begin{proof}
    By  \cref{prop:groupgen}, we just need to compute the group generated by $\tau$, $\sigma_+$, and $\sigma_-$. But in this case $\sigma_+ =\tau = \Id$, and $\sigma_-^2 = \tau = \Id$, so the claimed result holds.
\end{proof}

\subsection{Computing the mapping class group}

\begin{theorem}\label{thm-MCG lens spaces}
    The group $\pi_{0}\Diff(L(p,q))$ for $p\geq 2$ is isomorphic to:
    \begin{enumerate}[(a)]
        \item\label{lens-space-case-a} $\mathbb{Z}/2$, generated by $\tau$, if $q^{2} \not\equiv \pm 1 \mod{p}$;
        \item\label{lens-space-case-b} $\mathbb{Z}/2 \oplus \mathbb{Z}/2$, generated by $\tau$ and $\sigma_{+}$, if $q^{2} \equiv 1 \mod{p}$ and $q \not\equiv \pm 1 \mod{p}$;
        \item\label{lens-space-case-c} $\mathbb{Z}/2$, generated by $\tau$, if $q \equiv \pm 1 \mod{p}$ and $p\neq 2$;
        \item\label{lens-space-case-d} $\mathbb{Z}/4$, generated by $\sigma_{-}$, if $q^{2} \equiv -1 \mod{p}$ and $p \neq 2$;
        \item\label{lens-space-case-e} $\mathbb{Z}/2$, generated by $\sigma_{-}$, if $p=2$.
    \end{enumerate}
    In each case the orientation-preserving subgroup $\pi_{0}\Diff^+(L(p,q))$ is the subgroup generated by $\tau$ and $\sigma_+$.  Thus
    $\pi_{0}\Diff^+(L(p,q)) = \pi_{0}\Diff(L(p,q))$ in cases \eqref{lens-space-case-a}, \eqref{lens-space-case-b}, and \eqref{lens-space-case-c}, while $\pi_{0}\Diff^+(L(p,q)) \cong \Z/2$ in case \eqref{lens-space-case-d} and $\pi_{0}\Diff^+(L(2,1)) \cong \pi_{0}\Diff^+(\RP^3) = \{0\}$ in case \eqref{lens-space-case-e}.
    \end{theorem}
\begin{proof}
We already saw in \cref{prop:mcg-L21} that $\pi_{0}\Diff(L(2,1)) \cong \Z/2$, so we assume $p \geq 3$ throughout and prove cases \eqref{lens-space-case-a} -- \eqref{lens-space-case-d}.

    Let $G(p,q)$ be the abstract group with generator $\tau$ and, whenever $L(p,q)$ admits them, generators $\sigma_{+}$ and $\sigma_{-}$. Recall that $\sigma_+$ is defined when $q^2 \equiv 1 \mod{p}$ and $\sigma_-$ is defined when $q^2 \equiv -1 \mod{p}$. Define the relations of $G(p,q)$ to be the same as those inherited from the composition of maps previously discussed. We recall them below:
    \begin{itemize}
        \item $\tau^{2} = \sigma_{+}^{2} = \sigma_{-}^{4} = {\rm Id}$;
        \item $\sigma_{-}^{2}= \tau$;
        \item $\tau \sigma_{+} = \sigma_{+}\tau$;
    \end{itemize}
    Therefore, the group $G(p,q)$ is isomorphic to:
    \begin{enumerate}[(a)]
        \item $\mathbb{Z}/2$ when $q^{2} \not\equiv \pm 1 \mod{p}$, as $\sigma_{\pm}$ are not defined, and so $G(p,q)$ is only generated by the involution $\tau$;
        \item $\mathbb{Z}/2 \oplus \mathbb{Z}/2$ when $q^{2} \equiv 1 \mod{p}$ and
         $q \not\equiv \pm 1 \mod{p}$, as in this case $\sigma_{-}$ does not exist;
        \item $\mathbb{Z}/2 \oplus \mathbb{Z}/2$ when $q \equiv \pm 1 \mod{p}$, as in this case $\sigma_{-}$ does not exist;
        \item $\mathbb{Z}/4$ when $q^{2} \equiv -1 \mod{p}$, generated by $\sigma_{-}$, as in this case $\sigma_{+}$ does not exist and $\tau = \sigma_-^2$;
    \end{enumerate}
     We are now interested in the following composition of maps, where $f$ takes a generator of $G(p,q)$ to the corresponding diffeomorphism of $\pi_{0}\Diff(L(p,q))$ and $g_n$ is the map which records the action on the $n$-th homology group:
\begin{displaymath}
\begin{tikzcd}
    G(p,q) \ar[r, "f"] \ar[rr, bend right, "g_n\circ f"] & \pi_{0} \Diff( L(p,q)) \ar[r, "g_n"] & \Aut H_{n} (L(p,q)).
\end{tikzcd}
\end{displaymath}
Note that by Proposition \ref{prop:groupgen} the map $f$ is surjective, and thus we need to compute the kernel of $f$ to determine $\pi_{0} \Diff( L(p,q))$. For most of this proof we will only be interested in $H_{1}(L(p,q)) \cong \mathbb{Z}/p$, on which the maps $\tau$, $\sigma_{+}$, $\sigma_{-}$ act by multiplication by $-1$, $-q$, and $q$, respectively. From this it follows that when $q \not\equiv \pm 1 \mod{p}$ the homomorphism $g_1\circ f$ is injective, thus $f$ is injective and $\pi_{0}\Diff( L(p,q)) \cong G(p,q)$. When working with $H_{1}(L(p,q))$ is not enough, we will consider $H_3(L(p,q))$. The diffeomorphisms $\tau$, $\sigma_{+}$ and $\sigma_{-}$ act on $H_3(L(p,q))$ via multiplication by $1$, $1$ and $-1$, respectively.

With these observations in mind we now compute the mapping class group of $L(p,q)$ for our different cases.
\begin{enumerate}[(a)]
    \item  If $q^{2} \not\equiv \pm 1 \mod{p}$, then $q \not\equiv \pm 1 \mod{p}$ and, by the aforementioned statements, $f$ is injective and $\pi_{0}\Diff( L(p,q)) \cong G(p,q) \cong \mathbb{Z}/2$, generatred by $\tau$. Recall that in this case $\sigma_{-}$ and $\sigma_{+}$ are not defined.

    \item  Let $q^{2} \equiv 1 \mod{p}$ and $q \not\equiv \pm 1 \mod{p}$. Hence $g_1 \circ f$ is injective. Indeed note that $\tau$, $\sigma_{+}$, and $\tau\sigma_{+}$ act nontrivially on $H_{1}(L(p,q))$. Thus $\pi_{0}\Diff( L(p,q)) \cong G(p,q) \cong \mathbb{Z}/2 \oplus \mathbb{Z}/2$, generated by $\tau$ and $\sigma_+$. In this case $\sigma_-$ does not exist as $q^2 \not\equiv -1 \mod{p}$.

    \item We consider the case that $q \equiv \pm 1 \mod{p}$, divided into the two cases of $q \equiv 1 \mod{p}$ and $q \equiv -1 \mod{p}$.
    \begin{enumerate}[(i)]
        \item   Let $q \equiv 1 \mod{p}$. This implies that $G(p,q) \cong \mathbb{Z}/2 \oplus \mathbb{Z}/2$. Note that $\tau\sigma_{+}$ acts via multiplication by $q$ on $H_{1}(L(p,q))\cong \mathbb{Z}/p$ and hence the action is trivial. The action of $\tau$ on $H_{1}(L(p,q))$ implies that $\tau$ is not isotopic to the identity, therefore $\ker(g_1\circ f) \cong \mathbb{Z}/2$. By Lemma \ref{lemma: isomaps}, $\sigma_{+}$ is isotopic to $\tau$, thus $\tau\sigma_{+}$ is mapped to the identity in $\pi_{0}\Diff( L(p,q))$ and $\ker(f) \cong \mathbb{Z}/2$. Hence $\pi_{0}\Diff( L(p,q))  \cong G(p,q)/\ker(f) \cong \mathbb{Z}/2$, generated by $\tau$.
        \item
            Now let $q \equiv -1 \mod{p}$. Again, $G(p,q) \cong \mathbb{Z}/2 \oplus \mathbb{Z}/2$. In this case $\sigma_{+}$ acts trivially on $H_{1}(L(p,q))$, so $\ker(g_1\circ f) \cong \mathbb{Z}/2$. By Lemma \ref{lemma: isomaps}, we have that $\sigma_{+}$ is isotopic to the identity, so $\ker(f) \cong \mathbb{Z}/2$ and again $\pi_{0}\Diff( L(p,q)) \cong \mathbb{Z}/2$, generated by $\tau$.
            \end{enumerate}
        \item Now suppose that $q^{2} \equiv -1 \mod{p}$.
        The action of $\tau$ and $\sigma_{-}$ is not trivial on $H_{1}(L(p,q))$.  Moreover, we know that $\tau$ acts trivially on $H_{3}(L(p,q))$, while  $\sigma_{-}$ acts nontrivially.  Therefore $g_1 \circ f$ is injective, and $\pi_{0}\Diff( L(p,q)) \cong G(p,q) \cong \mathbb{Z}/4$, generated by $\sigma_-$. \qedhere
\end{enumerate}
\end{proof}

This computation concludes our discussion of the mapping class groups of lens spaces. We once again highlight the importance of Theorem \ref{thm: UniqueTorus}, which serves as the basis for subsequeent results. Finally, we remark that the generalised Smale conjecture (\cref{conj - generalised smale}) holds for lens spaces, as proved in \cite{HKMR2012}, and more recently by Ketover--Liokumovich in \cite[Theorem~2.8]{KetoverLiokumovich2023}.

\section{Mapping class groups of elliptic 3-manifolds}\label{sec-elliptic}

In this section, we discuss elliptic manifolds.  In particular, we present their classification, talk about how to compute their isometry groups as well as their mapping class groups, and discuss the generalised Smale conjecture (\cref{conj - generalised smale}). We follow the exposition and notation of \cite{McCullough2002}. Specifically, we adopt the notation $\mathcal{I}(M)$ for $\pi_0\Isom(M)$ which in our case is the same as the mapping class group of $M$ (cf.~\cref{isometry table}).

\begin{theorem}\label{thm: MCG elliptic manifolds}
For $M$ an elliptic 3-manifold with fundamental group $\pi_1(M)\cong G$, the mapping class group is $\pi_0\Diff(M)\cong \pi_0 \Isom(M)$, and  $\pi_0\Isom(M)$ is given in \cref{isometry table}.
\end{theorem}

Throughout this section $M$ will denote a closed, orientable $3$-manifold. We start by defining elliptic manifolds, the following being equivalent to \cref{defn-elliptic}.

\begin{definition}
    A closed $3$-manifold $M$ is \textit{elliptic} if it is of the form $M = S^3/G$ where $G$ is a finite subgroup of $\SO(4)$ acting freely on $S^3$.
\end{definition}

The definition of an elliptic manifold $M$ is equivalent to requiring that $M$ admits a Riemannian metric of constant curvature 1. Note that if $M = S^3/G$, then $\pi_1(M) \cong G$ and therefore every elliptic manifold has finite fundamental group. The converse follows from Perelman's \emph{elliptisation theorem}: elliptic manifolds are exactly the ones with finite fundamental group \cite[Thm.~1.7.3]{AFW2015}.

Lens spaces are exactly the elliptic manifolds such that $G$ is a cyclic group. Since we discussed lens spaces in the previous section, we restrict here to a discussion of elliptic manifolds with non-cyclic fundamental group. We start by presenting the classification of such manifolds.

\subsection{Classification}
We present the classification of elliptic manifolds with non-cyclic fundamental group. From the definition of elliptic manifolds, it follows that there is a one-to-one correspondence between isomorphism classes of elliptic manifolds and conjugacy classes of finite $G \subseteq \SO(4)$ acting freely on $S^3$. Therefore, we need to study subgroups of $\SO(4).$

We equip $S^3$ with the group structure arising from viewing it as the unit sphere in the space of quaternions. Each element in $S^3$ acts on $S^3$ by left (or right) multiplication and hence defines an orientation preserving isometry of $S^3$, or in other words an element of $\SO(4)$. Consider now the following Lie group homomorphism

$$F \colon S^3 \times S^3 \to \SO(4),$$
$$(q_1, q_2) \mapsto (x \mapsto q_1xq_2^{-1}).$$\\
$F$ is the universal covering of $\SO(4)$ and $\ker(F) = \{(1,1), (-1,-1)\} \cong \mathbb{Z}/2$.

Let $P$ be the $2$-sphere of unit quaternions with real part $0$. Then $P \subseteq S^3.$ Carrying out the computation shows that conjugating by an element in $S^3$ preserves the condition of having real part $0$.  Therefore, we obtain a map
$$S^3 \to \Isom^+(P),$$
$$q \mapsto (c_q)|_P$$
where $c_q\colon S^3 \to S^3, \ x \mapsto qxq^{-1}.$ Since $\Isom^+(P) \cong \SO(3),$ we can view the above as a map
$$H \colon S^3 \to \SO(3).$$
In fact this map is the universal covering of $\SO(3)$ and $\ker(H) = \{1, -1\} \cong \mathbb{Z}/2.$

Let $\mathrm{D}_{2n}$ be the dihedral group of order $2n$. Furthermore, let $\mathrm{T}_{12}, \mathrm{O}_{24}$ and $\mathrm{I}_{60}$ be the groups of orientation preserving symmetries of the tetrahedron, octahedron and icosahedron respectively. (The subscripts denote the order of the respective group.) Then $\mathrm{D}_{2n}, \mathrm{T}_{12}, \mathrm{O}_{24}$, and $\mathrm{I}_{60}$ are subgroups of $\SO(3)$ and in fact (apart from cyclic ones) they are the only finite subgroups \cite[Thm.~4.2.2]{Ka16}. Denote by $\mathrm{D}^*_{4n}, \mathrm{T}^*_{24}, \mathrm{O}^*_{48}, \mathrm{I}^*_{120}$ the preimages of the above groups in $S^3$ under the map $H$. These groups are called the binary dihedral group, binary tetrahedral group etc.

Let $\{1, i, j, k\}$ be the standard basis for the quarternions. The set \[\mathrm{Q}_8 := \{\pm 1, \pm i, \pm j, \pm k\}\] forms a subgroup of $S^3$ called the quaternion group.

For $n \in \mathbb{N}$, let $\xi_n := \cos{\frac{2\pi}{n}} + i \sin{\frac{2\pi}{n}}$. Then $\xi_n$ generates a cyclic subgroup of $S^3$ of order $n$, which we denote by $C_n$.

We are now ready to present the classification of closed elliptic 3-manifolds. \cref{classification table} contains all the groups (apart from cyclic subgroups that correspond to lens spaces) arising as fundamental groups of elliptic manifolds. Furthermore, for each group $G$ below, there is a unique conjugacy class in $\SO(4)$ consisting of isomorphic copies of $G$, which implies that for every $G$ there is a unique elliptic manifold (up to isomorphism) with fundamental group $G$ (e.g.~ \cite[Theorem 2.10]{Sakuma1990}). Note that this is not true in the case of lens spaces where $G$ is cyclic.

\begin{table}[ht]
\begin{center}
\begin{tabular}{ l l }
 \textbf{Fundamental group} & \textbf{Name of elliptic manifold}\\[0.5ex]
 $\mathrm{D}^*_{4n}$ & prism  \\[0.5ex]
 $\mathrm{D}^*_{4n} \times C_m$ & prism  \\[0.5ex]
 $H \leq \mathrm{D}^*_{4n} \times C_{4m}$ & prism \\[0.5ex]
 $\mathrm{T}^*_{24}$ & tetrahedral \\[0.5ex]
 $\mathrm{T}^*_{24} \times C_n$ & tetrahedral \\[0.5ex]
 $K \leq \mathrm{T}^*_{24} \times C_{6n}$ & tetrahedral \\[0.5ex]
 $\mathrm{O}^*_{48}$ & octahedral \\[0.5ex]
 $\mathrm{O}^*_{48} \times C_n$ & octahedral \\[0.5ex]
 $\mathrm{I}^*_{120}$ & icosahedral \\[0.5ex]
 $\mathrm{I}^*_{120} \times C_n$ & icosahedral \\[0.5ex]
 $\mathrm{Q}_8$ & quaternionic \\[0.5ex]
 $\mathrm{Q}_8 \times C_n$ & quaternionic
\end{tabular}
\end{center}
\caption{Classification of elliptic manifolds}
\label{classification table}
\end{table}
The groups $H$ are certain index $2$ subgroups of $\mathrm{D}^*_{4n} \times C_{4m}$ with $(n,m) = 1$. The groups $K$ are certain index $3$ subgroups of $\mathrm{T}^*_{24} \times C_{6n}$ with $n$ odd and divisible by $3$. The names `prism' etc.\ stand for what the corresponding manifolds are called; for example octahedral manifolds are exactly the elliptic manifolds with a subgroup of their fundamental group isomorphic to $\mathrm{O}^*_{48}$. Note that all the groups in the above table are naturally subgroups of $S^3 \times S^3$, that are mapped to an isomorphic image under $F$ and therefore can also be viewed as subgroups of $\SO(4)$. More detail on the classification can be found in \cite{HKMR2012}.

\subsection{Isometries}
In this section, we determine the isometry groups of all the elliptic manifolds occurring in the classification table from the previous section. This means that we again omit discussing lens spaces. Among all elliptic 3-manifolds, only certain lens spaces admit orientation-reversing isometries \cite[Proposition 1.1]{McCullough2002}. Hence, for our discussion, we have $\Isom(M) = \Isom^+(M)$, where $M$ is any elliptic 3-manifold with non-cyclic fundamental group.

Let $G \subseteq \SO(4)$ be such that $M = S^3/G$ is elliptic. Given $f \in \SO(4)$, we can define a unique $\Bar{f} \in \Isom(M)$ such that the following diagram commutes
\begin{center}
\begin{tikzcd}
S^3 \arrow{r}{f} \arrow{d}{} & S^3 \arrow{d}{}\\
M \arrow{r}{\Bar{f}} & M,
\end{tikzcd}
\end{center}
if and only if the following holds.
\[\text{For all } x \in S^3 \text{ and for all } g\in G, \text{ there exists } \ g' \in G \text{ such that }  f(g(x)) = g'(f(x)).\]
The final condition is equivalent to $f \in \Norm(G) = \{h \in \SO(4) \mid hGh^{-1} = G\}$. Hence, we obtain a map $\Norm(G) \to \Isom(M)$ which sends $f$ to $\Bar{f}$. Since $S^3$ is simply connected, every isometry of $M$ can be lifted (non-uniquely) to an isometry of $S^3$, which implies that the map $\Norm(G) \to \Isom(M)$ is surjective. Any two lifts of an isometry of $M$ differ by a deck transformation, i.e.~an element of $G$, so we have the short exact sequence
$$1 \to G \to \Norm(G) \to \Isom(M) \to 1.$$
We summarise the above discussion as a proposition.

\begin{proposition}
    For an elliptic manifold $M = S^3/G$ with $G$ not cyclic, we have that $\Isom(M) \cong \Norm(G)/G$.
\end{proposition}

The above proposition also holds for lens spaces that do not admit orientation-reversing isometries (those for which $\sigma_-$ is not defined). If a lens space does admit orientation-reversing isometries, its isometry group is an extension of $\mathbb{Z}/2$ by $\Norm(G)/G$.

In \cref{isometry table} we present all elliptic manifolds with non-cyclic fundamental group and their isometry groups. We also compute the group of connected components $\pi_0\Isom(M)$, which we denote by $\mathcal{I}(M)$. The explicit computations can be found in \cite{McCullough2002}.

\begin{table}[ht]
\begin{center}
\begin{tabular}{ l l l}
 $\mathbf{G}$ & \textbf{Isom(M)} & $\mathbf{\mathcal{I}(M)}$ \\[2ex]
 $\mathrm{D}^*_{4n}$ & $\SO(3) \times \mathbb{Z}/2$ & $\mathbb{Z}/2$ \\[1ex]
 $\mathrm{D}^*_{4n} \times C_m$ & $O(2) \times \mathbb{Z}/2$ & $\mathbb{Z}/2 \times \mathbb{Z}/2$\\[1ex]
 $H \leq \mathrm{D}^*_{4n} \times C_{4m}$ & $\mathrm{O}(2) \times \mathbb{Z}/2$ & $\mathbb{Z}/2 \times \mathbb{Z}/2$\\[1ex]
 $\mathrm{T}^*_{24}$ & $\SO(3) \times \mathbb{Z}/2$ & $\mathbb{Z}/2$\\[1ex]
 $\mathrm{T}^*_{24} \times C_n$ & $\mathrm{O}(2) \times \mathbb{Z}/2$ & $\mathbb{Z}/2 \times \mathbb{Z}/2$\\[1ex]
 $K \leq \mathrm{T}^*_{24} \times C_{6n}$ & $\mathrm{O}(2)$ & $\mathbb{Z}/2$\\[1ex]
 $\mathrm{O}^*_{48}$ & $\SO(3)$ & 1\\[1ex]
 $\mathrm{O}^*_{48} \times C_n$ & $\mathrm{O}(2)$ & $\mathbb{Z}/2$\\[1ex]
 $\mathrm{I}^*_{120}$ & $\SO(3)$ & 1\\[1ex]
 $\mathrm{I}^*_{120} \times C_n$ & $\mathrm{O}(2)$ & $\mathbb{Z}/2$\\[1ex]
 $\mathrm{Q}_8$ & $\SO(3) \times S_3$ & $S_3$\\[1ex]
 $\mathrm{Q}_8 \times C_n$ & $\mathrm{O}(2) \times S_3$ & $\mathbb{Z}/2 \times S_3$
\end{tabular}
\end{center}
\caption{Isometry groups of elliptic manifolds}
\label{isometry table}
\end{table}
\vspace{0.5cm}

\subsection{Mapping class groups}
Let $M = S^3/G$ be an elliptic manifold. (We should note that everything in this section holds in particular for lens spaces.) The mapping class group of $M$ is $\pi_0\Diff(M)$. In the previous section, we already computed $\mathcal{I}(M) = \pi_0\Isom(M)$. The following theorem (also known as the $\pi_0$-Smale conjecture) shows that these two groups are the same.

\begin{theorem}[{\cite[Thm.~3.1]{McCullough2002}}]\label{pi0Smale}
    The inclusion $\iota \colon \Isom(M) \to \Diff(M)$ induces a bijection on path components.
\end{theorem}
\begin{remark}
The proof of this theorem relies on already knowing most of the groups $\pi_0\Diff(M)$, since historically these were determined earlier.
\end{remark}

We now present a sketch of how $\pi_0\Diff(M)$ can be computed without knowing $\mathcal{I}(M)$ for prism manifolds (i.e.~the ones having a dihedral group as a subgroup of their fundamental group). See \cite{Asano1978} for a detailed proof.

Consider the following construction. Take an orientable $S^1$-bundle $\pi \colon N \to B$ over a M\"obius band $B$. Then the boundary $\partial N$ is a torus. Glue a solid torus $V$ to $N$ along their boundaries. For certain choices of gluing, the resulting manifold $M$ turns out to be a prism manifold, and in fact every prism manifold can be constructed this way.

Let $\alpha$ be the core curve of $B$. Then $\pi^{-1}(\alpha) =: K$ is a Klein bottle.
One can show that $K$ is incompressible and that every diffeomorphism of $K$ extends to one of~$M$.  Furthermore, any isotopy of $K$ extends to an isotopy of~$M$, which implies that we have a well-defined map $\rho \colon \pi_0 \Diff(K) \to \pi_0 \Diff(M).$ For most prism manifolds~$M$, the surface~$K$ is the unique incompressible Klein bottle in $M.$ Therefore, every diffeomorphism of $M$ can be isotoped such that it fixes $K$. Hence, for a given diffeomorphism of $M,$ we obtain (after a possible isotopy) a diffeomorphism of $K$ and applying $\rho$ to the class of this diffeomorphism yields back the original one. In other words, the homomorphism $\rho$ is surjective.

It is known that $\pi_0\Diff(K) \cong \mathbb{Z}/2 \times \mathbb{Z}/2$, generated by a Dehn twist about a two-sided curve and the $y$-homeomorphism of Lickorish~\cite{Lickorish1963}. By the above discussion, computing the kernel of $\rho$ determines $\pi_0\Diff(M)$.

The kernel will depend on the gluing used in the construction of $M$. For example, assume that the meridian of the solid torus is glued to a curve in $\partial N$ that bounds a M\"obius band. Then there is an embedded projective plane inside $M$, and in fact $M$ is an $S^1$-bundle over $\mathbb{R}\mathrm{P}^2.$ There is an isotopy of the projective plane that switches the orientation of one-sided curves. This can be lifted to an isotopy of $M$ switching the orientation of one-sided curves on $K.$ The latter implies that the $y$-homeomorphism is in the kernel of $\rho$. The Dehn twist is not in the kernel, since it acts nontrivially on $\pi_1(M)$. We conclude that $\ker(\rho) \cong \mathbb{Z}/2$ and consequently $\pi_0\Diff(M) \cong \mathbb{Z}/2$ in this particular case.

We end this section by presenting the so-called realisation theorem \cite[Thm.~1.3]{McCullough2002}, which states that the mapping class group can be modelled as a group of isometries. Let $\iota \colon \Isom(M) \to \Diff(M)$ and $p \colon \Diff(M) \to \pi_0\Diff(M)$ be the natural inclusion and projection maps, respectively. Let $\phi = p \circ \iota$ be the composition.

\begin{theorem}
    There is a subgroup $\Theta \subseteq \Isom(M)$ such that the map $$\phi \colon \Isom(M) \to \pi_0\Diff(M)$$ restricted to $\Theta$ induces an isomorphism $\Theta \xrightarrow{\cong} \pi_0\Diff(M)$.
\end{theorem}

\begin{proof}
    Consider the short exact sequence
    $$1 \to \Isom_0(M) \to \Isom(M) \to \mathcal{I}(M) \to 1,$$
    where $\Isom_0(M)$ denotes the path-component of the identity of $\Isom(M)$ and $\mathcal{I}(M) = \pi_0\Isom(M)$. Using that
    $$1 \to \SO(2) \to \mathrm{O}(2) \to \mathbb{Z}/2 \to 1$$
    is a split short exact sequence, and checking every possible case for every $M$ in \cref{isometry table}, one sees that the above sequence always splits. (The same holds for all lens spaces by checking \cite[Table 3]{McCullough2002}.) 

    Let $s \colon \mathcal{I}(M) \to \Isom(M)$ be a section of the natural map $\Isom(M) \to \mathcal{I}(M)$ and let $\Theta := \im{s}$. Restrict $\phi \colon \Isom(M) \to \pi_0\Diff(M)$ to the subgroup $\Theta$, i.e.~consider $\phi|_\Theta \colon \Theta \to \pi_0\Diff(M)$. By the definition of $\Theta$ and \cref{pi0Smale}, we obtain that $\phi|_\Theta$ is an isomorphism.
\end{proof}

\subsection{The Smale conjecture}
Let $M$ be an elliptic 3-manifold. The generalised Smale conjecture (\cref{conj - generalised smale}) asks whether $\Isom(M) \to \Diff(M)$ is a homotopy equivalence. From \cref{pi0Smale} we know that $\Isom(M)$ and $\Diff(M)$ have the same $\pi_0$, i.e.~the same path components. This reduces the generalised Smale conjecture in the case of elliptic 3-manifolds to the weak Smale conjecture (\cref{thm-weak Smale}), which asks whether the inclusion of the components of the identity $\Isom_0(M) \to \Diff_0(M)$ is a homotopy equivalence. This has recently been proven affirmatively, combining earlier work of Ivanov \cite{Ivanov79, Ivanov82} with \cite{HKMR2012} and the final cases appearing in the work of Bamler--Kleiner \cite{BamlerKleiner2019, BamlerKleiner2023b}. This leads to the following theorem.

\begin{theorem}
    Let $M$ be an elliptic manifold. Then $\Isom(M) \to \Diff(M)$ is a homotopy equivalence.
\end{theorem}

\section{Mapping class groups of Haken 3-manifolds}\label{sec-Haken}

Throughout this section, our 3-manifolds $M$ will be compact and orientable, with possibly nonempty boundary.

\subsection{Definition of a Haken 3-manifold}
Part (4) of the following definition agrees with~\cref{defn-Haken}.

\begin{definition} Let $M$ be a compact, orientable 3-manifold.
\begin{enumerate}
    \item We say that $M$ is \emph{irreducible} if and only if every 2-sphere in $M$ bounds a copy of $D^3$ in $M$.
    \item We say that an orientable embedded surface $F$ in $M$ with $\partial F \subseteq \partial M$ is \emph{compressible} if either
    \begin{enumerate}[(a)]
        \item there exists a curve $k \subseteq \mathring{F}$, homotopically essential in $F$, and a disc $D \subseteq M$ with $\mathring{D} \subseteq \mathring{M}$ and $D \cap F = \partial D =k$; or
        \item there exists a 3-ball $E \subseteq M$ with $E \cap F = \partial E$.
    \end{enumerate}
    If an orientable, embedded surface $F$ is not compressible, we say that $F$ is \emph{incompressible}.
    \item We say that $M$ is \emph{boundary irreducible} if and only if $\partial M$ is incompressible.
    \item We say that $M$ is \emph{Haken} (or sufficiently large) if it contains an incompressible surface.
\end{enumerate}
\end{definition}

\begin{remark}\leavevmode \label{remark:non-boundary-Haken}
We discuss some consequences of the definition, in particular the relationship between Haken and Seifert fibred 3-manifolds.
\begin{enumerate}[(i)]
    \item  \label{item:nbH} If $\partial M \neq \emptyset$, then $M$ is Haken. This is because a small properly embedded boundary parallel disc is incompressible.  So $D^3$ is Haken, although it is not boundary irreducible.
    \item A closed Seifert fibred 3-manifold is Haken if it does not admit any Seifert fibration with base orbifold $S^2$ and at most three singular fibres~\cite[Theorem~VI.15]{Jaco1980}.
    \item Seifert fibred 3-manifolds with base orbifold $S^2$ and at most two singular fibres are lens spaces, $S^1 \times S^2$, or $S^3$. Lens spaces and $S^3$ are not Haken, but $S^1 \times S^2$ is Haken (it is, however, reducible).
    \item A Seifert fibred 3-manifold $M$ with base orbifold $S^2$ and three singular fibres is Haken if and only if $H_1(M;\Z)$ is infinite, again by \cite[Theorem~VI.15]{Jaco1980}.
   \item On the other hand, a closed, orientable Haken 3-manifold $M$ admits a Seifert fibred structure if and only if $\pi_1(M)$ has an infinite cyclic normal subgroup~\cite[Theorem~VI.24]{Jaco1980}.
\end{enumerate}
\end{remark}

\begin{lemma}\label{lemma:kpi1}
    Let $M$ be a compact, orientable 3-manifold that is irreducible, boundary irreducible, and Haken. Then $M \simeq K(\pi_1(M),1)$, with $\pi_1(M)$ infinite.
\end{lemma}

\begin{proof}
   Since $M$ is irreducible, the sphere theorem (see e.g.~ \cite{Hempel1976}) implies that $\pi_2(M)=0$.
    We argue that $\pi_1(M)$ is infinite.
    Suppose that $\partial M = \emptyset$. Then $M$ contains a closed incompressible surface $F$, which must be orientable and have positive genus because $M$ is irreducible.  By the loop theorem, $\pi_1(F)$ injects into $\pi_1(M)$, and hence $\pi_1(M)$ is infinite.
    Now suppose that $\partial M \neq \emptyset$. If $\p M$ contains $S^2$ as a connected component, then $M \cong D^3$ by irreducibility. But $\p D^3$ is compressible, and by hypothesis $M$ is boundary incompressible, hence $M$ is not $D^3$. Since $M$ is orientable, so is $\p M$, and it follows that every component of $\partial M$ is a closed, orientable surface of positive genus.  Since $\p M$ is incompressible, $\pi_1(\p M) \to \pi_1(M)$ is injective, and hence $\pi_1(M)$ is infinite.

    Now we consider the universal cover $\wt{M}$ of $M$. We have $\pi_1(\wt{M})=0$, and $\pi_2(\wt{M})= \pi_2(M)=0$, as noted above. Hence $H_1(\wt{M}) = 0 = H_2(\wt{M})$ by the Hurewicz theorem.  Since $\pi_1(M)$ is infinite, $\wt{M}$ is a noncompact 3-manifold, so $H_3(\wt{M})=0$. Hence $H_i(\wt{M}) =0$ for all $i>0$. By the Hurewicz theorem, $\pi_i(\wt{M})=0$ for all $i >0$. So $M \simeq K(\pi_1(M),1)$ as desired.
\end{proof}

As a consequence of the lemma, there is a bijection between homotopy classes of homotopy equivalences of $M$ and outer automorphisms of $\pi_1(M)$.

\begin{lemma}\label{lemma:haut-equals-out-for-Kpi1}
Let $X$ be a $K(\pi,1)$-space where $\pi$ is a group. Then, for each choice of basepoint $x_0 \in X$ and identification $\pi = \pi_1(X,x_0)$, there is an isomorphism
\[ \phi \colon \pi_0 \hAut(X) \to \Out(\pi)\]
given by sending $f \mapsto [(f_{\gamma})_*]$ where $(f_{\gamma})_* \colon \pi \to \pi$, $[\alpha] \mapsto [\gamma \cdot (f \circ \alpha) \cdot \gamma^{-1}]$ where $\gamma \colon [0,1] \to X$ is any choice of path from $x_0$ to $f(x_0)$.
\end{lemma}

\begin{proof}
This is a straightforward consequence of combining \cite[Proposition 1B.9]{Ha02}, which gives a correspondence between homotopy classes of based homotopy automorphisms of $X$ with $\Aut(\pi)$, and the fact that change of basepoint corresponds to conjugation on $\pi_1$ (see, for example, \cite[p59]{Ha02}).
\end{proof}

\subsection{Statement of the main theorems on Haken 3-manifolds}
In the case that $\p M \neq \emptyset$, to be able to compare with diffeomorphisms, we restrict to  homotopy equivalences that map boundary to boundary. On the level of fundamental groups, this leads to the following definition.

Let $\Aut_{\p}(\pi_1(M))$ denote the automorphisms of $\pi_1(M)$ that preserve the peripheral structure. Here an automorphism $\psi$ of $\pi_1(M)$ is said to \emph{preserve the peripheral structure} if for every connected component $F \subseteq \partial M$, there exists a  connected component $G \subseteq \p M$ and subgroup $A \leq \pi_1(M)$ that is conjugate to $\pi_1(G) \subseteq \pi_1(M)$, such that $\psi(\pi_1(F)) \subseteq A$.

We can now state Waldhausen's main result~\cite{Waldhausen1968} on the mapping class groups of Haken 3-manifolds.

\begin{theorem}[Waldhausen]\label{thm:Waldhausen}
        Let $M$ be a compact, orientable 3-manifold that is irreducible, boundary irreducible, and Haken. Then taking the induced action on $\pi_1$ determines an isomorphism
    \[\rho \colon \pi_0 \Diff(M) \xrightarrow{\cong} \Aut_\p (\pi_1(M))/\Inn(\pi_1(M)).\]
\end{theorem}

More generally, after Waldhausen, Hatcher and Ivanov~\cite{Hatcher1976,Ivanov1976} proved the following theorem. Here, on both sides, the boundary must be fixed pointwise. That is, $\Diff_\p (M)$ denotes the topological group of diffeomorphisms of $M$ that fix $\partial M$ pointwise, and $\hAut_\p(M)$ denotes the topological monoid of homotopy equivalences of $M$ that fix $\partial M$ pointwise.

\begin{theorem}[Hatcher, Ivanov]
    Let $M$ be a compact, orientable 3-manifold that is irreducible, boundary irreducible, and Haken. Then the forgetful map is a homotopy equivalence
        \[\Diff_\p (M) \xrightarrow{\simeq} \hAut_\p(M).\]
\end{theorem}

So for Haken 3-manifolds there is an extremely strong correspondence between homotopy equivalences and diffeomorphisms. This contrasts with the results on geometric 3-manifolds, as described for example in \cref{sec-elliptic} and \cref{sec-hyperbolic}, where diffeomorphism groups are compared with isometry groups.

Next we mention an interesting theorem on mapping class groups of Haken 3-manifolds due to Johannson~\cite{Johannson1979}.

\begin{theorem}[Johannson]\label{thm:Johannson-finiteness-Haken}
    Let $M$ be an irreducible, boundary irreducible, and Haken 3-manifold.
Suppose also that every incompressible annulus or torus in $M$ is boundary parallel. Then $\pi_0 \Diff(M)$ is finite.
\end{theorem}

This includes cases where the analogous algebraic fact is hard to see, so Johannson also proved new algebraic results on finiteness of $\Out(\pi_1(M))$ for $M$ a closed, Haken 3-manifold.  See \cref{section:finiteness-properties} for more information on finiteness properties of mapping class groups of 3-manifolds.

\subsection{Sketch proof of surjectivity in \cref{thm:Waldhausen}}
Now we begin to outline the methods used in Waldhausen's proof. As well as the original article, good references include \cite[Chapter~13]{Hempel1976} and \cite{Jaco1980}.
The key is Haken's notion of hierarchies~\cite{Haken}.

\begin{definition}[Hierarchies]
Let $M_1$ be an irreducible 3-manifold.  A \emph{hierarchy} for $M_1$ of length $n$ is a sequence of triples
\[\big\{(M_j,F_j \subseteq M_j,N(F_j)) \big\}_{j=1}^n,\]
 where $F_j \subseteq M_j$ is an incompressible surface, and $N(F_j)$ is a tubular neighbourhood of $F_j$, such that $M_{j+1} = M_j \sm \mathring{N(F_j)}$ and $M_{n+1}$ is a union of 3-balls.
\end{definition}

\begin{theorem}[Existence of hierarchies]\label{thm:existence-hierarchies}
    Let $M_1$ be an irreducible 3-manifold with $\p M_1 \neq \emptyset$. Then there exists finite $n$ and a hierarchy for $M_1$ of length $n$,  $\big\{(M_j,F_j \subseteq M_j,N(F_j))\big\}_{j=1}^n$ such that $0 \neq [\partial F_j] \in H_1(\p M_j)$ for all $j$.
\end{theorem}

\begin{proof}
    See \cite[Section~2]{Waldhausen1968} or \cite[Thm~13.3]{Hempel1976}.
\end{proof}

To apply the theorem to a closed Haken manifold, we first cut along the given incompressible surface.
The condition that $0 \neq [\partial F_j] \in H_1(\partial M)$ is important in the use of \cref{thm:existence-hierarchies}, but this will not be explained in this section.

To prove \cref{thm:existence-hierarchies}, Waldhausen defined a complexity of a handle decomposition, in terms of three nonnegative integers. He took an incompressible surface, and modified it to be a `normal surface'. Cutting along a normal surface reduces the complexity, and hence the procedure terminates in finite time.

Next, to sketch the proof of surjectivity in Waldhausen's Theorem~\ref{thm:Waldhausen}, we will follow Scott's exposition~\cite{Scott1972}, and use the following proposition, which arranges for a homotopy equivalence to be a homeomorphism on surfaces in $M$.

\begin{proposition}\label{prop:scott-surfaces}
    Let $M$ be irreducible and boundary irreducible. Let $f \colon M \to M$, with $f^{-1}(\partial M) = \partial M$, be a homotopy equivalence. Let $F \subseteq M$ be incompressible and suppose that  either $F$ is non-separating, or, if  $M \sm \nu F$ has two components, suppose that neither of these components  has fundamental group $\pi_1(F)$. Then $f \simeq g$ such that either
    \begin{enumerate}[(a)]
        \item\label{item:prop-scott-surfaces-a} $g| \colon g^{-1}(F) \to F$ is a homeomorphism; or
        \item\label{item:prop-scott-surfaces-b} $g(M) \subseteq \p M$, and $M$ is an interval bundle over $F$, i.e.~$M \cong F \times I$ or $M \cong F \wt{\times} I$.
    \end{enumerate}
\end{proposition}

We will to use this to sketch the proof of surjectivity in \cref{thm:Waldhausen}.  To do this, we prove the following theorem, using \cref{prop:scott-surfaces}.

\begin{theorem}\label{thm:that-implies-wladhausens-thm}
    Let $M$ be irreducible and boundary irreducible. Let $f \colon M \to M$, with $f^{-1}(\partial M) = \partial M$, be a homotopy equivalence.  Then either
    \begin{enumerate}[(a)]
        \item\label{item-waldhausen-thm-a} $f$ is properly homotopic to a homeomorphism $g \colon M \to M$; or
        \item\label{item-waldhausen-thm-b} $f$ is properly homotopic to $g$ with $g(M) \subseteq \p M$, and $M$ is an interval bundle over $F$ for some surface $F$, i.e.~$M \cong F \times I$ or $M \cong F \wt{\times} I$.
    \end{enumerate}
\end{theorem}

The theorem almost implies surjectivity in Waldhausen's theorem. First, because $M$ is an Eilenberg-MacLane space, every automorphism of $\pi_1(M)$ is realised by a homotopy self-equivalence of $M$. If \cref{thm:that-implies-wladhausens-thm}~\eqref{item-waldhausen-thm-a} occurs, then we are done. If \cref{thm:that-implies-wladhausens-thm}~\eqref{item-waldhausen-thm-b} occurs, then more argument is required. The fact that every homotopy self-equivalence of a surface is homotopic to a homeomorphism is a key input. This, together with further arguments of Waldhausen \cite{Waldhausen1968} for the case of interval bundles over surfaces, into the details of which we will not go, allow us to conclude.

Also note that it suffices to find a homeomorphism, by the theorem of Cerf~\cite{Cerf1959} and Hatcher~\cite{Hatcher1983} that every homeomorphism between 3-manifolds is isotopic to a diffeomorphism.
This concludes our sketch of surjectivity in Waldhausen's \cref{thm:Waldhausen}, modulo
a sketch of the proof of \cref{thm:that-implies-wladhausens-thm}.

The idea of the proof of \cref{thm:that-implies-wladhausens-thm} is to repeatedly cut along surfaces, using hierarchies, until we obtain pieces that are sufficiently simple 3-manifolds, namely either 3-balls or interval bundles over surfaces.

\begin{proof}[Sketch proof of \cref{thm:that-implies-wladhausens-thm}]
    Choose a hierarchy for $M$ as in \cref{thm:existence-hierarchies}.  Induct on the length of the hierarchy.
    Let $F \subseteq M$ be irreducible. Apply \cref{prop:scott-surfaces}. If \cref{prop:scott-surfaces}~\eqref{item:prop-scott-surfaces-b} holds, then \cref{thm:that-implies-wladhausens-thm}~\eqref{item-waldhausen-thm-b} holds.

    If \cref{prop:scott-surfaces}~\eqref{item:prop-scott-surfaces-a} holds, then we give more information on the induction. Label one copy of $M$ as $N$ and think of $f$ as a map $f \colon M \to N$.   Cut $N$ along $F$ and cut $M$ along $f^{-1}(F)$. We obtain a map $f' \colon M' \to N'$
between the cut manifolds. Either $f'$ satisfies the hypotheses of the theorem or  $M' = N'$ is a union of copies of $D^3$. The latter is the bases case, dealt with below. Also  $f|_{\p M'}$ is a homeomorphism (apply the fact that every homotopy equivalence of a surface can be approximated by a homeomorphism, to any boundary component where this is not yet the case).
By the inductive hypothesis, and the fact that the length of the hierarchy for $M'$ is less than the length of the hierarchy for $M$, either \eqref{item-waldhausen-thm-a} or \eqref{item-waldhausen-thm-b} hold in the statement of \cref{thm:that-implies-wladhausens-thm}.  But $f|_{\p M'}$ is a homeomorphism, so any $g$ arising from a proper homotopy must also satisfy $g^{-1}(\partial M) \subseteq \partial M$. Thus \eqref{item-waldhausen-thm-b} cannot hold, and so  \eqref{item-waldhausen-thm-a} does hold. Therefore $f'$ is properly homotopic to a homeomorphism. Glue back together to obtain a homeomorphism.

  For the base case, if the hierarchy has length one, then after cutting we  have a union of 3-balls.  But every homotopy equivalence of $D^3$ that is a homeomorphism on the boundary is properly homotopic to a homeomorphism by the Alexander trick -- coning the given homeomorphism of $S^2$ produces an extension to a homeomorphism of~$D^3$, then a straight line homotopy yields a proper homotopy between this homeomorphism and the original homotopy equivalence.
  \end{proof}

\section{Mapping class groups of hyperbolic 3-manifolds}\label{sec-hyperbolic}

In this section, we will consider the case of compact, orientable hyperbolic manifolds.
We will begin by discussing the basic properties of hyperbolic 3-manifolds. We will then discuss Mostow rigidity \cite{Mostow1968} as well as its consequences for mapping class groups. Then we will discuss the work of Gabai \cite{Gabai1997} and Gabai--Meyerhoff--Thurston \cite{GabaiMeyerhoffThurston2003} on completing the proof that \[\pi_0\Diff(M) \cong \pi_0\Isom(M) \cong \Out(\pi_1(M))\] for closed hyperbolic 3-manifolds.
Finally, we will discuss the resolution of the Smale conjecture for closed hyperbolic 3-manifolds due to Gabai \cite{Gabai2001}.

From now on, we will write $\cong_{\isom}$ and $\cong$ when two manifolds are isometric or diffeomorphic respectively. We will write $\sim$ when two maps between manifolds are isotopic.

\subsection{Definition and properties} \label{ss:Hyp-defs}

The following definition is equivalent to \cref{defn-hyperbolic} when $M$ is a 3-manifold.

\begin{definition}
For $n \ge 2$, an $n$-manifold $M$ is \textit{hyperbolic} if its interior admits a complete Riemannian metric of constant curvature $-1$.
\end{definition}

Our prototypical example of a hyperbolic $n$-manifold is real hyperbolic space $\H^n$.
This has multiple definitions (which are equivalent up to isometry) but, for our purposes, it will be convenient to use the \textit{Poincar\'{e} disc model} where $\H^n =
\Int D^n$ is taken to be the open unit ball in $\R^n$ with metric
\[ \frac{dx_1^2 + \cdots + dx_n^2}{(1-(x_1^2 + \cdots + x_n^2))^2}.\]
The following is well known (see, for example, \cite[Theorem 3.32]{BH99}).

\begin{proposition}
    Every simply connected hyperbolic $n$-manifold is isometric to real hyperbolic space $\H^n$.
\end{proposition}

An important isometry invariant of a hyperbolic manifold $M$ is its volume:
\[ \vol(M) := \int_M \omega_g \]
where $\omega_g$ is the volume form associated to a choice of (hyperbolic) Riemannian metric $g$.

If $M$ is a hyperbolic $n$-manifold, then its universal cover is a simply connected hyperbolic $n$-manifold and so is isometric to $\H^n$. This implies that $M \cong \H^n/ G$ for some $G \le \Isom(\H^n)$.

We will now consider the case of compact, orientable hyperbolic 3-manifolds $M$.
Since 3-manifolds with nonempty boundary are Haken, and so the results of \cref{sec-Haken} apply, nothing will be lost from our development if we restrict further to closed, orientable hyperbolic 3-manifolds.

Since $\Isom^{+}(\H^3) \cong \PSL_2(\C)$ where $\PSL_2(\C):= \SL_2(\C)/\pm 1$, every closed orientable hyperbolic 3-manifold is of the form $\H^3/G$ where $G \le \PSL_2(\C)$ is a discrete subgroup which is of finite covolume in the sense that $\vol(\H^3/G)<\infty$ (since $M$ is closed).
Such groups $G$ are known as \textit{Kleinian groups}. Not all Kleinian groups correspond to manifolds since, in general, $\H^3/G$ will be an orbifold rather than a manifold.

\begin{example}
Here are some  examples of hyperbolic 3-manifolds.
\begin{enumerate}
    \item
Exteriors of hyperbolic knots are compact hyperbolic 3-manifolds. Here, a knot $K \subseteq S^3$ is hyperbolic if $S^3 \sm \nu K$ is a hyperbolic 3-manifold, or equivalently if the interior of the exterior, which is homeomorphic to $S^3 \sm K$, admits a hyperbolic structure. Since $S^3 \sm \nu K$ has nonempty boundary, it is  Haken. For instance the figure eight knot $4_1$ is hyperbolic.

\item
Hyperbolic Dehn surgery. This is an operation which can be used to construct further hyperbolic 3-manifolds from existing ones. It can be used to produce infinitely many (non-diffeomorphic) hyperbolic 3-manifolds with $\vol(M) \le V$ for some constant $V > 0$.
Given a hyperbolic knot $K \subseteq S^3$ such that $S^3 \sm K$ is hyperbolic, we can fill $S^3 \sm \nu K$ with a solid torus $S^1 \times D^2$.  A slope in $\Q \cup \{\infty\}$ corresponds to a curve in $\partial(S^3 \sm\ \nu K)$.  The slope of the filling is the curve that is glued to $\{1\} \times \partial D^2$. This operation is called \emph{Dehn filling}.  Thurston showed that there exist at most finitely many so-called \emph{exceptional slopes} for which this construction does not yield a closed hyperbolic 3-manifold (see \cite[Theorem 5.8.2]{Thurston-Levy-book} and the example which follows). For $K=4_1$, the exceptional slopes are $\{\infty, 0, \pm 1, \pm 2, \pm 3, \pm 4\}$.  So one hyperbolic knot yields infinitely many closed, hyperbolic 3-manifolds.

\item More generally, Thurston showed the following more general statement \cite[Theorem 5.8.2]{Thurston-Levy-book}. Let $M$ be a compact, orientable, hyperbolic 3-manifold with $\partial M \cong T^2$. Then there are at most finitely many exceptional slopes in $\Q \cup \{\infty\}$ for which the corresponding Dehn filled 3-manifold is not hyperbolic.

    \item
Arithmetic hyperbolic 3-manifolds. If $\mathcal{O}$ is an order (a subring that is also a lattice) in a quaternion algebra $A$ over a number field $K$ with exactly one complex place, then, subject to certain conditions, one can construct a Kleinian group
$G_{\mathcal O}\le \PSL_2(\C) \cong \Isom^+(\H^3)$
and a corresponding closed, orientable, hyperbolic 3-manifold $M_{\mathcal{O}} = \H^3/G_{\mathcal{O}}$. For more details, see \cite[Chapter 8]{MR03}.

\item \label{item:jorg}
It was shown by J{\o}rgensen \cite[Theorem 1]{Jo77} that there exist compact hyperbolic $3$-manifolds $M$ that fibre over $S^1$.  These are the mapping tori of \emph{pseudo-Anosov} mapping classes of surfaces.
\end{enumerate}
\end{example}

\begin{remark}\leavevmode
\begin{enumerate}[(a)]
    \item  There is no analogue of hyperbolic Dehn surgery in higher dimensions. In fact, for all $n \ge 4$ and $V >0$, there are finitely many hyperbolic $n$-manifolds $M$ with $\vol(M) < V$ (up to diffeomorphism). So, in some sense, there are not that many hyperbolic $n$-manifolds for $n \ge 4$. This is another reason why hyperbolic $3$-manifolds are especially interesting.
\item Following Example (\ref{item:jorg}) above, one can also ask in what other dimensions $n \ge 4$ there exists a hyperbolic $n$-manifold that fibres over $S^1$. Such examples do not exist for $n$ even, since manifolds $M$ which fibre over $S^1$ necessarily have $\chi(M) = 0$, whilst it is an elementary consequence of the Gauss-Bonnet theorem \cite[Theorem 9.3.1]{Ra94} that $\chi(M) \ne 0$ for $M$ an even-dimensional hyperbolic manifold. An example in the case $n=5$ was found by Italiano-Martelli-Migliorini \cite{IMM23} but the case of $n \ge 7$ odd remains open.
\end{enumerate}
\end{remark}

Here are some properties of closed hyperbolic 3-manifolds.
\begin{enumerate}
\item
They are aspherical since $\wt M \cong \H^3$ is contractible. This implies, for example, that $\pi_1(M) = \G$ is torsion-free.
    \item
They are irreducible and atoroidal. Conversely, by Thurston's hyperbolisation theorem, every closed irreducible atoroidal 3-manifold which is not Seifert fibred is hyperbolic~\cite[Theorem 1.7.5]{AFW2015}.
\item
If $M$ is a closed hyperbolic 3-manifold, then $\pi_1(M)$ is a hyperbolic group (in the sense of Gromov). This follows from the fact that $\pi_1(M)$ acts freely by isometries on $\H^3$ and so, by the Milnor-Svarc lemma \cite[Proposition I.8.19]{BH99}, the Cayley graph of $\pi_1(M)$ is quasi-isometric to $\H^3$, from which the result follows. In fact, this is true for fundamental groups of closed hyperbolic $n$-manifolds for all $n \ge 2$.
\item
Every closed hyperbolic $3$-manifold $M$ has a finite cover $M' \to M$ where $M'$ fibres over the circle. This is (a special case of) the virtual fibring conjecture, and was shown by Agol \cite{Ag13}.
\end{enumerate}

Let $M$ be a closed hyperbolic $3$-manifold. Since $M$ is aspherical, we have that $M \simeq K(\pi_1(M),1)$. 
Therefore, by \cref{lemma:haut-equals-out-for-Kpi1}, for a closed hyperbolic 3-manifold $M$ we have that
\[ \pi_0 \hAut(M) \cong \Out(\pi_1(M)). \]

\subsection{Mostow rigidity}

We will now discuss the following famous result, due to Mostow \cite{Mostow1968}.

\begin{theorem}[Mostow rigidity]\label{thm-Mostow rigidity}
Let $M$ and $N$ be closed, orientable, hyperbolic $3$-manifolds. If $f \colon M \to N$ is a homotopy equivalence, then $f$ is homotopic to an isometry.
\end{theorem}

\begin{remark} \label{remark:Mostow}
\leavevmode
\begin{enumerate}[(a)]
\item This was generalised to not necessarily-closed finite volume hyperbolic 3-manifolds by Prasad~\cite{Pr73}.

\item This holds for closed, hyperbolic $n$-manifolds for all $n \ge 3$ but fails for $n=2$ where the surfaces $\Sigma_g$ admit infinitely many inequivalent hyperbolic metrics for any $g \ge 2$. Hyperbolic 3-manifolds therefore occupy a middle ground where the dimension $n=3$ is large enough for Mostow rigidity but small enough that infinitely many hyperbolic 3-manifolds exist with bounded volume.

\item  The Borel conjecture can be viewed as a broad generalisation of Mostow rigidity. Let $M$ and $N$ be closed, aspherical, topological $n$-manifolds. Then the Borel conjecture is that, if $f \colon M \to N$ is a homotopy equivalence, then $f$ is homotopic to a homeomorphism. This is true for $n=3$ by Perelman's resolution to Thurston's geometrisation conjecture~\cite{Perelman:2002-1,Perelman:2003-1,Perelman:2003-2}.
The Borel conjecture is true for $n \geq 5$ for closed, aspherical $n$-manifolds whose fundamental groups satisfy the Farrell-Jones conjecture. If the fundamental group is moreover a good group, then the Borel conjecture holds in dimension four as well (see \cite[Theorem 2.28]{Lu10}).

\end{enumerate}
\end{remark}

In particular, if two closed hyperbolic 3-manifolds are diffeomorphic, then they are isometric. That is, the choice of hyperbolic metric $g$ on a closed hyperbolic 3-manifold $M$ is essentially unique. Moreover, we have the following corollary.

\begin{corollary}
    Let $M$, $N$ be closed, orientable, hyperbolic $3$-manifolds. Then
    \[ M \cong_{\isom} N \,\Leftrightarrow\, M \cong N \,\Leftrightarrow\,  M \simeq N \,\Leftrightarrow\,  \pi_1(M) \cong \pi_1(N).\]
\end{corollary}

We will now discuss the general idea of the proof of Mostow rigidity. More details can be found in \cite{Mostow1968}, as well as in \cite{Ro15,Zh14}.
Let $\overline{\H^3} \cong D^3$ denote the closed ball compactification in the Poincar\'{e} ball model (see, for example, \cite[Section 2.1]{Thurston-Levy-book}), and let $\partial \H^3 := \overline{\H^3} \backslash \H^3$ $(\cong S^2)$.

\begin{proof}[Sketch of proof]
Let $f \colon M \to N$ be a homotopy equivalence between closed hyperbolic $3$-manifolds $M$ and $N$.
Then $f$ induces a homotopy equivalence $F \colon \wt M \to \wt N$.
Since $\wt M \cong \H^3$ and $\wt N \cong \H^3$, we will write this as $F \colon \H^3 \to \H^3$.

The proof now consists of the following three steps:
\begin{enumerate}
    \item Show that there exists a (continuous) extension $\overline{F} \colon \overline{\H^3} \to \overline{\H^3}$ such that $\overline{F} \mid_{\partial \H^3} \colon \partial \H^3 \to \partial \H^3$ is a conformal diffeomorphism (i.e.~an angle-preserving diffeomorphism).
    \item
There is a one-to-one correspondence between isometries of $\H^3$ and conformal diffeomorphisms of $\partial \H^3$ given by
\[ \Isom(\H^3) \xrightarrow[]{\cong} \text{\normalfont ConfDiffeo}(\partial \H^3), \quad G \mapsto \overline{G} \mid_{\partial \H^3}.  \]
    \item
Pick $G \in \Isom(\H^3)$ such that $\overline{F} \mid_{\partial \H^3} = \overline{G} \mid_{\partial \H^3}$. Then show that $G$ induces an isometry $g \colon M \to N$ such that $f \simeq g$.   \qedhere
\end{enumerate}
\end{proof}

We conclude this section by noting that this does not suffice to compute the mapping class group $\pi_0\Diff(M) = \Diff(M)/\sim$.

It is well-known and can be proven using an elementary argument that, for $M$ closed hyperbolic 3-manifold, the set of isometries $\Isom(M)$ is finite and discrete. In particular, we have that $\Isom(M)$ and $\pi_0\Isom(M) = \Isom(M)/\sim$ coincide, as do $\Isom(M)$ and $\Isom(M)/\simeq$.
By combining \cref{thm-Mostow rigidity} with the result established at the end of \cref{ss:Hyp-defs}, we obtain the following diagram:
\[
\begin{tikzcd}
  \overbrace{\Isom(M)/\sim}^{\pi_0\Isom(M)} \ar[r] \ar[d,"\cong"] & \overbrace{\Diff(M)/\sim}^{\pi_0\Diff(M)} \ar[d,twoheadrightarrow] & & \\
  \Isom(M)/\simeq \ar[r,"\cong"] & \Diff(M)/\simeq \ar[r,"\cong"] &  \underbrace{\hAut(M) /\simeq}_{\pi_0 \hAut(M)} \ar[r,"\cong"] & \Out(\pi_1(M))
\end{tikzcd}
\]
where the bottom arrows and the left vertical arrow are all bijections. The missing ingredient is whether homotopic diffeomorphisms $f, g \colon M \to M$ are actually isotopic, i.e.\ injectivity of the right-hand vertical map.

\subsection{The Gabai--Meyerhoff--Thurston theorem}

The aim of this section will be to discuss the following major theorem due to Gabai, Meyerhoff and N.Thurston \cite{GabaiMeyerhoffThurston2003}, building upon a previous result of Gabai \cite{Gabai1997}. In the case where $M$ is Haken, this was proved in the previous chapter.

\begin{theorem}[Gabai--Meyerhoff--Thurston] \label{thm:GMT}
Let $M$ be a closed hyperbolic $3$-manifold. If $f, g \in \Diff(M)$ and $f \simeq g$, then $f \sim g$. It then follows that
\[ \pi_0\Isom(M) \cong \pi_0\Diff(M) \cong \pi_0\hAut(M) \cong \Out(\pi_1(M)). \]
\end{theorem}

\begin{remark} \label{remark:GMT}
\leavevmode
\begin{enumerate}[(a)]
    \item
    The main result in \cite{GabaiMeyerhoffThurston2003} actually proved more. Let $M$ be a closed 3-manifold and let $N$ be a closed hyperbolic 3-manifold. If $f \colon M \to N$ is a homotopy equivalence, then $f$ is homotopic to an isometry. (In particular, $M$ is a hyperbolic 3-manifold.) This is stronger than Mostow rigidity because here we do not need to know a priori that $M$ is hyperbolic.
\item As mentioned in the previous section, $\pi_0\Isom(M)$ is finite and so the above theorem shows that mapping class groups of closed hyperbolic 3-manifolds are finite.
\end{enumerate}
\end{remark}

\begin{theorem}[\cite{Ko88}]
    Every finite group arises as the mapping class group of a closed hyperbolic 3-manifold.
\end{theorem}

The remainder of this section will be dedicated to a discussion of the proof of this theorem. We will begin with the work of Gabai \cite{Gabai1997} which proved the theorem subject to a certain technical condition which we will now discuss.

\begin{definition}
Let $M$ be a closed hyperbolic $3$-manifold and let $\delta$ be a simple closed geodesic in $M$. Then $\delta$ lifts to a set of hyperbolic lines $\{\delta_i\}$ in $\H^3$. For each $i \ne j$, define the \textit{orthocurve} $s_{ij}$ to be the shortest hyperbolic line segment in $\H^3$ between $\delta_i$ and $\delta_j$. Define the \textit{midplane} $D_{ij}$ to be the hyperbolic plane in $\H^3$ which meets $s_{ij}$ orthogonally at its midpoint.

By considering the completion $\overline{\H^3}$ and corresponding boundary $\partial \H^3$ ($\cong S^2$), we have a curve $\lambda_{ij} = \partial D_{ij}$ which separates the pair of points $\partial \delta_i$ from the pair of points $\partial \delta_j$.
The set of simple closed curves $\{ \lambda_{ij} \mid i \ne j\}$ in $\partial \H^3$ is known as the \textit{Dirichlet insulator family} associated to the geodesic $\delta$.

We say that a Dirichlet insulator family $\{\lambda_{ij} \mid i \ne j\}$ is \textit{non-coalescable} if there does not exist $i, j_1,j_2,j_3$ such that $\lambda_{ij_1} \cup \lambda_{ij_2} \cup \lambda_{ij_3}$ separates the two points $\partial \delta_i$. See \cref{fig:DIT} for an example where this condition fails.
\end{definition}

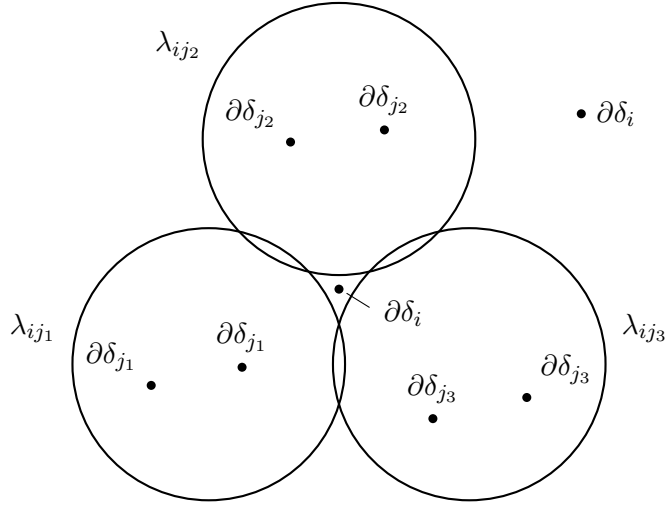
\begin{figure}[h]
    \centering
\begin{tikzpicture}[scale=0.8, line cap=round, line join=round]
  \usetikzlibrary{calc}

  \def\R{2.25}
  \def\d{4.3} 

  \pgfmathsetmacro{\h}{0.866025403784*\d} 

  \coordinate (A) at (0,0);
  \coordinate (B) at (\d,0);
  \coordinate (C) at (\d/2,\h);

  \coordinate (shift) at ($(-\d/2,-\h/3)$);
  \coordinate (L) at ($(A)+(shift)$);
  \coordinate (R) at ($(B)+(shift)$);
  \coordinate (T) at ($(C)+(shift)$);

  \draw[thick] (L) circle (\R);
  \draw[thick] (R) circle (\R);
  \draw[thick] (T) circle (\R);

  \draw (0.13,-0.07) -- (0.5,-0.3);

  \node[left]  at ($(L)+(-\R-0.1,0.6)$) {$\lambda_{ij_1}$};
  \node[right] at ($(R)+(\R+0.1,0.6)$)  {$\lambda_{ij_3}$};
  \node[left]  at ($(T)+(-2.1,\R-0.7)$) {$\lambda_{ij_2}$};

  \tikzset{dot/.style={circle, fill, inner sep=1.2pt}}

  \node[dot, label=above left:{$\partial\delta_{j_1}$}] at ($(L)+(-0.95,-0.35)$) {};
  \node[dot, label=above:{$\partial\delta_{j_1}$}]      at ($(L)+( 0.55,-0.05)$) {};

  \node[dot, label=above left:{$\partial\delta_{j_2}$}] at ($(T)+(-0.80,-0.05)$) {};
  \node[dot, label=above:{$\partial\delta_{j_2}$}]      at ($(T)+( 0.75, 0.15)$) {};

  \node[dot, label=above:{$\partial\delta_{j_3}$}]       at ($(R)+(-0.60, -0.9)$) {};
  \node[dot, label=above right:{$\partial\delta_{j_3}$}] at ($(R)+( 0.95,-0.55)$) {};

  \node[dot] at (0,0) {};
  \node[label=right:{$\partial\delta_i$}] at (0.4,-0.4) {};

  \node[dot, label=right:{$\partial\delta_i$}] at (4.0,2.9) {};
\end{tikzpicture}
    \caption{Example of a Dirichlet insulator family which is not non-coalescable: the four pairs of points are $\partial \delta_i$, $\partial \delta_{j_1}$, $\partial \delta_{j_2}$, $\partial \delta_{j_3}$, the circles are $\lambda_{ij_1}$, $\lambda_{ij_2}$, $\lambda_{ij_3}$ and the diagram is drawn on a patch of the boundary sphere $\partial \H^3 \cong S^2$.}
    \label{fig:DIT}
\end{figure}

Gabai defined a more general notion of \emph{insulator family}, which is a collection of simple closed curves $\{\lambda_{ij} \mid i \ne j\}$ that must satisfy certain criteria but need not arise from the construction involving orthocurves and midplanes described above \cite[Definition 0.4]{Gabai1997}. The following is equivalent to \cite[Theorem 0.10]{Gabai1997}.

\begin{theorem}[Gabai] \label{thm:Gabai 97 main thm 0.9}
Let $M$ be a closed hyperbolic $3$-manifold. Suppose there exists a simple closed geodesic $\delta$ in $M$ which has an associated non-coalescable insulator family. Then, if $f, g \in \Diff(M)$ and $f \simeq g$, then $f \sim g$.
\end{theorem}

Since this conclusion is known already in the Haken case, it suffices to treat the case where $M$ is a closed non-Haken hyperbolic 3-manifold.
By replacing $f$ with $f \circ g^{-1}$ we can assume that $f \simeq \id$ and we aim to show that $f \sim \id$. The strategy of the proof is to use a 1981 theorem attributed to Siebenmann (not published by Siebenmann, but proven in the article as \cite[Proposition 2.11]{Gabai1997}), which implies that $f \sim \id$ provided $f$ is isotopic to a map which fixes the geodesic $\delta$ pointwise.
Such an isotopy is built using the insulator family. This construction is beyond the scope of this survey, but further details can be found in the outline given in \cite[pp.~38--40]{Gabai1997}.

Next, Gabai observes that, if $M$ contains a hyperbolic tube of radius $\log(3)/2$ ($\approx 0.5$) about a simple closed geodesic $\delta$, then its associated Dirichlet insulator family is non-coalescable \cite[Lemma 5.9]{Gabai1997}.
Hence, in order to prove Theorem \ref{thm:GMT}, it remains to prove the result in the case where $M$ is a closed hyperbolic $3$-manifold such that no simple closed geodesic is contained in a hyperbolic tube of radius $\log(3)/2$.
This was achieved by Gabai--Meyerhoff--Thurston by showing that \cref{thm:Gabai 97 main thm 0.9} actually applies to all closed hyperbolic $3$-manifolds. Their proof involved the introduction of a new insulator family associated to certain geodesics, known as the \emph{Corona insulator family} (for a definition, see \cite[Section 2]{GabaiMeyerhoffThurston2003}), and they gave a criteria for this family to be non-coalescable.

Their main result is as follows. This implies \cref{thm:GMT} by the discussion above.

\begin{theorem}[Gabai--Meyerhoff--Thurston] \label{thm:CMT-lemma}
Let $M$ be a closed hyperbolic $3$-manifold such that no simple closed geodesic is contained in a hyperbolic tube of radius $\log(3)/2$. Then:
\begin{enumerate}[\normalfont (a)]
    \item
$M$ is contained in one of seven families $\mathcal{R}_0, \cdots, \mathcal{R}_6$ {\normalfont(}which can be viewed as subsets of $\C^2${\normalfont)}.
    \item
For each $M \in \mathcal{R}_i$ for $i=1, \cdots, 6$, the shortest geodesic $\delta$ has an associated \textit{Corona insulator family} $\{\kappa_{ij}\}$ which is non-coalescable.
    \item
The set $\mathcal{R}_0$ consists of a single manifold $\text{\normalfont Vol} 3$, and $\delta$ in $\text{\normalfont Vol} 3$ has a non-coalescable insulator family.
\end{enumerate}
\end{theorem}

The proof of all three parts of the proof of \cref{thm:CMT-lemma} is carried out with the assistance of a rigorous computer program. We will now briefly explain the basic approach taken. Let $\delta$ be a shortest geodesic in the closed hyperbolic $3$-manifold $M$ and consider the
$2$-generator subgroup $G = \langle f, w \rangle$ of $\pi_1(M) \le \Isom^{+}(\H^3)$ where $f$ is a primitive hyperbolic isometry whose fixed axis $\delta_0\subseteq \H^3$
projects to $\delta$, and $w$ is a hyperbolic isometry which
takes $\delta_0$ to a nearest translate.
By hypothesis, $\delta$ does not have a $\ln(3)/2$ tube, and it follows that the shortest geodesic in $Z=\H^3/G$ (which corresponds to $\delta$) also does not have a $\ln(3)/2$ tube.
Such a $2$-generator subgroup is a \emph{relevant} subgroup of $\Isom^{+}(\H^3)$ (this has a precise definition given in \cite[Definition 1.12]{GabaiMeyerhoffThurston2003}) and the space of all relevant subgroups is naturally parametrised as a subset $P$ of $\C^2$.
Such a parameter space is amenable to a computer analysis; for example, the criteria for the Corona insulator family associated to $\delta$ being non-coalescable can be expressed in terms of estimates on the points in the corresponding regions of $\C^2$ (see, for example, \cite[Proposition 2.8]{GabaiMeyerhoffThurston2003}). For further details on the outline of the proof, see \cite[p336-338]{GabaiMeyerhoffThurston2003}.

\subsection{The generalised Smale conjecture}

The following was shown by Gabai \cite{Gabai2001}, building upon the Gabai--Meyerhoff--Thurston theorem. This resolves the generalised Smale conjecture (\cref{conj - generalised smale}) for closed hyperbolic 3-manifolds.

\begin{theorem}[Gabai]\label{thm-strong Smale for hyperbolic}
Let $M$ be a closed hyperbolic $3$-manifold. Then the inclusion map
\[ \Isom(M) \hookrightarrow \Diff(M) \]
is a homotopy equivalence.
\end{theorem}

The idea of the proof is as follows. Since $\Isom(M)$ is finite, we have that $\Isom_0(M) = \{\ast\}$. Since $\pi_0\Isom(M) \cong \pi_0\Diff(M)$ by Theorem \ref{thm:GMT}, it therefore suffices to prove that $\Diff_0(M) \simeq \{\ast\}$ (which, in this case, is equivalent to proving the weak Smale conjecture, \cref{thm-weak Smale}).
By Mostow rigidity, $\Diff_0(M)$ can be equivalently viewed as the space of hyperbolic metrics $\text{\normalfont Hyp}(M)$ on $M$. Gabai shows that this space is contractible using a generalisation of the methods used in \cite{Gabai1997} (in particular, using the theory of insulators to reduce to the case of Haken 3-manifolds).

\section{Mapping class groups of Seifert fibred 3-manifolds}\label{sec-Seifert fibred}
Here we collect the results on mapping class groups of Seifert fibred manifolds. We restrict to compact, orientable \(3\)-manifolds.

\subsection{Definition of Seifert fibred 3-manifolds}\label{subsec:defn-of-SFS}
First, we discuss the definition, following~\cite{NeumannLectureNotes}.
Denote the half-disc by \[D^2_+ := \{x = (x_1,x_2) \in \R^2 \mid \|x\| \leq 1, x_2 \geq 0\}.\]

\begin{definition}\label{defn:SFS-defn}
A \emph{Seifert fibration} on a compact, orientable $3$-manifold $M$ is a map $\pi \colon M \to B$, where  $B$ is a compact surface, possibly nonorientable, such that the following holds.
\begin{enumerate}
    \item For every $x$ in the interior of $B$
there is a $D^2$ neighbourhood of $x$ and a diffeomorphism $\pi^{-1}(D^2) \cong D^2 \times S^1$, with respect to which the restriction of $\pi$ to $\pi^{-1}(D^2)$ is given by
\begin{align*}
\pi \colon D^2 \times S^1 &\to D^2 \\
(rt_1,t_2) & \mapsto rt_1^at_2^b,
\end{align*}
where $t_i \in S^1$, $r \in [0,1]$, and $a,b \in \Z$ with $a \neq 0$ and $\gcd(a,b)=1$.
\item For every $x \in \partial B$, there is a half-disc neighbourhood $D^2_+$ of $x$ and a diffeomorphism $\pi^{-1}(D^2_+) \cong D^2_+ \times S^1$, with respect to which the restriction of $\pi$ to $\pi^{-1}(D^2)$ is given by
\begin{align*}
\pi \colon D^2_+ \times S^1 &\to D^2_+ \\
(rt_1,t_2) & \mapsto rt_1,
\end{align*}
where $t_i \in S^1$ and $r \in [0,1]$ such that $rt_1 \in D^2_+$.
\end{enumerate}
\end{definition}

In the case that $x \in \mathring{B}$, the centre of the disc $0 \in D^2$ has inverse image the core circle of $D^2 \times S^1$, and the other points in $D^2 \sm \{0\}$ have inverse images circles that wrap $a$ times around the core in the longitudinal direction and and $-b$ times in the meridional direction.

\begin{definition}
We set up some more terminology.
\begin{enumerate}
    \item     A compact, orientable \(3\)-manifold \(M\) is \emph{Seifert fibred} if it admits a Seifert fibration.
    \item The inverse image of a point $x \in B$ is called a \emph{fibre}.
    \item  The \emph{multiplicity} of an interior fibre is the number $a$ associated to that fibre in the local model from \cref{defn:SFS-defn}. The multiplicity of a boundary fibre is $1$.
    \item  A fibre is \emph{exceptional} or \emph{singular} if $p \neq \pm 1$, and is \emph{regular} otherwise.
\item       Let $g \in \Z$ be the genus of $B$ for $B$ an orientable surface, and let $ g:= -\beta_1(B)$ for  $B$ nonorientable.
\item  Let $c \in \Z_{\geq 0}$ be the number of boundary components of $B$.
\end{enumerate}
\end{definition}

There are finitely many singular fibres, each of which is the fibre of a point in the interior of $B$.  Each circle in $\partial B$ gives rise to a torus boundary component of $M$.
Next we describe an alternative way to think about Seifert fibred 3-manifolds using standard fibred tori.

\begin{definition}\label{def-standard fibred torus}
Let \(a\) and \(b\) coprime integers with $0 \leq b < a$. Define an automorphism \[f_{a,b} \colon D^2 \rightarrow D^2 ; \ z \mapsto e^{2\pi i b/a}z\] which is a rotation of the disc by an angle of \(2\pi b/a\).
Consider the mapping torus $f_{a,b}$.  For each $p\in D^2$ we take $\{p\}\times I$ and identify $\{p\}\times \{0\}$ with $\{f_{a,b}(p)\}\times \{1\}$.
Under this identification the union of intervals
\[
\cup_{k=0}^{a-1}\{(f_{a,b})^k(p)\}\times I
\]
is a simple closed curve in the mapping torus.
This collection of simple closed curves endows the mapping torus of $f_{a,b}$, which is diffeomorphic to $S^1\times D^2$, with a decomposition into disjoint copies of $S^1$.
A \emph{standard fibred torus} is the mapping torus of $f_{a,b}$, together with the induced decomposition into copies of $S^1$, for some \(a\) and \(b\).
\end{definition}

\cref{def-standard fibred torus} provides another way to visualise the local model for a Seifert fibred 3-manifold from \cref{defn:SFS-defn}.
The following is a direct consequence of \cref{defn:SFS-defn}.

\begin{proposition}
    Every closed Seifert fibred 3-manifold admits a decomposition into disjoint copies of \(S^1\), such that each \(S^1\) admits a neighbourhood that is a union of copies of $S^1$, homeomorphic as a decomposed space to a standard fibred torus $S^1\times D^2$, with decomposition as above for some coprime integers $a>0$ and $b$.
\end{proposition}

The space of fibres in a Seifert fibred 3-manifold form a \(2\)-dimensional orbifold. Thus, by the map which collapses each fibre to a point, we can think of a Seifert fibred \(3\)-manifold as a circle bundle over this base orbifold.

We only consider orientable, compact Seifert fibred spaces, so there will be a finite number of exceptional fibres and the only singularities of the base orbifold will be cone points. Indeed, the exceptional fibres correspond to the cone points on the orbifold.

\begin{example}
The 3-sphere $S^3$ is Seifert fibred via the Hopf map. More generally, every lens space is Seifert fibred via its decomposition into two tori (\cref{thm: UniqueTorus}); see \cite[Sec.~4]{NeumannLectureNotes} for details.

In fact six of eight of Thurston's geometries (excluding hyperbolic and Sol) correspond to Seifert fibred geometric 3-manifolds. See~\cref{sec: JSJ}, in particular \cref{prop-SFgeometry}, for more details.
\end{example}

\subsection{Encoding Seifert fibrations by Seifert symbols}\label{subsec:Seifert-symbols}
We describe how to encode Seifert fibrations into the data of a \emph{Seifert symbol}:
\[(g,c; (a_1,b_1), \dots, (a_n,b_n))\] with \(g\in\Z\), $a_i \in \Z_{>0}$,
and \(b_i\in \Z\), with \(\gcd(a_i,b_i)=1\) for all \(i\), we can build a Seifert fibred \(3\)-manifold.

The construction of a 3-manifold from the Seifert symbol proceeds as follows. Take a surface $\Sigma$, either $\Sigma_{g,c}$ for $g \geq 0$, or the nonorientable surface $F_{g,c}$ for $g<0$.  Remove $n$ open discs from the interior, to obtain $\Sigma'$. Form the unique circle bundle over $\Sigma'$ with orientable total space, and consider this as a Seifert fibred 3-manifold.
In the $i$th new toroidal boundary component just created, $S^1_i \times S^1$, for $i = 1,\dots,n$, glue a solid torus $T_i \cong D^2 \times S^1$, with the identification determined by $(a_i,b_i)$. We describe this identification in the case of $\Sigma$ orientable, and refer to \cite[pp.13--4]{NeumannLectureNotes} in the nonorientable case.

In the orientable case, we have a trivial circle bundle $\Sigma' \times S^1$. Let
\begin{align*}
R &:= \Sigma' \times \{1\};    \\
Q_i &:= R \cap(S^1_i \times S^1) = S^1_i \times \{1\}; \text{ and} \\
H_i &:= \{1\} \times S^1 \subseteq S^1_i \times S^1.
\end{align*}
Glue $T_i$ to $S^1_i \times S^1$ in such a way that the meridian $\mu_i = S^1 \times \{1\} \subseteq T_i$ is glued to a simple closed curve in $S^1_i \times S^1$ representing the homology class of $a_iQ_i + b_i H_i$.
This completes the construction of a Seifert fibred 3-manifold corresponding to the data $(g,c; (a_1,b_1), \dots, (a_n,b_n))$.

\begin{theorem}[{\cite[Theorems~1.5~and~1.8]{NeumannLectureNotes}}]
 For any  Seifert fibred space \(M\), there is a fibre-preserving diffeomorphism between \(M\) and the space determined by \((g,c; (a_1,b_1), \dots, (a_n,b_n))\) for some \(g,c,a_i,b_i\).
Thus we can write \[M= M(g,c; (a_1,b_1), \dots, (a_n,b_n)).\]
\end{theorem}

The Seifert symbol is uniquely determined by $(M,B,\pi)$, up to a collection of operations one can perform on the Seifert symbol data. 
See \cite[Theorems~\(1.5\)~and~\(1.8\)]{NeumannLectureNotes} for more details on these operations.
In many cases, the class of the Seifert symbol, modulo to the same operations, is determined by the Seifert fibred 3-manifold $M$, without first fixing $(\pi,B)$. The exceptions to this are listed in \cref{thm-fibre preserving equiv to diff} below.

\begin{example}
    The orientable circle bundle over the closed, orientable surface $\Sigma_g$ of genus $g$, with Euler number $e$, can be constructed using the Seifert symbol $(g,0;(1,e))$.
\end{example}

\subsection{Which Seifert fibred 3-manifolds are Haken?}\label{subsec:SFS-vs-Haken}

By Jaco \cite[Lemma VI.7.]{Jaco1980} the only reducible oriented Seifert fibred manifolds are $\S^1\times S^2$ and $\mathbb{RP}^3\# \mathbb{RP}^3$. Jaco also clarifies that the overlap between Haken and Seifert fibred manifolds is large, in the following theorem.

\begin{theorem}[{\cite[p.~\(96\)]{Jaco1980}}]\label{thm-SF Haken Classification}
        With the exception of lens spaces, $S^3$, \(S^1\times S^2\), and \(\mathbb{RP}^3\# \mathbb{RP}^3\),
    an oriented Seifert fibred manifold $M$ is either Haken or has base \(S^2\) and exactly \(3\) exceptional fibres;
    in this last case \(M\) is Haken if and only if \(H_1(M)\) is infinite.
\end{theorem}

We will not discuss the cases of lens spaces, $S^3$, nor \(S^1\times S^2\) as their mapping class groups have been covered in Sections \ref{section-S^3}, \ref{sec-S^1xS^2}, and \ref{sec-lens spaces} respectively.

We will survey the developments in the calculation of these mapping class groups, considering two separate classes of Seifert fibred \(3\)-manifolds: Haken 3-manifolds, and 3-manifolds with base \(S^2\) and exactly \(3\) exceptional fibres.
We give an account of the calculation of the mapping class group and the status of the generalised Smale conjecture in both cases. First we introduce the subgroup of fibre preserving diffeomorphisms, which can be useful for computations.

\subsection{The subgroup of fibre preserving diffeomorphisms}\label{subsec:fibre-pres-diffeos}

To discuss the mapping class group of Seifert fibred \(3\)-manifolds, we first consider a useful subgroup of the diffeomorphism group, comprising those diffeomorphisms that respect the fibre structure.

Fix a Seifert fibring \(\varphi\) of \(M\), and denote by \(\Diff^f(M, \varphi)\) the subgroup of \(\Diff(M)\) consisting of those diffeomorphisms which send fibres of \(\varphi\) diffeomorphically to fibres. We call such diffeomorphisms \textit{fibre-preserving}. If \(\varphi\) and \(\varphi'\) are two Seifert fibrings of \(M\) such that there exists a diffeomorphism \(M \to M\) sending the fibres of \(\varphi\) to the fibres of \(\varphi'\), then  \(\Diff^f(M, \varphi)\) and  \(\Diff^f(M, \varphi')\) are conjugate subgroups in \(\Diff(M)\).

By work of Connor-Raymond \cite{ConnerRaymond77} (building on results by Waldhausen in the Haken case) every homotopy equivalence \(M\to M\) is homotopic to a fibre-preserving diffeomorphism of \(M\).

For the rest of the section, we will abuse notation and drop the \(\varphi\) and write \(\Diff^f(M)\) for fibre-preserving diffeomorphisms of \(M\), implicitly choosing a preferred Seifert fibring of \(M\).
This notion of fibre-preserving diffeomorphism is helpful for the  classification of Seifert fibred \(3\)-manifolds, and for the computation of their mapping class groups.

The following result tells us that, in most cases, the classification of Seifert fibred 3-manifolds up to fibre-preserving diffeomorphism is in fact a classification up to arbitrary diffeomorphism.

\begin{theorem}[Waldhausen \cite{Waldhausen1967a, Waldhausen1967b}, Jaco {\cite[Theorem VI.17]{Jaco1980}}]\label{thm-fibre preserving equiv to diff}
    Let \(M_0\) and \(M_1\) be closed, orientable Seifert fibred spaces which are not
    \begin{itemize}
            \item lens spaces,
            \item the 3-torus \(T^3 = S^1 \times S^1 \times S^1\),
            \item \(M(0; (2,1), (2,1), (2,-1), (2,-1))\), the double of the twisted \(I\)-bundle over the Klein bottle, or
            \item \(M(0; (2,b_1), (2,b_2), (a_3,b_3))\).
    \end{itemize}
    Then every diffeomorphism \(M_0\to M_1\) is isotopic to a fibre-preserving diffeomorphism.
\end{theorem}

We know the mapping class groups of lens spaces by \cref{sec-lens spaces}. The 3-torus is Haken, and so $\pi_0 \Diff(T^3) \cong \Out(\Z^3) \cong \GL(3,\Z)$.  The double of the twisted $I$-bundle over the Klein bottle is also Haken.

On the level of mapping class groups, the inclusion map \(\Diff^f(M) \hookrightarrow \Diff(M)\) induces a homomorphism \(\pi_0\Diff^f(M) \to \pi_0\Diff(M)\), which a priori does not need to be injective nor surjective. The above result shows, however, that in most cases it is in fact surjective.
For injectivity, we have the following result of Hong--Kalliongis--McCullough--Rubinstein.

\begin{theorem}[{\cite[Theorem~\(3.9.1\)]{HKMR2012}}]
 Let $M$ be a closed, orientable, Haken Seifert fibred 3-manifold. The inclusion map   \(\Diff^f_0(M) \to \Diff_0(M)\) is a homotopy equivalence, and hence the map \(\pi_0 \Diff^f(M) \to \pi_0 \Diff(M)\) is injective.
\end{theorem}

We will see next that being able to restrict to fibre-preserving diffeomorphisms can be very useful for computations.

\subsection{Haken Seifert fibred 3-manifolds}\label{subsec:Haken-SFS}

In this subsection we let \(M\) be closed, orientable, Seifert fibred, irreducible, and Haken.
The Seifert fibred 3-manifolds $M$ with nonempty boundary are also Haken, but we do not consider these here. We also restrict this discussion to the case that $B$ is orientable.

Then \cref{lemma:kpi1} applies and \(M\simeq K(\pi_1(M),1) \) with \(\pi_1(M)\) infinite. Recall that by \cref{thm:Waldhausen}, for Haken \(3\)-manifolds \(\pi_0\Diff(M) \cong\Out(\pi_1(M))\). 
We can compute the mapping class group using the following strategy, provided $M$ is not one of the Haken 3-manifolds listed as exceptional in \cref{thm-fibre preserving equiv to diff}. For, under this assumption, the mapping class group of fibre-preserving diffeomorphisms is isomorphic to the mapping class group, so we can calculate the mapping class groups/outer automorphism groups by considering the base orbifold, following \cite{ConnerRaymond77}.

Since \(\pi_1(M)\) is infinite, we get the short exact sequence
\[\begin{tikzcd}
	1 & {\pi_1(S^1)} & {\pi_1(M)} & {\pi_1^{\mathrm{orb}}(B)} & 1
	\arrow[from=1-1, to=1-2]
	\arrow[from=1-2, to=1-3]
	\arrow[from=1-3, to=1-4]
	\arrow[from=1-4, to=1-5]
\end{tikzcd}\] where \(B\) is the base orbifold and \(\pi_1^{\mathrm{orb}}(B)\) is the orbifold fundamental group (see e.g.\ \cite{Davis-orbifold-lectures}).
Since \(M\) is Haken, \(\pi_1(M)\) is a central extension of \(\pi_1^{\mathrm{orb}}(B)\) by \(\Z\).
This means that there is a corresponding extension class \(a\in H^2(\pi_1^{\mathrm{orb}}(B);\Z)\) which determines the above short exact sequence.

The idea is to use this extension to construct a corresponding short exact sequence containing \(\Out(\pi_1(M))\). Indeed, let \(\Out(a): = \Aut(a) / \Inn(\pi_1^{\mathrm{orb}}(B))\), where \(\Aut(a)\) is the subgroup of \(\Aut(\pi_1^{\mathrm{orb}}(B))\) given by \[\Aut(a) := \{\varphi \in \Aut(\pi_1^{\mathrm{orb}}(B)) \mid \exists \, \varepsilon \in  \{\pm 1\}  \text{ such that } \varphi^*(a) = \varepsilon a \}.\]
We get the following result from \cite[Corollary~\(2\), Subsection~\(6.13\)]{ConnerRaymond77}, which means we can calculate \(\Out(\pi_1(M))\) using the base orbifold \(B\).

\begin{theorem}\label{thm:computing-Out-pi1-Haen-SFS}
    Suppose that the class \(a\in H^2(\pi_1^{\mathrm{orb}}(B);\Z)\) is not \(2\)-torsion, then we have an exact sequence
    \[\begin{tikzcd}
	K & {H^1(\pi_1^{\mathrm{orb}}(B);\Z)} & {\Out(\pi_1(M))} & {\Out(a)} & 1.
	\arrow[from=1-1, to=1-2]
	\arrow[from=1-2, to=1-3]
	\arrow[from=1-3, to=1-4]
	\arrow[from=1-4, to=1-5]
\end{tikzcd}\]
Moreover, if $\pi_1^{\mathrm{orb}}(B)$ is centreless, or $a=0$, then $K=1$, and we have a short exact sequence.
\end{theorem}

This gives us a way to calculate mapping class groups of Haken Seifert fibred \(3\)-manifolds.
Piwek~\cite[Proposition~2.15]{Pivek} showed that if $B$ is closed, orientable, and has genus at least two, or genus one and at least one cone point, or genus zero and at least four cone points, then the centre of $\pi_1^{\mathrm{orb}}(B)$  is trivial, and so $K=1$. 

We give two examples of \cref{thm:computing-Out-pi1-Haen-SFS} in action.

\begin{example}
    Let \(M\) be the orientable circle bundle over \(T^2\) with Euler number one, which is the Heisenberg 3-manifold. In this case, the base orbifold is again a manifold and \(\pi_1(T^2)\cong \Z^2\). Since $T^2=K(\Z^2,1)$, we have that \(H^n(\pi_1(T^2);\Z) \cong H^n(T^2;\Z)\) for all \(n\). Since the Euler number of the bundle is one, we get that \(a=1\in H^2(T^2;\Z)\cong \Z\). The automorphism group is \(\Aut(\pi_1(T^2)) \cong \GL{(2,\Z)}\), and the inner automorphism group is trivial since \(\Z^2\) is abelian. Hence \[\Out(a) =\Aut(a) =\{\varphi \in \GL{(2,\Z)} \mid \varphi^*(1) = \pm 1 \in H^2(\Z^2;\Z) \} \cong \GL(2,\Z). \] 
    Thus, applying the exact sequence in \cref{thm:computing-Out-pi1-Haen-SFS}, and \cref{thm:Waldhausen} to equate $\pi_0 \Diff(M) \cong \Out(\pi_1(M))$, we get the following exact sequence. 
    \[\begin{tikzcd}
    	K & {\Z^2} & {\pi_0\Diff(M)} & {\GL(2,\Z)} & 1.
    	\arrow[from=1-1, to=1-2]
    	\arrow[from=1-2, to=1-3]
    	\arrow[from=1-3, to=1-4]
    	\arrow[from=1-4, to=1-5]
    \end{tikzcd}\]
We need to determine the map $K \to \Z^2$. 
The fundamental group of  $M$ is \[\pi_1(M) \cong \langle \alpha,\beta,t \mid [\alpha,t],[\beta,t],[\alpha,\beta]=t\rangle.\]
The inner automorphism that conjugates by $\alpha$ fixes $\alpha$ and $t$, and sends $\beta$ to $t\beta$. Similarly the inner automorphism that conjugates by $\beta$ fixes $\beta$ and $t$ and sends $\alpha$ to $t^{-1}\alpha$.
These automorphisms are exactly the image of $(1,0)$ and $(0,-1)$ in $\Z^2$ in $\Aut(\pi_1(M)$, according to \cite[Proof of Cor.~3.1]{Chen-Tshishiku-MCG-circle-bundles}.  Since these are both inner automorphisms, this shows that the map $\Z^2 \to \pi_0\Diff(M)$ is trivial, and hence $K \to \Z^2$ is surjective.
We deduce that $\pi_0 \Diff(M) \cong \GL(2,\Z)$. 
\end{example}

\begin{example}
  Let \(M\) be the orientable circle bundle over the closed, orientable surface $\Sigma_g$ of genus $g \geq 2$. As mentioned above, in this case $K=1$. Also the base orbifold is the manifold $\Sigma_g$, and similarly to the previous example, $\Out(a) = \Out(\pi_1(\Sigma_g))$. So we have a short exact sequence
   \[\begin{tikzcd}
	1 & {H^1(\Sigma_g;\Z)} & {\Out(\pi_1(M))} & {\Out(\pi_1(\Sigma_g))} & 1.
	\arrow[from=1-1, to=1-2]
	\arrow[from=1-2, to=1-3]
	\arrow[from=1-3, to=1-4]
	\arrow[from=1-4, to=1-5]
\end{tikzcd}\] 
We have that $\Out(\pi_1(\Sigma_g)) \cong \pi_0 \Diff(\Sigma_g)$ by the Dehn--Nielsen--Baer theorem, so we obtain a short exact sequence
  \[\begin{tikzcd}
	1 & \Z^{2g} & {\pi_0\Diff(M)} & \pi_0 \Diff(\Sigma_g) & 1.
	\arrow[from=1-1, to=1-2]
	\arrow[from=1-2, to=1-3]
	\arrow[from=1-3, to=1-4]
	\arrow[from=1-4, to=1-5]
\end{tikzcd}\] 
Chen--Tshishiku~\cite{Chen-Tshishiku-MCG-circle-bundles} investigated when this sequence splits.
\end{example}

Another useful short exact sequence in the case of Haken Seifert fibred \(3\)-manifolds which do not have base \(S^2\) and \(3\) exceptional fibres, is due to Johannson \cite[Proposition~\(25.3\)]{Johannson1979b}.

\subsection{The weak generalised Smale conjecture}
For Haken Seifert fibred \(3\)-manifolds, the weak generalised Smale conjecture was proven independently by Hatcher \cite{Hatcher1976} and Ivanov \cite{Ivanov1976}. Furthermore, the homotopy type of \(\Diff_0\) is known.

\begin{proposition}[\cite{Hatcher1976, Ivanov1976}]
    Let \(M\) be a closed, orientable, geometric, Haken, Seifert fibred 3-manifold. Then \[\Diff_0(M) \simeq \prod_b \text{SO}(2)\] such that \(b\) is the rank of the centre \(Z(\pi_1(M))\) of $\pi_1(M)$.
\end{proposition}

The full generalised Smale conjecture (\cref{conj - generalised smale}) does not hold for all Haken Seifert fibred \(3\)-manifolds; for instance, as mentioned in the introduction, it is false for $S^1 \times S^2$ and for $T^3$.

\subsection{Base orbifold \(S^2\) with exactly three exceptional fibres}\label{subsec:SFS-non-Haken-case}

We have the following proposition.

\begin{proposition}
Let \(M\) be Seifert fibred with base \(S^2\) and exactly three exceptional fibres.
    Exactly one of the following holds:
    \begin{enumerate}
        \item $M$ is Haken, so $\pi_0 \Diff(M) \cong \Out(\pi_1(M))$;
        \item $M$ is geometric with $\mathbb{H}^2 \times \R$ geometry, and the standard maps induce isomorphisms
       \[ \pi_0\Isom(M)\xrightarrow{\cong} \pi_0\Diff^f(M) \xrightarrow{\cong} \pi_0\Diff(M) \xrightarrow{\rho,\cong} \Out(\pi_1(M));\]
        \item $M$ is geometric with $\SL(2,\mathbb{R})$ geometry, and the standard maps induce isomorphisms
       \[ \pi_0\Isom(M)\xrightarrow{\cong} \pi_0\Diff^f(M) \xrightarrow{\cong} \pi_0\Diff(M) \xrightarrow{\rho,\cong} \Out(\pi_1(M));\]
        \item $M$ is geometric with elliptic geometry, and the standard maps induce isomorphisms
       \[ \pi_0\Isom(M)\xrightarrow{\cong} \pi_0\Diff(M).\]
     \end{enumerate}
    Moreover, in each case the mapping class group $\pi_0 \Diff(M)$ is finite.
\end{proposition}

\begin{proof}
McCullough~\cite[p.~21]{McCullough1991} showed that for such 3-manifolds, \(\Out(\pi_1(M))\) is finite.

If \(H_1(M)\) is infinite, then by \cref{thm-SF Haken Classification}, \(M\) is Haken and the results of \cref{subsec:Haken-SFS} apply. In particular, since \(\pi_0\Diff(M) \cong\Out(\pi_1(M))\), the mapping class group of \(M\) is finite \cite[Prop.~3.4.5]{McCullough1991}. In this case, \(\pi_0\Diff^f(M)\) is a subgroup of \(\pi_0\Diff(M)\) by \cite[Thm.~\(3.9.1\)]{HKMR2012} as discussed above, so the mapping class group of fibre-preserving diffeomorphisms is also finite for \(M\).

If $\pi_1(M)$ is infinite, then we are either in the Haken case of \cref{subsec:Haken-SFS} above, or $H_1(M)$ is finite and $M$ admits $\mathbb{H}^2\times \mathbb{R}$ or $\SL(2,\mathbb{R})$ geometry. In these cases McCullough--Soma \cite[Prop.~8.3]{McCulloghSoma2013} show that each of the following maps is a bijection on path components
\[
\Isom(M)\to \Diff^f(M) \to \Diff(M) \to \Out(\pi_1(M)),
\]
hence once again the mapping class group is finite (and moreover the generalised Smale conjecture holds for these manifolds).

If \(\pi_1(M)\) is finite then by the elliptisation theorem \cite[1.7.3]{AFW2015}, \(M\) is spherical.
Thus, in this case the results of \cref{sec-elliptic} apply, and again the mapping class group is finite.
\end{proof}

In the elliptic case, excluding the exceptional cases of \cref{thm-fibre preserving equiv to diff}, we also have \(\Diff^f(M) \cong \Diff(M)\).
For the following non-Haken Seifert fibred \(3\)-manifolds, Birman--Rubinstein demonstrate some explicit calculations of the mapping class group.

\begin{theorem}[{\cite[Main Theorem II]{BirmanRubinstein1984}}]
    Let \(M= M(0; (2,1),(4k,2k-1),(m,n))\) such that \(1\leq k\), \(1\leq n \leq m\) and \((m,n) \neq (4k,1)\). Then the mapping class group is as follows.
    \[\pi_0\Diff(M) \cong \begin{cases}
        \Z/2\times\Z/2 & \text{if } (m,n)=(4k,2k-1)  \text{ or }  (m,n)= (2,1)  \text{ and } k\neq 1 ; \\
        \Z/2 &  \text{if \((m,n,k)=(2,2,1)\)} \\  & \text{or \((m,n)\neq (2,1),(4k,2k-1)\) and \((m,n,k)\neq (3,1,1)\); } \\
        \{1\} & \text{if \((m,n,k)=(3,1,1)\). } \\
    \end{cases}\]
In each case, the map $\rho \colon \pi_0 \Diff(M) \to \Out(\pi_1(M))$ is injective.
\end{theorem}
Furthermore, notice that unless \(m=2\), each diffeomorphism can be taken to be fibre-preserving by \cref{thm-fibre preserving equiv to diff}.

The key point of restricting to these particular Seifert fibred \(3\)-manifolds is that they each contain a closed, nonorientable embedded surface \(K\) of genus \(3\), unique up to isotopy. This surface represents a fixed element in \(H_2(M;\Z/2)\). The existence and uniqueness of such a surface means that a structure can be put on these Seifert fibred \(3\)-manifolds which is restrictive enough to compute the mapping class group. This structure is called a ``one-sided Heegaard Diagram'' and is studied extensively in \cite{Rubinstein1978, BirmanRubinstein1984}.

\begin{proof}[Sketch of proof]
The method employed in \cite{BirmanRubinstein1984} to obtain the above theorem is the following. Let \[G:= \left \{ [h]\in \pi_0\Diff(M) \mid h_* = \Id \colon H_1(M;\Z/2) \to H_1(M;\Z/2) \right\}.\]
That is, \(G\) is the kernel of the map \(\pi_0 \Diff(M) \to \Aut(H_1(M;\Z/2))\).
Hence, to calculate the mapping class group, it suffices to calculate \(G\) and the image of this map.
Notice that this is similar to how we calculated the mapping class group in the case of lens spaces, in the proof of \cref{thm-MCG lens spaces}.
In the case of Seifert fibred \(3\)-manifolds which satisfy the hypotheses of the theorem, it turns out that \(H_1(M;\Z/2)\) is either \(\Z/2\) or \(\Z/2\times \Z/2\). In the former case, the automorphism group is trivial, in the latter it is nontrivial but finite. Birman--Rubinstein use this to calculate the image. 

To determine \(G\), we have that any isotopy class in \(G\) contains a diffeomorphism~\(f\) such that \(f(K) =K\), and conversely any diffeomorphism \(f\colon M\to M\) such that \(f(K)=K\) is the identity on \(H_1(M;\Z/2)\). This means that \(G \cong G_1 / G_2\), where \(G_1\) is the subgroup of \(\pi_0\Diff(K)\) containing diffeomorphisms of \(K\) which extend to diffeomorphisms of \(M\), and \(G_2\) is the subgroup of \(G_1\) such that the extended diffeomorphism on \(M\) is isotopic to the identity.

In \cite{BirmanRubinstein1984}, they use the fact that \(\pi_0\Diff(K)\cong \GL{(2,\Z)}\) to compute both \(G_1\) and \(G_2\), in order to obtain \(G\) and calculate the mapping class group.
\end{proof}

We also mention the following result of Boileau-Otal. 

\begin{theorem}[{\cite[Th\'eor\`eme~3]{BoileauOtalHeegardSplittings91}}]
 Let $M$ be a closed, Seifert fibred 3-manifold with base $S^2$ and three exceptional fibres  of orders $(2,3,p)$, $p \geq 5$, or $(3,3,q)$, $q \geq 2$. Then every diffeomorphism that is homotopic to the identity is isotopic to the identity.
\end{theorem}

\section{JSJ decompositions and mapping class groups}\label{sec: JSJ}

In this section we recall the main results concerning JSJ decompositions, and then relate the mapping class group of a 3-manifold to those of its JSJ components. Throughout, $M$ denotes a compact, oriented, irreducible 3-manifold with toroidal boundary.

\subsection{JSJ decompositions}
All results stated in this subsection can be found in \cite[Section 1.7]{AFW2015}. The elliptisation, hyperbolisation, and geometrisation theorems are due to Perelman~\cite{Perelman:2002-1,Perelman:2003-1,Perelman:2003-2}.

As mentioned in the introduction, the JSJ decomposition of Jaco--Shalen and Johannson \cite{JacoShalen1976, Johannson1979b} decomposes a compact, orientable, irreducible 3-manifold $M$ with toroidal boundary into Seifert fibred (see \cref{defn:SFS-defn}) and atoroidal pieces.

\begin{definition}
A 3-manifold $M$ is \emph{atoroidal} if every incompressible torus is boundary parallel.
\end{definition}

\begin{theorem}[JSJ]
Suppose that $M$ is a compact, oriented, irreducible  3-manifold with $\partial M \cong \bigsqcup^k T^2$ for some $k \geq 0$.
\begin{enumerate}[(a)]
\item There is an $m \geq 0$ and a submanifold $\bigsqcup_{i=1}^m T_i \subseteq M$ such that for every $i$, $T_i \subseteq M$ is an incompressible torus, and each component of $M$ cut along $\bigsqcup_{i=1}^m T_i$ is Seifert fibred or atoroidal.
\item Any two minimal submanifolds with the above property are isotopic.
\end{enumerate}
\end{theorem}

Recall that by the elliptisation theorem if $\pi_1(M)$ is finite, then $M$ is elliptic (hence Seifert fibred). Recall that if $M$ is hyperbolic, then it is atoroidal.

\begin{theorem}[Hyperbolisation theorem]
Suppose that $M$ is a compact, oriented, irreducible  3-manifold with $\partial M \cong \bigsqcup^k T^2$ for some $k \geq 0$. If $M$ is atoroidal and $\pi_1(M)$ is infinite, then it is either hyperbolic with finite volume or diffeomorphic to one of $S^1 \times D^2$, $T^2 \times I$ or the nontrivial $I$-bundle over the Klein bottle.
\end{theorem}

Note that all $3$ exceptional manifolds in the hyperbolisation theorem are Seifert fibred.

\begin{theorem}[Geometrisation theorem]
Suppose that $M$ is a compact, oriented, irreducible  3-manifold with $\partial M \cong \bigsqcup^k T^2$ for some $k \geq 0$.
\begin{enumerate}[(a)]
\item  There is an $m \geq 0$ and a submanifold $\bigsqcup_{i=1}^m T_i \subseteq M$ such that for every $i$, $T_i \subseteq M$ is an incompressible torus, and each component of $M$ cut along $\bigsqcup_{i=1}^m T_i$ is Seifert fibred or hyperbolic.
\item  Any two minimal submanifolds with the above property are isotopic.
\end{enumerate}
\end{theorem}

\begin{proof}
By the previous observations ``each component is Seifert fibred or hyperbolic" $\Leftrightarrow$ ``each component is Seifert fibred or atoroidal".
\end{proof}

\begin{proposition}\label{prop-SFgeometry}
Suppose that $M$ is a compact, oriented, irreducible 3-manifold with $\partial M \cong \bigsqcup^k T^2$ for some $k \geq 0$. Suppose that  $M$ is not diffeomorphic to $S^1 \times D^2$, $T^2 \times I$, or the nontrivial $I$-bundle over the Klein bottle. Then $M$ is Seifert fibred if and only if it has a geometry other than hyperbolic or Sol.
\end{proposition}

\begin{remark}
There is also a unique minimal decomposition into geometric components, which is different from the JSJ decomposition in general (it may contain fewer tori). This difference is due to the Sol components, which may be decomposed further to get a JSJ decomposition. For instance, Sol geometry arises for torus bundles over $S^1$ with Anosov monodromy, meaning that the trace of the action of the monodromy on the fibre is greater than two. 
\end{remark}

\subsection{Mapping class groups}

The goal of this section is to relate the mapping class group of $M$ to those of its JSJ components. We will do this in multiple steps, replacing the mapping class group with other groups that are gradually closer to what we can obtain from assembling diffeomorphisms of the individual JSJ components, and describe how the group changes in each step. We will also consider the full diffeomorphism group alongside the mapping class group.

It will be convenient to only consider orientation-preserving diffeomorphisms. Recall that there are always exact sequences
\[
0 \to \Diff^+(M) \to \Diff(M) \to \Z/2
\]
\[
0 \to \pi_0 \Diff^+(M) \to \pi_0 \Diff(M) \to \Z/2
\]
so either $\Diff^+(M) \cong \Diff(M)$ and $\pi_0 \Diff^+(M) \cong \pi_0 \Diff(M)$ or, if there exists an orientation-reversing diffeomorphism, $\Diff^+(M)$ and $\pi_0 \Diff^+(M)$ are index two subgroups of $\Diff(M)$ and $\pi_0 \Diff(M)$, respectively.

\begin{remark}
If $M$ has a nontrivial JSJ decomposition, then it is Haken.
Thus Theorem \ref{thm:Waldhausen} can be used to determine the mapping class group of $M$ (but not to compare it with the mapping class groups of the JSJ components).  However, in some cases, the description in terms of $\Out(\pi_1(M))$ might not be the most useful one.
\end{remark}

Fix a compact, oriented, irreducible 3-manifold $M$ with $\partial M = \bigsqcup_{h=1}^k B_h$ for some $k \geq 0$, where each boundary component $B_h \cong T^2$ is a torus. Also fix a JSJ decomposition $T = \bigsqcup_{i=1}^m T_i \subseteq M$, and let $M_1, \ldots, M_n$ denote the JSJ components.

\begin{definition}
Define $\Diff^+(M,T) := \{ f \in \Diff^+(M) \mid f(T)=T \}$.
\end{definition}

\begin{proposition} \label{prop:diff-mt-fibr}
There is a fibration
\[
\Diff^+(M,T) \rightarrow \Diff^+(M) \rightarrow (\Emb(T,M)/\Diff(T))_T
\]
where $\Emb(T,M)/\Diff(T)$ is the moduli space of submanifolds of $M$ diffeomorphic to $T$ and $(\Emb(T,M)/\Diff(T))_T$ is the path component of $T$.

Further, $(\Emb(T,M)/\Diff(T))_T$ is either contractible $($hence $\Diff^+(M,T) \simeq \Diff^+(M))$, or $(\Emb(T,M)/\Diff(T))_T$ is homotopy equivalent to $S^1$.
\end{proposition}

\begin{proof}
There is a map $t \colon \Diff^+(M) \rightarrow \Emb(T,M)/\Diff(T)$ given by $t(f)=f(T)$. By the uniqueness part of the JSJ theorem, $\Image(t) \subseteq (\Emb(T,M)/\Diff(T))_T$. By isotopy extension $\Image(t) \supseteq (\Emb(T,M)/\Diff(T))_T$. One can use parametrised isotopy extension to check that $t$ is a fibration.

If $M$ is not a torus bundle over $S^1$, then by the proof of \cite[Corollary 5.21]{BoydBregmanSteinebrunnerModuliSpacesFinite24}, $(\Emb(T,M)/\Diff(T))_T$ is contractible. If $M$ is a torus bundle, then by cutting it along a fibre torus, we get $T^2 \times I$, which is Seifert fibred. Therefore the JSJ decomposition $T$ of $M$ is either empty, or it is $T \cong T^2$ which can be taken to be a fibre, depending on whether $M$ itself is Seifert fibred or not. In the former case $(\Emb(T,M)/\Diff(T))_T$ is contractible, and in the latter case it is homotopy equivalent to $S^1$ by \cite[Thm 1(a)]{Hatcher1999}.
\end{proof}

\begin{proposition} \label{prop:diff-mt-pi0}
There is a short exact sequence
\[
1 \rightarrow \pi_1 (\Emb(T,M)/\Diff(T))_T \rightarrow \pi_0 \Diff^+(M,T) \rightarrow \pi_0 \Diff^+(M) \rightarrow 1
\]
and $\pi_1 (\Emb(T,M)/\Diff(T))_T$ is either trivial $($hence $\pi_0 \Diff^+(M,T) \cong \pi_0 \Diff^+(M))$ or isomorphic to $\Z$.
\end{proposition}

\begin{remark}
$\pi_1(\Diff^+(M),\Diff^+(M,T)) \cong \pi_1 (\Emb(T,M)/\Diff(T))_T$ is called the motion group of the pair $(M,T)$.
\end{remark}

\begin{proof}
Aside from the $1$ on the left, this is the long exact sequence of the fibration in Proposition \ref{prop:diff-mt-fibr}. It remains to prove that the map $\pi_1 (\Emb(T,M)/\Diff(T))_T \rightarrow \pi_0 \Diff^+(M,T)$ is injective.

As we saw in the proof of Proposition \ref{prop:diff-mt-fibr}, if $M$ is not a torus bundle over $S^1$, or if it is Seifert fibred, then $(\Emb(T,M)/\Diff(T))_T$ is contractible, and hence $\pi_1 (\Emb(T,M)/\Diff(T))_T \cong 1$. So we will assume that $M$ is a torus bundle over $S^1$ and $T \cong T^2$ is a fibre, and that $M$ is not Seifert fibred.
This means that $M$ is an Anosov bundle~\cite[p. 24]{BoydBregmanSteinebrunnerModuliSpacesFinite24}.
The isomorphism class of the torus bundle $M$ is determined by an element $[g] \in \pi_0 \Diff^+(T^2) \cong SL_2(\Z)$, and since it is Anosov, $[g]$ has infinite order.

By \cite[Thm 1(a)]{Hatcher1999} we have $\pi_1 (\Emb(T,M)/\Diff(T))_T \cong \Z$ and the generator~$\ell$ is given by the map $S^1 \rightarrow (\Emb(T,M)/\Diff(T))_T$ sending each point of $S^1$ to the fibre of $M \rightarrow S^1$ over that point. Describe $M$ as $T^2 \times \R / (x,r) \sim (g(x),r+1)$. Then the path in $\Diff^+(T^2 \times \R)$ given by $(x,r) \mapsto (x,r+t)$ for $t \in [0,1]$ determines a path in $\Diff^+(M)$, which is a lift of $\ell$. In particular, the diffeomorphism determined by $(x,r) \mapsto (x,r+1)$ is the image of $\ell$ in $\pi_0 \Diff^+(M,T)$. Hence the composition $\pi_1 (\Emb(T,M)/\Diff(T))_T \rightarrow \pi_0 \Diff^+(M,T) \rightarrow \pi_0 \Diff^+(T)$ sends $\ell$ to $[g^{-1}]$, which has infinite order, and this implies that $\pi_1 (\Emb(T,M)/\Diff(T))_T \rightarrow \pi_0 \Diff^+(M,T)$ is injective.
\end{proof}

From now on we consider $\pi_0 \Diff^+(M,T)$.

\begin{definition}
We define
\begin{multline*}
\Diff^+(M,(B_h),(T_i),(M_j)) \coloneqq \\
\{ f \in \Diff^+(M,T) \mid f(B_h)=B_h, f(T_i)=T_i, f(M_j)=M_j \text{ for every $h,i,j$} \}
\end{multline*}
\end{definition}

\begin{definition}
Let $G = (V,E,\alpha)$ be the not necessarily simple, undirected, vertex-decorated graph with vertex and edge set
\[
\begin{aligned}
V &= \{ M_j \mid 1 \leq j \leq n \} \cup \{ * \} \\
E &= \{ B_h \mid 1 \leq h \leq k \} \cup \{ T_i \mid 1 \leq i \leq m \}
\end{aligned}
\]
and vertex decoration determined by
\[
\alpha(M_j) = [\text{the diffeomorphism class of $M_j$}], \quad \alpha(*) = \emptyset
\]
where the edge $B_h$ connects the distinguished vertex $*$ to the vertex $M_j$ if $B_h$ is a boundary component of $M_j$, and the edge $T_i$ connects the vertices corresponding to the neighbouring components $M_j$ (which may be equal).
\end{definition}

Note that all automorphisms of $G$ fix $*$, since it is the only vertex with decoration~$\emptyset$.

\begin{proposition} \label{prop:ses-aut-g}
There are exact sequences
\[
1 \rightarrow \Diff^+(M,(B_h),(T_i),(M_j)) \rightarrow \Diff^+(M,T) \rightarrow \Aut(G)
\]
\[
1 \rightarrow \pi_0 \Diff^+(M,(B_h),(T_i),(M_j)) \rightarrow \pi_0 \Diff^+(M,T) \rightarrow \Aut(G)
\]
\end{proposition}

\begin{proof}
A diffeomorphism in $\Diff^+(M,T)$ permutes the $B_h$, the $T_i$ and the $M_j$, preserving incidence and the diffeomorphism types of the components, and hence induces an automorphism of $G$. The kernel of the resulting map $ \Diff^+(M,T) \rightarrow \Aut(G)$ is $\Diff^+(M,(B_h),(T_i),(M_j))$ by definition. Since $\Aut(G)$ is discrete, the sequence remains exact after taking $\pi_0$.
\end{proof}

\begin{definition}
We define
\begin{multline*}
\overline{\Diff}^+(M,(B_h),(T_i),(M_j)) \coloneqq \\
\{ f \in \Diff^+(M,(B_h),(T_i),(M_j)) \mid f \big| _{T_i} \colon T_i \rightarrow T_i \text{ is orientation preserving for every $i$} \}
\end{multline*}
\end{definition}

\begin{remark}
Since a diffeomorphism $f \in \Diff^+(M,(B_h),(T_i),(M_j))$ is orientation-preserving on $M$, it is orientation-preserving on a given $T_i$ if and only if its derivative is orientation-preserving on the normal bundle of $T_i \subseteq M$. Because $f$ preserves the components of $M$, this always holds except possibly when $T_i$ has the same $M_j$ on both of its sides or, equivalently, if it corresponds to a loop in $G$.
\end{remark}

\begin{definition}
Let $\overline{G} = (V,E,\alpha)$ be the $1$-dimensional $\Delta$-complex with decoration that is (the geometric realisation of) $G$.
\end{definition}

\begin{remark}
There is a short exact sequence
\[
1 \rightarrow (\Z/2)^{\ell} \rightarrow \Aut(\overline{G}) \rightarrow \Aut(G) \rightarrow 1
\]
where $\ell$ is the number of loops in $G$. A standard generator of $(\Z/2)^{\ell}$ is mapped to the automorphism of $\overline{G}$ that reverses the corresponding loop and is the identity on other simplices. The sequence (non-canonically) splits on the right: if we arbitrarily fix orientations for all loops, then an automorphism of $G$ can be uniquely lifted to an automorphism of $\overline{G}$ that respects the chosen orientations, determining a splitting homomorphism $\Aut(G) \rightarrow \Aut(\overline{G})$.
\end{remark}

\begin{proposition} \label{prop:aut-G-bar}
There are exact sequences
\[
1 \rightarrow \overline{\Diff}^+(M,(B_h),(T_i),(M_j)) \rightarrow \Diff^+(M,T) \rightarrow \Aut(\overline{G})
\]
\[
1 \rightarrow \pi_0 \overline{\Diff}^+(M,(B_h),(T_i),(M_j)) \rightarrow \pi_0 \Diff^+(M,T) \rightarrow \Aut(\overline{G})
\]
\end{proposition}

\begin{proof}
For each $T_i$ that corresponds to a loop in $\overline{G}$, fix a bijection between the orientations of (the normal bundle of) $T_i$ and the orientations of the loop. Then any diffeomorphism in $\Diff^+(M,T)$ induces an automorphism of $\overline{G}$. The rest goes as in Proposition \ref{prop:ses-aut-g}.
\end{proof}

From now on we consider $\pi_0 \overline{\Diff}^+(M,(B_h),(T_i),(M_j))$.

\begin{definition}
Given a manifold with boundary $N$, define
\[
\Diff^+_{\pi_0 \partial}(N) := \{ f \in \Diff^+(N) \mid \pi_0 \partial f = \id \colon \pi_0 \partial N \rightarrow \pi_0 \partial N \}.
\]
That is, $f \in \Diff^+_{\pi_0 \partial}(N)$ if $f(S)=S$ for every connected component $S$ of $\partial N$.
\end{definition}

\begin{remark} \label{rem:diff-pi0-d}
There are exact sequences
\[
1 \rightarrow \Diff^+_{\pi_0 \partial}(N) \rightarrow \Diff^+(N) \rightarrow \Aut(\pi_0 \partial N)
\]
\[
1 \rightarrow \pi_0 \Diff^+_{\pi_0 \partial}(N) \rightarrow \pi_0 \Diff^+(N) \rightarrow \Aut(\pi_0 \partial N)
\]
where $\Aut(\pi_0 \partial N)$ is the symmetric group on the finite set $\pi_0 \partial N$.
\end{remark}

For every connected component $S$ of $\partial N$ there is a natural restriction map $\Diff^+_{\pi_0 \partial}(N) \rightarrow \Diff^+(S)$.

\begin{definition}
Define a diagram $D = D(M,(T_i),(M_j))$
in the category of topological groups as follows. The objects of $D$ are $\Diff^+_{\pi_0 \partial}(M_j)$ and $\Diff^+(T_i)$ for every $j$ and $i$. There is a morphism $\Diff^+_{\pi_0 \partial}(M_j) \rightarrow \Diff^+(T_i)$, given by restriction, if $T_i$ is a boundary component of $M_j$ (or two morphisms, if $M_j$ is on both sides of $T_i$).

Let $\pi_0 D$ be the diagram in the category of groups obtained from $D$ by applying $\pi_0$ to every object and morphism.
\end{definition}

\begin{proposition} \label{prop:diff-lim}
We have $\overline{\Diff}^+(M,(B_h),(T_i),(M_j)) \cong \lim D(M,(T_i),(M_j))$.
\end{proposition}

\begin{proof}
There are restriction maps $\overline{\Diff}^+(M,(B_h),(T_i),(M_j)) \rightarrow \Diff^+_{\pi_0 \partial}(M_j)$ and $\overline{\Diff}^+(M,(B_h),(T_i),(M_j)) \rightarrow \Diff^+(T_i)$ for every $j$ and $i$, which commute with the morphisms in $D$. We need to show that this system satisfies the universal property of limits. 

If there is a system of maps from some topological group $X$ to $D$, then for every $x \in X$, we get a diffeomorphism in $\Diff^+_{\pi_0 \partial}(M_j)$ for every $j$. For each $i$, the restrictions to $T_i$ of the diffeomorphisms on the two neighbouring components (or, if $T_i$ has the same $M_j$ on both of its sides, the restrictions of the diffeomorphism on $M_j$ to the two boundary components corresponding to $T_i$) are equal. Hence all of the diffeomorphisms can be glued together to form a diffeomorphism of $M$, which is in $\overline{\Diff}^+(M,(B_h),(T_i),(M_j))$. A diffeomorphism of $M$ whose restriction to $M_j$ agrees with the given element of $\Diff^+_{\pi_0 \partial}(M_j)$ for every $j$ is obviously unique. So, by applying this gluing for each $x \in X$, we get a unique map $X \rightarrow \overline{\Diff}^+(M,(B_h),(T_i),(M_j))$, which is a morphism of topological groups.
\end{proof}

Since $\pi_0$ does not preserve limits, an additional step will be needed in the case of $\pi_0 \overline{\Diff}^+(M,(B_h),(T_i),(M_j))$.
In particular, applying a collar twist near a torus $T_i$, given by an element of $\pi_1 \Diff^+(T_i)$, to a diffeomorphism in $\overline{\Diff}^+(M,(B_h),(T_i),(M_j))$ has no effect on the isotopy classes of its restrictions to the $M_j$ and $T_{i'}$ (see Lemma \ref{lem:iota-properties} (c)). Showing that this is the only kind of change that cannot be detected by $\lim \pi_0 D$ will be the main step in relating this limit to $\pi_0 \overline{\Diff}^+(M,(B_h),(T_i),(M_j))$ (see Proposition \ref{prop:mcg-lim}). We first set up some notation.

\begin{definition} \label{def:iota-adj}
Choose one side of $T_i$ in $M$. Choose a small collar neighbourhood on that side and identify it with $T_i \times [0, 1]$ (such that $T_i = T_i \times \{ 0 \}$). We get a map (independent of the choice of collar neighbourhood and identification)
\[
\iota_i \colon \pi_1 \Diff^+(T_i) \rightarrow \pi_0 \Diff^+_{T_i \times \{ 0, 1 \} }(T_i \times [0, 1]) \rightarrow \pi_0 \overline{\Diff}^+(M,(B_h),(T_i),(M_j))
\]
by taking the adjunction and extending a diffeomorphism of $T_i \times [0, 1]$ by the identity on $M \setminus T_i \times [0, 1]$.
\end{definition}

\begin{lemma} \label{lem:iota-properties}
The map $\iota_i$ has the following properties.
\begin{enumerate}[(a)]
\item $\Image(\iota_i)$ is independent of which side of $T_i$ is chosen to define $\iota_i$.
\item $\Image(\iota_i) \leq \pi_0 \overline{\Diff}^+(M,(B_h),(T_i),(M_j))$ is a normal subgroup.
\item The composition $\pi_1 \Diff^+(T_i) \xra{\iota_i} \pi_0 \overline{\Diff}^+(M,(B_h),(T_i),(M_j)) \rightarrow \pi_0 \Diff^+_{\pi_0 \partial}(M_j)$ is trivial for every $j$.
\end{enumerate}
\end{lemma}

\begin{proof}
(a) Identify a collar neighbourhood on the other side of $T_i$ with $T_i \times [-1, 0]$, and let $\iota'_i$ be the map obtained by using that neighbourhood. Given a loop $u \colon [0,1] \rightarrow \Diff^+(T_i)$ representing $[u] \in \pi_1 \Diff^+(T_i)$, we define the isotopy $p_t \colon M \rightarrow M$, $t \in [0,1]$, by setting $p_t(x,s) = (u(t+s)(x),s)$ for $(x,s) \in T_i \times [-t, 1-t]$, and $p_t$ is the identity outside $T_i \times [-t, 1-t]$. Then $p_t$ is an isotopy in $\overline{\Diff}^+(M,(B_h),(T_i),(M_j))$ between representatives of $\iota_i([u])$ and $\iota'_i(-[u])$, showing that $\iota_i = \iota'_i \circ (-\id)$, and hence $\Image(\iota_i) = \Image(\iota'_i)$.

(b) Suppose that $[u] \in \pi_1 \Diff^+(T_i)$ and $[f] \in \pi_0 \overline{\Diff}^+(M,(B_h),(T_i),(M_j))$. By applying an isotopy, we can assume that $[f]$ is represented by a diffeomorphism $f$ such that $f \big| _{T_i \times [0,1]} = f_0 \times \id$, where $f_0 = f \big| _{T_i} \colon T_i \rightarrow T_i$. Then $[f] \circ \iota_i([u]) \circ [f]^{-1}$ is represented by the diffeomorphism that is given by $(x,t) \mapsto (f_0 \circ u(t) \circ f_0^{-1}(x),t)$ for $(x,t) \in T_i \times [0, 1]$ and is the identity outside $T_i \times [0, 1]$. Hence $[f] \circ \iota_i([u]) \circ [f]^{-1} = \iota_i([f_0uf_0^{-1}]) \in \Image(\iota_i)$.

(c) Suppose that $[u] \in \pi_1 \Diff^+(T_i)$. If the chosen collar neighbourhood $T_i \times [0, 1]$ is not in $M_j$, then $\iota_i([u])$ restricts to the identity on $M_j$ by construction. If $T_i \times [0, 1]$ is in $M_j$, then let $p_t \colon M_j \rightarrow M_j$ be given by $p_t(x,s) = (u(t+s)(x),s)$ for $(x,s) \in T_i \times [0, 1-t]$ and the identity outside $T_i \times [0, 1-t]$. Then $p_t$ is an isotopy in $\Diff^+_{\pi_0 \partial}(M_j)$ between the restriction of $\iota_i([u])$ and $\id$.
\end{proof}

\begin{definition}
For every $i$ choose a collar neighbourhood $T_i \times [0, 1]$ as above, and define $\iota_i$. By choosing small enough neighbourhoods, we can assume that they are pairwise disjoint, and then the elements in the images of different $\iota_i$ maps will commute with each other. Hence we obtain a map $\sum_i \iota_i \colon \bigoplus_i \pi_1 \Diff^+(T_i) \rightarrow \pi_0 \overline{\Diff}^+(M,(B_h),(T_i),(M_j))$. It follows from Lemma \ref{lem:iota-properties} (b) that $\Image(\sum_i \iota_i)$ is a normal subgroup. We will denote the cokernel of $\sum_i \iota_i$ by
\[
\pi_0 \overline{\Diff}^+(M,(B_h),(T_i),(M_j)) \Big/ \bigoplus_i \pi_1 \Diff^+(T_i).
\]
\end{definition}

While the map $\sum_i \iota_i$ depends on which side of each $T_i$ is chosen, by Lemma \ref{lem:iota-properties} (a) its cokernel $\pi_0 \overline{\Diff}^+(M,(B_h),(T_i),(M_j)) / \bigoplus_i \pi_1 \Diff^+(T_i)$ is independent of all choices. By the definition of $\pi_0 \overline{\Diff}^+(M,(B_h),(T_i),(M_j)) \Big/ \bigoplus_i \pi_1 \Diff^+(T_i)$ we have the following proposition.

\begin{proposition} \label{prop:diff-coker}
There is an exact sequence
\begin{multline*}
\bigoplus_i \pi_1 \Diff^+(T_i) \xra{\sum_i \iota_i} \pi_0 \overline{\Diff}^+(M,(B_h),(T_i),(M_j)) \rightarrow \\
 \rightarrow \pi_0 \overline{\Diff}^+(M,(B_h),(T_i),(M_j)) \Big/ \bigoplus_i \pi_1 \Diff^+(T_i) \rightarrow 1
\end{multline*}
where $\pi_1 \Diff^+(T_i) \cong \pi_1 \Diff^+(T^2) \cong \Z^2$ for every $i$.
\end{proposition}

We can now prove the analogue of Proposition \ref{prop:diff-lim} for mapping class groups.

\begin{proposition} \label{prop:mcg-lim}
We have
\[
\pi_0 \overline{\Diff}^+(M,(B_h),(T_i),(M_j)) \Big/ \bigoplus_i \pi_1 \Diff^+(T_i) \cong \lim \pi_0 D(M,(T_i),(M_j)).
\]
\end{proposition}

\begin{proof}
By applying $\pi_0$ to the limit cone of $D$ and \cref{prop:diff-lim}, we get a collection of homomorphisms $\pi_0 \overline{\Diff}^+(M,(B_h),(T_i),(M_j)) \rightarrow \pi_0 \Diff^+_{\pi_0 \partial}(M_j)$ and $\pi_0 \overline{\Diff}^+(M,(B_h),(T_i),(M_j)) \rightarrow \pi_0 \Diff^+(T_i)$ for every $j$ and $i$, which commute with the morphisms in $\pi_0 D$. These induce well-defined homomorphisms from $\pi_0 \overline{\Diff}^+(M,(B_h),(T_i),(M_j)) / \bigoplus_i \pi_1 \Diff^+(T_i)$ to $\pi_0 \Diff^+_{\pi_0 \partial}(M_j)$ (by Lemma \ref{lem:iota-properties} (c)) and to $\pi_0 \Diff^+(T_i)$ (by Definition \ref{def:iota-adj}).

If there is a system of homomorphisms from some group $X$ to $\pi_0 D$, then for every $x \in X$ we get a representative diffeomorphism in $\Diff^+_{\pi_0 \partial}(M_j)$ for every $j$ such that for each $i$ the two restrictions to $T_i$ are isotopic. For each $i$, we can choose one of the components $M_j$ neighbouring $T_i$, and isotope the diffeomorphism of $M_j$ on a collar neighbourhood of $T_i$, to make its restriction to $T_i$ equal the restriction of the diffeomorphism on the other side of $T_i$. Then the diffeomorphisms can be glued together to a diffeomorphism of $M$. Thus we get an element in $\pi_0 \overline{\Diff}^+(M,(B_h),(T_i),(M_j))$, and hence in $\pi_0 \overline{\Diff}^+(M,(B_h),(T_i),(M_j)) / \bigoplus_i \pi_1 \Diff^+(T_i)$. We need to show that the latter is uniquely determined.

Suppose we have two elements of $\pi_0 \overline{\Diff}^+(M,(B_h),(T_i),(M_j)) / \bigoplus_i \pi_1 \Diff^+(T_i)$, represented by diffeomorphisms $f,g \in \overline{\Diff}^+(M,(B_h),(T_i),(M_j))$, such that $f \big| _{M_j}$ is isotopic to $g \big| _{M_j}$ for every $j$. Choose an isotopy $p_{j,t} \colon M_j \rightarrow M_j$, $t \in [0,1]$ from $p_{j,0} = f \big| _{M_j}$ to $p_{j,1} = g \big| _{M_j}$. For a $T_i$, if the neighbouring components are $M_{j_1}$ and $M_{j_2}$, let $P_{i,t} = p_{j_2,1-t} \big| _{T_i} \circ (p_{j_1,1-t} \big| _{T_i})^{-1} \colon T_i \rightarrow T_i$, $t \in [0,1]$, which is a loop in $\Diff^+(T_i)$ based at $\id$, and hence represents an element $[P_i] \in \pi_1 \Diff^+(T_i)$. (If $T_i$ has the same $M_j$ on both of its sides, then replace $p_{j_1,1-t} \big| _{T_i}$ and $p_{j_2,1-t} \big| _{T_i}$ with the restrictions of $p_{j,1-t}$ to the corresponding two boundary components of $M_j$.) Define $g' \in \overline{\Diff}^+(M,(B_h),(T_i),(M_j))$ by composing $g$ with $\iota_i([P_i])$ (where the $M_{j_1}$ side of $T_i$ is used to define $\iota_i$), for every $i$.

Define isotopies $q_{j,t} \colon M_j \rightarrow M_j$, $t \in [0,1]$ as follows. For every $i$ such that $M_j$ contains the previously chosen collar neighbourhood $T_i \times [0,1]$, let $q_{j,t}(x,s) = (P_{i,1+s-t}(x),s)$ for $(x,s) \in T_i \times [0,t]$, and let $q_{j,t}$ be the identity outside these neighbourhoods $T_i \times [0,t]$. Let $p'_{j,t} = q_{j,t} \circ p_{j,t}$, it is an isotopy from $f \big| _{M_j}$ to $g' \big| _{M_j}$, for every $j$. Moreover, for every $T_i$ with neighbouring components $M_{j_1}$ and $M_{j_2}$, we have $p'_{j_1,t} \big| _{T_i} = q_{j_1,t} \big| _{T_i} \circ p_{j_1,t} \big| _{T_i} = P_{i,1-t} \circ p_{j_1,t} \big| _{T_i} =  p_{j_2,t} \big| _{T_i} = p'_{j_2,t} \big| _{T_i}$ for all $t \in [0,1]$. This means that the isotopies $p'_{j,t}$ for all $j$ can be glued together, forming an isotopy from $f$ to $g'$ in $\overline{\Diff}^+(M,(B_h),(T_i),(M_j))$. By construction, $g$ and $g'$ represent the same element in $\pi_0 \overline{\Diff}^+(M,(B_h),(T_i),(M_j)) / \bigoplus_i \pi_1 \Diff^+(T_i)$. Therefore $f$ and $g$ represent the same element in $\pi_0 \overline{\Diff}^+(M,(B_h),(T_i),(M_j)) / \bigoplus_i \pi_1 \Diff^+(T_i)$, as required.
\end{proof}

The following chains summarise the relationship between the diffeomorphism or mapping class group of $M$ and those of its JSJ components:
\[
\xymatrix{
\Diff^+(M) & \Diff^+(M,T) \ar@{_{(}->}[l] & \overline{\Diff}^+(M,(B_h),(T_i),(M_j)) \cong \lim D(M,(T_i),(M_j)) \ar@{_{(}->}[l]
}
\]
\[
\xymatrix@!C=4cm{
\pi_0 \Diff^+(M) & \pi_0 \Diff^+(M,T) \ar@{>>}[l] & \pi_0 \overline{\Diff}^+(M,(B_h),(T_i),(M_j)) \ar@{_{(}->}[l]  \ar@{>>}[d] \\
 & & \lim \pi_0 D(M,(T_i),(M_j)) \cong \pi_0 \overline{\Diff}^+(M,(B_h),(T_i),(M_j)) / \bigoplus_i \pi_1 \Diff^+(T_i). \quad \quad \quad \quad \quad \quad \quad \quad \quad \quad \quad \quad \quad \quad \quad \quad \quad
}
\]
The individual steps were described in Propositions \ref{prop:diff-mt-fibr}, \ref{prop:aut-G-bar} and \ref{prop:diff-lim} for the diffeomorphism groups, and in Propositions \ref{prop:diff-mt-pi0}, \ref{prop:aut-G-bar}, \ref{prop:diff-coker} and \ref{prop:mcg-lim} for the mapping class groups (see also Remark \ref{rem:diff-pi0-d}).

\section{Mapping class groups of reducible 3-manifolds}\label{sec-reducible}
This section focuses on the mapping class groups of compact, orientable, reducible 3-manifolds.
First we describe what is known about how to obtain the mapping class group of a reducible 3-manifold from that of its prime summands.
Then we consider the kernel of the map $\rho \colon \pi_0 \Diff(M) \to \Out(\pi_1(M))$ for $M$ reducible. This leads us to consider an exact sequence of Laudenbach, for closed, orientable 3-manifolds, and we describe how a splitting for it.

\subsection{The mapping class group of reducible 3-manifolds in terms of its summands}\label{subsec-reducible-in-terms-summands}

Let $M$ be a compact, orientable, reducible 3-manifold. We describe what is known about $\pi_0\Diff(M)$ in terms of the mapping class groups of its irreducible summands.

The literature on diffeomorphism groups of reducible 3-manifolds began with an announcement by C\'esar de S\'a--Rourke \cite{CesardeSaRourke1979}, that did not contain proofs.  According to McCullough~\cite{McCullough-Warsaw-notes}, they were ultimately not able to provide the proofs of the results they had announced.
Further work on the programme was done by Hendriks--Laudenbach~\cite{HendriksLaudenbach1984} and Hendriks--McCullough~\cite{HendriksMcCullough1987}. An unfinished preprint by Hatcher, available on his website, suggested an alternative perspective. Recently Boyd--Bregman--Steinebrunner~\cite{boyd2026primedecompositionfibresequence} described a new fibre sequence for $\BDiff(M)$.  Much of this work deals with the homotopy type of $\Diff(M)$.

We focus on the mapping class group of $M$.
A complete description of the mapping class group is not known, but we do know a generating set.
The following result was stated as \cite{CesardeSaRourke1979}, and proven by McCullough in \cite{McCullough-Warsaw-notes}.

\begin{theorem}\label{conjecture:MCG-reducible-generators}
    Every mapping class in $\pi_0\Diff_\partial (M)$ is a composition of:
    \begin{enumerate}[(i)]
        \item     slide diffeomorphisms;
        \item transpositions of repeated prime factors;
        \item flips of $S^1 \times S^2$ factors;
 \item Gluck twists on the $S^1 \times S^2$ factors; and
\item mapping classes in $\pi_0 \Diff_{D_i \cup \partial P_i}(P_i)$, where the $P_i$ are the irreducible factors of $M$ and $D_i \subseteq P_i$ is a 3-disc.
    \end{enumerate}
\end{theorem}

To understand this statement, we define the diffeomorphism involved.
First we fix notation and describe a model for $M$.
Let \[M=P_1 \# P_2 \# \cdots \# P_n \# (S^2\times S^1)^{\#g}\]
be the prime decomposition of $M$.
An explicit model for $M$ is given by taking a 3-sphere with $n+2g$ open balls removed, denoted by $B$, and attaching $P_i\setminus \mathring{D}_i$ to the boundary components $F_1,\dots,F_n$ created by removing $n$ of the open balls, and $g$ copies of $S^2\times I$ to the remaining $2g$ boundary components, $F_{n+1},\dots,F_{n+2g}$. The latter attachments are done so as to get an orientable 3-manifold.
Let $S_j \times I$ denote the $j$th copy of $S^2 \times I$. We assume that $S_j \times \{0\}$ is glued to $F_{n+2j-1}$, and $S_j \times \{1\}$ is glued to $F_{n+2j}$.

Now we define slide, transposition, and flip diffeomorphisms.

\begin{definition}\label{defn:slide-diffeo}
We define \emph{slide diffeomorphisms}.   Let $P_i$ be one of the prime components of $M$ and form $\widehat{M}$ by removing $P_i\sm D_i$ and gluing in a 3-ball $E$ to $F_i$.  Let $a$ be an oriented arc in $\widehat M$ such that $E\cap a$  consists of the two endpoints of $a$.  Now use $a$ to define a point-pushing isotopy that moves the centre of $E$ along the arc $a$ and then back to itself.  By isotopy extension this defines a diffeomorphism of $\widehat{M}$ to itself, and, by an argument analogous to that of proving that connected sum is well-defined, we can assume that this diffeomorphism is the identity on $E$.  Hence we can cut out $E$, replace it by $P_i\sm D_i$ and extend by the identity to obtain the \emph{slide diffeomorphism of} $P_i$ \emph{around} $a$.

There is a similar notion of a slide diffeomorphism obtained by gluing $E$ to $F_k$, for some $k = n+1,\dots,2g$  corresponding to an $S^1 \times S^2$ factor, and sliding $E$ around an arc in the resulting $\widehat{M}$, and back to itself, as before.
\end{definition}

\begin{definition}
    We define the \emph{transposition diffeomorphism} for a pair of prime factors $P_{i_1}$ and $P_{i_2}$ such that there is a diffeomorphism $\psi \colon P_{i_1} \cong P_{i_2}$ sending $D_{i_1}$ to $D_{i_2}$. Remove $P_{i_1}$ and $P_{i_2}$ from $M$, fill in the boundary component $F_{i_1}$ with a 3-ball $E_1$, and $F_{i_2}$ with another 3-ball $E_2$.  Consider an ambient isotopy of $S^3$ that switches the positions of $E_1$ and $E_2$. This results in a diffeomorphism of $\varphi\colon B \to B$ that fixes all boundary components apart from $F_{i_1}$ and $F_{i_2}$, which are switched. Extend by the identity everywhere apart from on $P_{i_1}$ and $P_{i_2}$. Send $P_{i_1} \sm \mathring{D}_{i_1}$ to $P_{i_2} \sm \mathring{D}_{i_2}$ using $\psi$, and send $P_{i_2} \sm \mathring{D}_{i_2}$ to $P_{i_1} \sm \mathring{D}_{i_1}$ using $\psi^{-1}$.
    \end{definition}

\begin{definition}
    We define \emph{flips of $S^1 \times S^2$ factors}. Remove $S_j \times I$ from $M$, fill in the boundary component $F_{n+2j-1}$ with a 3-ball $E_1$, and $F_{n+2j}$ with another 3-ball $E_2$.  Consider an ambient isotopy of $S^3$ that switches the positions of $E_1$ and $E_2$. This results in a diffeomorphism of $\varphi\colon B \to B$ that fixes all boundary components apart from $F_{n+2j-1}$ and $F_{n+2j}$, which are switched. Extend by the identity everywhere apart from on $S_j \times I$. We may assume that $\varphi$ determines the map $(x,t) \mapsto (-x,1-t)$ on $S_j \times \{0,1\}$. This is because the two 2-spheres are switched, and because they are glued to $B$ with maps of degree $(-1)^t$, in order to have an oriented manifold. We can therefore extend using the formula $(x,t) \mapsto (-x,1-t)$ on all of $S_j \times I$.
\end{definition}

\begin{remark}
This is essentially the element $(1,1,0) \in \pi_0 \Diff(S^1 \times S^2) \cong (\Z/2)^3$, the composition of the reflections in both of the factors, in the $j$th $S^1 \times S^2$ summand. We could have considered the $S^1 \times S^2$ factors on an equal footing with the irreducible summands in the model for $M$, and then we would not need to list the flips and Gluck twists separately. But, if we did this, we would miss the slide diffeomorphisms from the last paragraph of \cref{defn:slide-diffeo}. So it is more convenient to do it this way.
\end{remark}

Finally, we discuss the diffeomorphisms coming from each prime factor. Here we include $S^1 \times S^2$ factors, and so the Gluck twists are included here. In particular, note that the groups $\pi_0 \Diff_{D_i \cup \partial P_i}(P_i) = \pi_0 \Diff_\partial(P_i \sm \mathring{D}_i)$ appear in the statement of \cref{conjecture:MCG-reducible-generators}, rather than simply $\pi_0\Diff_{\partial P_i}(P_i)$, which is what one can hope to be able to compute for prime 3-manifolds from our survey thus far.
The two mapping class groups are related as follows.
The fibre sequence
\[\Diff_{D_i \cup \partial P_i}(P_i) \to \Diff_{\partial P_i}(P_i) \to \Emb(D^3,\mathring{P}_i)\]
yields an exact sequence, which is a portion of the long exact sequence in homotopy groups of the fibration:
\[\pi_1\Emb(D^3,\mathring{P}_i) \xrightarrow{\partial} \pi_0 \Diff_{D_i \cup \partial P_i}(P_i) \to \pi_0 \Diff_{\partial P_i}(P_i) \to \{1\}.\]
The additional diffeomorphisms in $\pi_0 \Diff_{D_i \cup \partial P_i}(P_i)$, that are in the image of $\partial$ and so die in $\pi_0 \Diff_{\partial P_i}(P_i)$, are point-pushing maps and sphere twists around $\partial D_i$.
We record this in the following proposition, before defining sphere twists.

\begin{proposition}
    The kernel of the map $\pi_0\Diff_{D_i \cup \partial P_i}(P_i) \to \pi_0\Diff_{\partial P_i}(P_i)$ is generated by a sphere twist along $\partial D_i$ and point-pushing diffeomorphisms.
\end{proposition}

\begin{definition}\label{def:sphere twist}
    For an embedded 2-sphere $S$ in $M$, define the \emph{sphere twist about $S$} as follows. Recall that $\pi_1(\SO(3)) \cong \mathbb{Z}/2$ is generated by a loop $\ell \colon [0, 1] \to \SO(3)$ which rotates $\mathbb{R}^3$ about an axis by a full turn. Identify $S$ with $S^2 \subseteq \mathbb{R}^3$ and define a diffeomorphism supported on a bicollar neighbourhood $S \times I \subseteq M$ by $r(s, t) = (\ell(s) \cdot t, t)$, to obtain the \emph{sphere twist about} $S$, denoted $T_S$.
\end{definition}

 Whether or not the sphere twists and point-pushing maps are nontrivial in $\pi_0\Diff_{D_i \cup \partial P_i}(P_i)$ depends on $P_i$, and it requires some work to determine precisely which 3-manifolds and which point pushes are nontrivial.

The sphere twists here correspond in $M$ to sphere twists on separating spheres in $M$.
We also have sphere twists on $S_j \times \{1/2\}$ for $j=1,\dots,g$, the Gluck twists on each of the $S^1 \times S^2$ factors.

\subsection{The kernel of $\rho$ }\label{subsec:kernel-fo-rho}

In this section we examine the kernel of the map which takes mapping classes to outer automorphisms of the fundamental group. For closed, orientable 3-manifolds other than $S^3$ and $\RP^3$, this kernel is generated entirely by sphere twists, as in \cref{thm:kernel-rho}.  This leads us to the Laudenbach short exact sequence, which we will study in \cref{subsection-laudenbach-seq}.   We will describe how to extend \cref{thm:kernel-rho} to all compact, orientable 3-manifolds.

Let $M$ be a compact, orientable 3-manifold. Recall the homomorphism
\[\rho \colon \pi_0\Diff(M) \to \Out(\pi_1(M))\]
from \eqref{eqn-MGC to out} arising from isotoping to fix a basepoint, and then taking the induces action on $\pi_1(M)$.
Let
\[\rho^+ \colon \pi_0\Diff^+(M) \to \Out(\pi_1(M))\]
denote its restriction to the orientation-preserving mapping class group of $M$.

We survey work examining this homomorphism; in particular, we explore the extent to which $\rho$ is an isomorphism and its splittings.
Recall (\cref{def:sphere twist}) that the sphere twist $T_S$ about an embedded 2-sphere $S \subseteq M$ is the isotopy class of the diffeomorphism supported on a product $S \times I \subseteq M$ that rotates the slices $S \times \{t\}$ according to a loop generating $\pi_1(\operatorname{SO}(3)) \cong \mathbb{Z}/2$. Similarly, for a properly embedded 2-disc $D \subseteq M$, define the \emph{twist about $D$} to be the isotopy class of the diffeomorphism supported on a product $D \times I \subseteq M$ which rotates the slices $D \times \{t\}$ according to a generator of $\pi_1(\operatorname{SO}(2)) \cong \mathbb{Z}$. Here we are identifying $D$ with $D^2 \subseteq \mathbb{R}^2$ and taking the generator of $\pi_1(\operatorname{SO}(2)) \cong \mathbb{Z}$  to be the loop corresponding to the rotation of $\mathbb{R}^2$ about the origin by one full turn.

Seven types of diffeomorphisms which are seen to give elements of $\ker(\rho)$ are:
\begin{enumerate}
    \item sphere twists, about some embedded sphere $S \subseteq M$;
    \item twists along properly embedded discs $D \subseteq M$;
    \item permuting two boundary spheres;
    \item sliding a boundary sphere around a loop with endpoints on the sphere;
 \item reflecting each fibre of an $I$-bundle over a surface about its midpoint;
 \item reflection of $S^3$;
 \item the diffeomorphism $\sigma_-$ of $\RP^3= L(2,1)$ (\cref{sec-lens spaces}).
\end{enumerate}

\begin{theorem}[\cite{Hatcher_Wahl_MCG3Manifold}, Prop.~2.1]\label{prop:generators of kernel of action of MCG on Out(pi1)}
Let $M$ be a compact, orientable 3-manifold.
\begin{enumerate}[(i)]
\item The kernel of $\rho$ is generated by the seven families of diffeomorphisms listed.
\item  The kernel of $\rho^+$ is generated by the families of diffeomorphisms (1)--(4).
\item If moreover $M$ is closed, then the kernel of $\rho$ is generated by the families of diffeomorphisms (1), (6), and (7).
\item If $M$ is closed, then the kernel of $\rho^+$ is generated by sphere twists as in (1).
\item If $M$ is closed, irreducible, and not equal to $S^3$ or $\RP^3$, then $\rho$ is injective.
\item If $M$ is closed and irreducible, then $\rho^+$ is injective.
\end{enumerate}
\end{theorem}

\begin{proof}
Hatcher--Wahl \cite{Hatcher_Wahl_MCG3Manifold} describe how to extract (ii) from the literature. The necessity and sufficiency of the remaining three diffeomorphisms is from Hong--McCullough \cite[Section~4]{HongMcCullough2013}.
\begin{itemize}
    \item The statement for reducible $M$ follows from the statement for irreducible $M$ together with a result of McCullough \cite[Theorem 1.5]{McCulloughTopAlgAuts3mfds90}. An alternative proof can be found in \cite[Appendix]{Hatcher_Wahl_MCG3Manifold}. Note that for irreducible $M$ with (possibly empty) incompressible boundary, the statement says that $\rho$ is injective.
    \item Waldhausen \cite[Theorem 7.1]{Waldhausen1968} shows the result for Haken manifolds that are closed or have incompressible boundary (\cref{thm:Waldhausen} is a corollary of this combined with \cite[Corollary 6.5]{Waldhausen1968}).
    \item Irreducible $M$ with nonempty compressible boundary satisfy the statement by work of McCullough--Miller \cite[Theorem 6.2.1]{McCulloughMillerCompressibleBoundary86}.
    \item Closed hyperbolic manifolds are $K(\pi, 1)$s, so the proposition follows as a corollary of work of Gabai \cite{Gabai1997} (\cref{thm:Gabai 97 main thm 0.9}) and Gabai--Meyerhoff--Thurston \cite{GabaiMeyerhoffThurston2003} (\cref{thm:GMT}) for such $M$.
    \item Closed non-Haken Seifert manifolds with infinite fundamental group are $K(\pi, 1)$s, whence the statement follows.
    \item Closed non-Haken Seifert manifolds with finite fundamental group are spherical manifolds, and now the statement follows by work of McCullough \cite[Theorem 3.1]{McCullough2002} (\cref{pi0Smale}).
  \item According to \cite[Section~4]{HongMcCullough2013}, to deal with orientation-reversing diffeomorphisms it suffices to add the diffeomorphisms of type (5), (6), and (7).
\end{itemize}
\end{proof}

\subsection{The Laudenbach sequence}\label{subsection-laudenbach-seq}

We describe the Laudenbach sequence, and present a proof that it splits.
For the rest of the section, we additionally assume $M$ is closed, and restrict to orientation-preserving mapping classes $\pi_0\Diff^+(M)$. So $\ker(\rho^+)$ is entirely generated by sphere twists.

Let $\Twist(M)$ denote the subgroup of $\pi_0\Diff^+(M)$ generated by sphere twists; denote by $\Twistns(M)$ the subgroup generated only by those about non-separating spheres.
Observe that $\Twist(M)$ is a normal subgroup of $\pi_0\Diff(M)$: for any $f \in \pi_0\Diff(M)$ and a sphere twist $T_S \in \Twist(M)$, we have
\begin{equation} \label{eqn:twist subgroup is normal computation}
    fT_Sf^{-1} = T_{f(S)}.
\end{equation}
If $S$ and $S'$ are two embedded spheres, there is a homotopy, and hence isotopy~\cite{Laudenbach1974} of $S'$ to $S''$ that is disjoint from $S$. It follows that there is an isotopy of $T_{S}$ to $T_{S''}$. The latter has support disjoint from that of $T_{S'}$, and hence $T_S$ and $T_{S''}$ commute. Thus $\Twist(M)$ is abelian.

\begin{remark}\label{rem: twist subgroup}
   McCullough~\cite{McCulloughTopAlgAuts3mfds90} showed that $\Twist(M)\cong (\mathbb{Z} / 2)^d$ for some $d$, and Chen--Tshishiku \cite[Thm 2.4]{ChenTshishiku2025} explicitly compute this group for any closed, oriented 3-manifold $M$, where $d$ is easily read off from the prime decomposition of $M$. In particular $d$ depends on the number of prime components that are lens spaces.
\end{remark}

By \cref{prop:generators of kernel of action of MCG on Out(pi1)}, we have a short exact sequence
\begin{equation}\label{eqn:gen Laudenbach seq}
    1 \to \Twist(M) \to \pi_0\Diff^+(M) \to G \to 1,
\end{equation}
where $G \leq \Out(\pi_1(M))$ is defined to be the image of $\rho^+$. See \cite[Proposition 2.2]{Hatcher_Wahl_MCG3Manifold} for conditions which ensure $G = \Out(\pi_1(M))$.  This is the \emph{generalised Laudenbach sequence}.

\begin{remark}
Note that $\pi_0\Diff^+(M)$ is an index $\leq 2$ subgroup of $\pi_0\Diff(M)$. Alternatively we could replace $G$ with $G \times \Z/2$ in \eqref{eqn:gen Laudenbach seq} if $M$ admits an orientation-reversing diffeomorphism, to obtain a short exact sequence
\[ 1 \to \Twist(M) \to \pi_0\Diff(M) \to G \times \Z/2 \to 1.\]
In what follows we will restrict to studying $\pi_0\Diff^+(M)$.
\end{remark}

Let $M_n \coloneqq (S^1 \times S^2)^{\#n} $, the connected sum of $n$ copies of $S^1 \times S^2$. We have $\pi_1(M_n) \cong F_n$, the free group on $n$ letters. Work of Whitehead \cite{WhiteheadSetsElementsFreeGp36, WhiteheadEquivSetsFreeGp36} implies $\rho^+$ is surjective in this case; Laudenbach \cite{Laudenbach1973, Laudenbach1974} originally identified $\ker(\rho^+) = \Twist(M_n)$ and proved that this kernel is generated by the sphere twists about the $n$ core spheres $\{\ast\} \times S^2$ of connected summands. The above short exact sequence in this case is the original \emph{Laudenbach sequence}:
\[1 \to \Twist(M_n) \to \pi_0\Diff^+(M_n) \to \Out(F_n) \to 1.\]

In light of \cref{prop:generators of kernel of action of MCG on Out(pi1)}, we will not survey Laudenbach's proof. An overview can also be found in \cite[Section 2]{BrendleBroaddusPutman2023}. This follows the outline of Laudenbach's proof but simplifies one step: that a basepoint-preserving diffeomorphism of a closed oriented 3-manifold which induces the identity on $\pi_1$ also does so on $\pi_2$ (yet another proof of this statement is provided by Hatcher--Wahl \cite[Appendix]{Hatcher_Wahl_MCG3Manifold}).

\subsubsection{Splitting the generalised Laudenbach sequence}\label{subsubsec:splitting Laudenbach sequence}

Brendle--Broaddus--Putman \cite{BrendleBroaddusPutman2023} show that Laudenbach's sequence splits; moreover, they give an explicit description of the image of the splitting $\Out(F_n) \to \pi_0\Diff^+(M_n)$. Their proof generalises to show the following.

\begin{theorem}
\label{thm: generalised version of BBP23 main thm}
    The mapping class group of any closed, oriented 3-manifold $M$ decomposes as a semi-direct product
    \[\pi_0\Diff^+(M) \cong \Twistns(M) \rtimes H,\]
    where $H$ is a subgroup of $\pi_0\Diff^+(M)$.
\end{theorem}

We present their proof in this level of generality. The main theorem of \cite{BrendleBroaddusPutman2023} then follows as a corollary of this proof in the case $M = M_n$. 

\begin{corollary}[\cite{BrendleBroaddusPutman2023}, Theorem A]\label{cor-BBP splitting}
    The mapping class group of $M_n$ splits as a semi-direct product
    \[\pi_0\Diff^+(M_n) \cong \Twist(M_n) \rtimes \Out(F_n),\]
    where the image of the splitting $\Out(F_n) \to \pi_0\Diff^+(M_n)$ is the stabiliser of the homotopy class of a trivialisation of the tangent bundle of $M_n$. Moreover, $\Twist(M_n) \cong H^1(M_n;\; \mathbb{Z}/2)$ as a $\pi_0\Diff^+(M_n)$-module.
\end{corollary}

The usefulness of the general case, \cref{thm: generalised version of BBP23 main thm}, may be limited without an explicit description of the subgroup $H \leq \pi_0\Diff^+(M)$ appearing in the statement, analogous to $H = \Out(F_n)$ in the case $M = M_n$. Fully generalising to a splitting of \eqref{eqn:gen Laudenbach seq} seems delicate, since there are (separating) sphere twists that are homotopic but not isotopic to the identity \cite{FriedmanWittHomotopyNotIsotopy86, HendriksApplicationsObstruction77} (note that all separating sphere twists are trivial in $M_n$, so $\Twist(M_n) = \Twistns(M_n)$).

Central to Brendle--Broaddus--Putman's proof is the notion of a \emph{crossed homomorphism}. For two groups $G$ and $H$, with $G$ acting on $H$ on the right by conjugation (denoted by $h \mapsto h^g$ for $g \in G$, $h \in H$), a crossed homomorphism is a set map $\lambda \colon G \to H$ such that
\[\lambda(g_1g_2) = \lambda(g_1)^{g_2}\lambda(g_2).\]
Note that if the $G$-action is trivial, then this is simply a homomorphism. The following lemma is standard (cf.~\cite[Proposition 2.1]{BrownCohomologyofGroups1stEd82}); a proof is given in \cite[Lemma 3.1]{BrendleBroaddusPutman2023}.

\begin{lemma}\label{lem:crossed homomorphisms and split SES}
    Let $A$ be an abelian normal subgroup of a group $G$ with $G$ acting on $A$ on the right by conjugation. Let $Q = {G}/{A}$. Then the short exact sequence
    \[1 \to A \to G \to Q \to 1\]
    splits if and only if there exists a crossed homomorphism $\lambda \colon G \to A$ that restricts to the identity on $A$.

    Moreover, if such a $\lambda$ exists, then we can choose a splitting $Q \to G$ whose image is $\ker(\lambda)$, whence $G = A \rtimes \ker(\lambda)$.
\end{lemma}

Let $M$ be a closed, oriented 3-manifold. We will construct the \emph{twisting crossed homomorphism}
\[\mathfrak{T} \colon \pi_0\Diff^+(M) \to \Twistns(M).\]

Let $\mathrm{Fr}(TM)$ denote the principle $\GL^+(3,\R)$-bundle of oriented frames of $TM$, whose elements are pairs $(p, \tau)$ for $p \in M$ a point and $\tau \colon \mathbb{R}^3 \to T_pM$ an orientation-preserving linear isomorphism. The right action of $\GL^+(3,\R)$ on $\mathrm{Fr}(TM)$ is by pre-composition.

An \emph{oriented trivialisation} of the tangent bundle $TM$ (which is trivial, since $M$ is 3-dimensional and oriented) is a continuous section $\sigma \colon M \to \mathrm{Fr}(TM)$. Let $\mathrm{Triv}(M)$ be the set of all such sections. Let $C(M, \GL^+(3,\R))$ denote the topological group of continuous maps $M \to \GL^+(3,\R)$, which acts simply transitively on $\mathrm{Triv}(M)$ on the right via
\begin{eqnarray*}
    \mathrm{Triv}(M) \times C(M, \GL^+(3,\R)) & \to & \mathrm{Triv}(M) \\
    (\sigma, \phi) & \mapsto & \big( p \mapsto \sigma(p)\cdot\phi(p)\big),
\end{eqnarray*}
where $\sigma(p) \cdot \phi(p)$ denotes the composition of $\phi(p) \in \GL^+(3,\R)$ with the orientation-preserving linear isomorphism $\sigma(p) \colon \mathbb{R}^3 \to T_pM$ (we use a dot to denote this composition, evoking products of matrices).

The derivative of any $f \in \mathrm{Diff}^+(M)$ has an induced map
\begin{eqnarray*}
    ({\rm D}f)_* \colon \mathrm{Fr}(TM) & \to & \mathrm{Fr}(TM) \\
    \big(\tau \colon \mathbb{R}^3 \to T_{f(p)}M\big) & \mapsto & \big(({\rm D}_pf) \circ \tau \colon \mathbb{R}^3 \to T_{f(p)}M \big).
\end{eqnarray*}
Using this, we define two actions of $\mathrm{Diff}^+(M)$:
\begin{itemize}
    \item a right action on $\mathrm{Triv}(M)$ by
    \begin{eqnarray*}
        \mathrm{Triv}(M) \times \mathrm{Diff}^+(M) & \to & \mathrm{Triv}(M) \\
        (\sigma, f) & \mapsto & \sigma^f \coloneqq \big({\rm D}f^{-1}\big)_* \circ \sigma \circ f;
    \end{eqnarray*}
    \item a right action on $C(M, \GL^+(3,\R))$ by
    \begin{eqnarray*}
        C(M, \GL^+(3,\R)) \times \mathrm{Diff}^+(M) & \to & C(M, \GL^+(3,\R)) \\
        (\phi, f) & \mapsto & \phi^f \coloneqq \phi \circ f.
    \end{eqnarray*}
\end{itemize}
One may verify that these two actions are related to the action of $C(M, \GL^+(3,\R))$ on $\mathrm{Triv}(M)$ by the formula
\[(\sigma \cdot \phi)^f = \sigma^f \cdot \phi^f.\]

Now fix an oriented trivialisation $\sigma_0 \in \mathrm{Triv}(M)$. Since $C(M, \GL^+(3,\R))$ acts on $\mathrm{Triv}(M)$ simply transitively, we have that for any $f \in \mathrm{Diff}^+(M)$ there exists a unique $\phi_f \in C(M, \GL^+(3,\R))$ such that
\[\sigma_0^f = \sigma_0 \cdot \phi_f.\]
Define
\begin{eqnarray*}
    \mathcal{D} \colon \mathrm{Diff}^+(M) & \to & C(M, \GL^+(3,\R)) \\
    f & \mapsto & \phi_f^{-1},
\end{eqnarray*}
which can be verified to be a crossed homomorphism by a short computation.

We now pass to homotopy, adopting the following notation.
\begin{itemize}
    \item for $f \in \Diff^+(M)$, let $[f] \in \pi_0\Diff^+(M)$ denote its isotopy class;
    \item for $\sigma \in \mathrm{Triv}(M)$, let $[\sigma]$ denote its homotopy class and denote the set of these by $\mathrm{HTriv}(M)$;
    \item for $\phi \in C(M, \mathrm{GL}^+_3(M))$, let $[\phi]$ denote its homotopy class in $\left[M, \mathrm{GL}^+_3(M)\right]$; this set inherits a group structure from $C(M, \mathrm{GL}^+_3(M))$.
\end{itemize}
Maintaining the same notation for the above actions after passing to homotopy, we have the relationship
\[([\sigma] \cdot [\phi])^{[f]} = [\sigma]^{[f]} \cdot [\phi]^{[f]}\]
for each $f \in \pi_0\Diff^+(M)$, $[\sigma] \in \mathrm{HTriv}(M)$, and $[\phi] \in \left[M, \mathrm{GL}^+_3(M)\right]$. The crossed homomorphism $\mathcal{D}$ therefore descends to a crossed homomorphism
\[\mathfrak{D} \colon \pi_0\Diff^+(M) \to \left[M, \GL^+(3,\R)\right],\]
which we call the \emph{derivative crossed homomorphism}.

Now, we have $\pi_1(\GL^+(3,\R)) \cong \mathbb{Z}/2$, so composing $\mathfrak{D}$ with the $\pi_1$-functor yields the \emph{twisting crossed homomorphism}
\[\mathfrak{T} \colon \pi_0\Diff^+(M) \to \mathrm{Hom}(\pi_1(M), \mathbb{Z}/2) = H^1(M; \mathbb{Z}/2).\]

Since the subgroup $\Twist(M) \leq \pi_0\Diff^+(M)$ acts trivially on $\pi_1(M)$, the reader may check (from the definition of crossed homomorphism) that the restriction of $\mathfrak{T}$ to $\Twist(M)$ is a \textit{bona fide} group homomorphism.

\begin{lemma}\label{lem:twisting crossed homo on sphere twists}
    Let $S$ be an embedded 2-sphere in $M$. Then $\mathfrak{T}(T_S)$ is the Poincar\'e dual to $[S] \in H_2(M; \mathbb{Z}/2)$.
\end{lemma}

Now, choosing a generating set of $\Twistns(M)$ such that the corresponding spheres are linearly independent in $H_2(M; \mathbb{Z}/2)$, \cref{lem:twisting crossed homo on sphere twists} implies the following.

\begin{corollary}\label{cor:twisting crossed homo is an iso on Twistns}
    The twisting crossed homomorphism $\mathfrak{T}$ restricts to an isomorphism of the subgroup $\Twistns(M) \leq \pi_0\Diff^+(M)$ onto its image.
\end{corollary}

Since $\mathfrak{T}$ is a crossed homomorphism, the isomorphism in \cref{cor:twisting crossed homo is an iso on Twistns} is actually one of $\pi_0\Diff^+(M)$-modules, where $\pi_0\Diff^+(M)$ acts on its normal subgroup $\Twistns(M)$ by conjugation.

We now have all the ingredients to deduce the claimed splitting. Put $Q \coloneqq {\pi_0\Diff^+(M)}/{\Twistns(M)}$ and apply \cref{lem:crossed homomorphisms and split SES} to the twisting crossed homomorphism $\mathfrak{T}$ to deduce that the short exact sequence
\[1 \to \Twistns(M) \to \pi_0\Diff^+(M) \to Q \to 1\]
splits, and moreover that we can choose a splitting $Q \to \pi_0\Diff^+(M)$ whose image is $\ker(\mathfrak{T})$; now
\[\pi_0\Diff^+(M) \cong \Twistns(M) \rtimes \ker(\mathfrak{T}),\]
as required.

\subsubsection{Explicit description of splitting in the case $M = M_n$}\label{subsubsec:geom description of splitting when M = M_n}

Specialise to the case $M = M_n$. Since $\Twist(M_n)$ is generated by the twists around the core spheres of the connected summands, we have just proved that Laudenbach's sequence splits:
\[\pi_0\Diff^+(M_n) = \Twist(M_n) \rtimes \Out(F_n).\]
In addition, it follows from the proof that $\Twist(M_n) \cong H^1(M_n; \mathbb{Z}/2)$ as $\pi_0\Diff^+(M_n)$-modules.

In this case, we can obtain a concrete description of the splitting. Fix an oriented trivialisation $\sigma_0 \in \mathrm{Triv}(M_n)$, and let $\mathfrak{T}$ be the associated twisting crossed homomorphism. By construction, any element of $\pi_0\Diff^+(M_n)$ which stabilises the homotopy class $[\sigma_0]$ is contained in $\ker(\mathfrak{T})$. It turns out that the converse is also true: $\ker(\mathfrak{T})$ fixes $[\sigma_0]$. To prove this, one can use the fact that the abelianisation of $\Out(\pi_1(M_n)) = \Out(F_n)$ is torsion, and conclude that the restriction of the derivative crossed homomorphism
\[\mathfrak{D} \colon \pi_0\Diff^+(M_n) \to \left[M_n, \GL^+(3,\R)\right]\]
to $\ker(\mathfrak{T})$ is trivial.

We can thus obtain an explicit description of the image of the splitting $\Out(F_n) \to \pi_0\Diff^+(M_n)$: for each $\theta$ in $\Out(F_n)$ there is a unique preimage that lies in the stabiliser of $[\sigma_0]$.

\section{Finiteness properties 
}\label{section:finiteness-properties}

We survey some finiteness properties which mapping class groups of 3-manifolds are known to enjoy.

\subsection{Survey of finiteness properties}
Throughout the subsection 3-manifolds are assumed to be compact and orientable. Nonempty boundary is permitted.

\begin{definition}\label{def:type FL, duality group, type F}
    We say that a group $G$ is:
    \begin{itemize}
        \item of \emph{type FL} if there is a finite-length resolution of $\mathbb{Z}$ as a $\mathbb{Z}G$-module by free finitely-generated $\mathbb{Z}G$-modules;
        \item a ($\mathbb{Z}$-)\emph{duality group} if there exists a (right) $\mathbb{Z}G$-module $D$ and $n \in \mathbb{Z}_{\geq 0}$ such that for all $k$ and all $\mathbb{Z}G$-modules $M$, we have natural isomorphisms
        \[H^k(G;M) \cong H_{n-k}(G;D \otimes_{\mathbb{Z}} M)\]
        (in this case say that the \emph{dimension} of $G$ is $n$);
        \item \emph{geometrically finite} (or \emph{of type F}) if there exists an aspherical finite CW-complex with fundamental group $G$, i.e.~there is a finite $K(G, 1)$-complex.
    \end{itemize}
    For any of these properties $P$, we say that $G$ is \emph{virtually P} if it has a finite-index subgroup which is $P$ (and we say $G$ is of \emph{type VFL} if it is virtually of type \emph{FL}).
\end{definition}

If a group $G$ is finitely presented, then we can give useful topological interpretations of being of type \emph{FL} or a duality group.

\begin{proposition}\label{prop:type FL and duality topologically for G f.p.}
\leavevmode
    \begin{enumerate}[(i)]
        \item A group $G$ is geometrically finite if and only if it is finitely presented and of type \emph{FL}.
        \item Suppose $G$ is finitely presented. Let $X$ be a compact $m$-manifold with nonempty boundary which is a $K(G, 1)$, and let $\widetilde{X}$ be its universal cover. Then $G$ is a duality group if and only if for some $q$ we have (reduced) homology $H_i(\partial \widetilde{X}) = 0$ for all $i \neq q$ and $H_q(\partial \widetilde{X})$ is torsion-free. In this case, the dimension of $G$ is $m-q-1$.
    \end{enumerate}
\end{proposition}

\begin{proof}
The first item is \cite[Prop.~10]{SerreCohomologieGroupesDiscrets71}, while the second was proved in \cite{BieriEckmannDualityGroups73}.
\end{proof}

The above finiteness properties for mapping class groups of 2-manifolds were established by the mid-1980s: McCool \cite{McCoolFPSubgroupsofAutFn75} showed finite presentability, Harvey \cite{HarveyGeometricStructureSurfaceMCGs79, HarveyBoundaryStructureModularGp81} showed they were of type \emph{VFL}, and Harer \cite{HarerVCDofMCGSurface86} proved they are virtual duality groups (and computed their \emph{virtual cohomological dimension} (cf. \cite[Ch. VIII, Section 11]{BrownCohomologyofGroups1stEd82}).

Moving to mapping class groups of 3-manifolds, we have the following result of McCullough.

\begin{theorem}[\cite{McCullough1991}]\label{thm:Haken MCGs are virtually geom finite}
    Let $M$ be a compact, orientable, irreducible, Haken 3-manifold. Then $\pi_0\Diff(M)$ is finitely presented and of type VFL $($and hence is virtually geometrically finite$)$.

    If additionally $M$ is closed or boundary-irreducible, then $\pi_0\Diff(M)$ is a virtual duality group.
\end{theorem}

In some cases, the stronger finiteness property of actually being finite holds, as mentioned already in \cref{thm:Johannson-finiteness-Haken}. We recall that theorem here for convenience.

\begin{theorem}[Johannson, \cite{Johannson1979}]
    Let $M$ be a compact, orientable, irreducible, boundary irreducible, and Haken 3-manifold.
Suppose also that every incompressible annulus or torus in $M$ is boundary parallel. Then $\pi_0 \Diff(M)$ is finite. 
\end{theorem}

Beyond looking at the mapping class group, one can ask about finiteness properties of the homotopy type of $\BDiff(M)$, the classifying space of the diffeomorphism group.
The first result in this direction that we discuss is due to Hatcher--McCullough.

\begin{theorem}[\cite{HatcherMcCulloughFinitenessBDiffM97}]
    Let $M$ be a compact, orientable, irreducible 3-manifold and let $R$ be a nonempty union of connected components of $\partial M$, including all the compressible ones. Then $\BDiff_R(M)$ has the homotopy type of an aspherical finite CW-complex.  In particular, $\BDiff_{\partial}(M)$ is homotopy finite, i.e.~has the homotopy type of a finite CW-complex..
\end{theorem}

The final sentence in this statement is a positive answer, in the case $M$ is oriented and irreducible, to a conjecture of  Kontsevich \cite[Problem 3.48]{KirbyProblemList97}, who asked whether $\BDiff_{\partial}(M)$ is homotopy finite for every compact 3-manifold with nonempty boundary~$M$.

Irreducible 3-manifolds with nonempty boundary are Haken, so asphericity of $\BDiff_{\partial}(M)$ was known by previous work of Hatcher \cite{Hatcher1976, Hatcher1983} and Ivanov \cite{Ivanov1976}. Therefore,  the proof of Hatcher--McCullough's result can be reduced to proving that the mapping class group $\pi_0(\Diff_{\partial}(M))$ is geometrically finite; their result can hence be seen as a refinement of McCullough's virtual geometrical finiteness \cref{thm:Haken MCGs are virtually geom finite} in this setting.

In fact, recent work of Boyd--Bregman--Steinebrunner \cite[Theorem 6.1]{BoydBregmanSteinebrunnerModuliSpacesFinite24} proved
Kontsevich's conjecture for orientable 3-manifolds.

\begin{theorem}[\cite{BoydBregmanSteinebrunnerModuliSpacesFinite24}]
   Let $M$ be a compact, orientable, 3-manifold with nonempty boundary. Then $\BDiff_{\partial}(M)$ is homotopy finite.
\end{theorem}

They then applied this theorem to deduce the following finiteness result.

\begin{theorem}[\cite{BoydBregmanSteinebrunnerModuliSpacesFinite24}, Theorem 6.12]
    Let $M$ be a compact, orientable, 3-manifold. Then $\BDiff(M)$ is of finite type, i.e.~homotopy equivalent to a CW complex with finite $n$-skeleton for each $n$.
\end{theorem}

Note that in general $\BDiff_{\partial}(M)$ is not aspherical, in contrast to the irreducible case handled by Hatcher--McCullough. For example, Kalliongis--McCullough \cite{KalliongisMcCulloughIsotopies3Manifolds96} found 3-manifolds $M$ for which $\pi_2(\BDiff_{\partial}(M)) \cong \pi_1(\Diff_{\partial}(M))$ is not finitely generated.

 \subsection{Finite presentation of reducible 3-manifold mapping class groups}\label{subsec:finite presentation of 3-mfd MCGs}

Now we discuss finite presentations of mapping class groups of reducible $3$-manifolds in detail, assuming the prime factors have finitely presented mapping class groups.  Throughout the subsection, let $M$ be a compact, orientable, reducible $3$-manifold, with possibly nonempty boundary.  Our main focus is the following theorem of Hatcher and McCullough.

    \begin{theorem}[{\cite[Thm.~4.1]{HatcherMcCullough1990}}]
    \label{thm-finite presentation}
        Let $M$ be a 3-manifold with nontrivial prime decomposition. If the mapping class group of each irreducible prime factor is finitely presented, then $\pi_0\mathrm{Diff}(M)$ is also finitely presented.
    \end{theorem}

We present their proof.  This uses the following condition proved by K.~Brown.

    \begin{theorem}[{\cite[Thm.~4]{Brown1984Presentations}}]
    \label{thm-kbrown}
        Let $G$ be a group that acts simplicially on a simply-connected simplicial complex such that each vertex stabiliser is finitely presented, each edge stabiliser is finitely generated, and the quotient has finite 2-skeleton. Then $G$ is a finitely presented group.
    \end{theorem}

    Hence to prove \cref{thm-finite presentation} it is (more than) sufficient to do the following.
    \begin{enumerate}[(a)]
        \item Construct a simply-connected simplicial complex $C$ and a simplicial action of $\pi_0\Diff(M)$ on $C$.
        \item Show that the quotient of $C$ by the action of $\pi_0\Diff(M)$ is finite.
        \item Show that the stabiliser of each simplex of $C$ is finitely presented.
    \end{enumerate}

We will deal with each of these in turn, hence completing the proof of \Cref{thm-finite presentation}.  More precisely, we first perform steps (a) and (b), and then provide an inductive argument on the number of prime summands in our 3-manifold to establish (c).  Starting this induction is precisely where we need the assumption that each prime factor has finitely presented mapping class group.

\subsubsection{The sphere complex $S(M)$}  We first deal with (a), and we begin by defining the relevant simplicial complex.

\begin{definition}
\label{spheresimplicial}
   The \emph{sphere complex} $S(M)$ of $M$ is a simplicial complex whose vertices are isotopy classes of embedded 2-spheres in $M$ that do not bound 3-balls---such spheres are called \emph{essential}. A collection of such isotopy classes $[S_0],\dots,[S_n]$ spans an $n$-simplex if and only if they are pairwise distinct and have pairwise disjoint representatives. In this case we call the set of embedded representatives $\Sigma\subseteq M$ a \emph{sphere system} representing the simplex.
\end{definition}

There is then an eminent candidate for an action of $\pi_0\Diff(M)$ on $S(M)$.  Given $[\varphi]\in S(M)$ and $v\colon S^2\hookrightarrow M$ a vertex, we set $[\varphi]\cdot v=\varphi\circ v$.  Since vertices are only isotopy classes of embeddings, this gives a well-defined action on vertices.  Furthermore, since diffeomorphisms of $M$ preserve disjointness of embedded spheres, it follows that this action extends to a simplicial action on $S(M)$.

      \begin{theorem}[{\cite[Thm~1.1]{HatcherMcCullough1990}}]
      \label{simplyconnectedsim}
            The geometric realisation of $S(M)$ is simply-connected.
        \end{theorem}

In fact it was shown by Hatcher and Wahl~\cite[Section 3]{HatcherWahl2005} that if $M$ is not an irreducible manifold nor a punctured irreducible manifold, then $S(M)$ is contractible.

\begin{proof}[Proof sketch]
To see that $S(M)$ is connected, we observe that one can construct a path from a chosen basepoint vertex $S_*$ to any vertex $[S]$ by performing a sequence of surgeries along circles in $S\cap S_*$ to build a path from $[S]$ to another vertex $[S_1]$ such that $S_1\cap S_*=\emptyset$, and then concatenating with the edge $[S_*,S_1]$.

A loop in $S(M)$ is determined by a sequence of spheres $(S_0,\dots,S_k)$ such that $S_i \cap S_{i+1} = \emptyset$ for $i=0,\dots,k-1$ and $S_0 = S_k = S_*$.
Hatcher--McCullough construct a null homotopy by using surgery to introduce intermediate vertices,  $S_i'$, such that $S_i'$ has fewer circles of intersection with $S_*$ than $S_i$, and also that $S_i'$ is still disjoint from $S_{i \pm 1}$. The resulting path is homotopic to the original. Repeating this process, by downward induction on the total number of circles of intersection of all $S_i$ with $S_*$, gives rise to a homotopic path with every vertex corresponding to a sphere that is disjoint from $S_*$, which is in turn null homotopic.
We refer the reader to \cite[Thm~1.1]{HatcherMcCullough1990} or \cite[Thm.~3.1]{HatcherWahl2005} for further details.
\end{proof}

We have now established (a).

\subsubsection{The quotient of $S(M)$}

We now work to establish (b), i.e.\ that the quotient $S(M)/(\pi_0\Diff(M))$ is finite.  For this it is useful to work with the explicit model for our manifold $M$, as in \cref{sec-reducible}.  Let $M=P_1\#\cdots\#P_n\#(S^2\times S^1)^{\#g}$ be the prime decomposition of $M$ and fix a 3-ball $D_i\subseteq P_i$ for each $i$. The explicit model for $M$ is given by taking a 3-sphere with $n+2g$ open 3-balls removed, $B$, and attaching $P_i\setminus \mathring{D}_i$ to the first $n$ boundary components of $B$ and $g$ copies of $S^2\times I$ to the remaining $2g$ boundary components of $B$.

 Slide diffeomorphisms (see \cref{sec-reducible})  allowed Hatcher--McCullough to prove the following, which is (b).

    \begin{proposition}[{\cite[Prop.~2.2]{HatcherMcCullough1990}}]
        \label{quotientfinite}
            The quotient $S(M)/\pi_0\Diff(M)$ is finite.
    \end{proposition}

\begin{proof}[Proof sketch]
Let $\Sigma$ be an embedded sphere system representing an $n$-simplex in $S(M)$. Using the model for $M$ above, Hatcher--McCullough show that there is a diffeomorphism $\varphi \in \Diff(M)$ given by a composition of slide diffeomorphisms, such that $\varphi(\Sigma)\subseteq B$.  Since there are only finitely many isotopy classes of embedded 2-spheres in $B$, specified by partitions of the boundary components, it follows that the action on $S(M)$ by $\pi_0\Diff(M)$ has finitely many orbits.
\end{proof}

 \subsubsection{Preliminaries for (c)}  We now work towards establishing (c), which is the trickiest of the three properties.  To that end we start by establishing some preliminary results.

    \begin{lemma}\label{lem:sesfiniteness}
        Let
        \[
        1\to N\to G\to Q\to 1
        \]
        be a short exact sequence of groups.  If $N$ is finitely generated and $G$ is finitely presented then $Q$ is finitely presented.  If $N$ is finitely presented then $G$ is finitely presented if and only if $Q$ is finitely presented.
    \end{lemma}

   Next one observes  that we do not need to consider the case where $\partial M$ contains any 2-spheres, by the following proposition.  

\begin{proposition}[{\cite[Prop.~2.3]{HatcherMcCullough1990}}]\label{prop-MCG finite presentation no boundary 2-spheres}
    Let $M$ be a compact 3-manifold, and suppose $S$ is a 2-sphere boundary component of $M$. Let $\widehat{M}$ be the manifold obtained from $M$ by filling the $S^2$-boundary with a 3-ball $E$. Then $\pi_0\mathrm{Diff}(M,S)$ is finitely presented if and only if $\pi_0\mathrm{Diff}(\widehat{M})$ is finitely presented.
\end{proposition}

\begin{proof}[Proof sketch]
    Fix a point $e_0\in \mathrm{int} E$. The fibration $\mathrm{Diff}(\widehat{M})\to \widehat{M}$ given by the action of $\Diff(M)$ on $e_0$ induces a long exact sequence and we analyse the last three terms:
\[\pi_1\widehat{M}\overset{\alpha}{\to} \pi_0\mathrm{Diff}(\widehat{M},e_0)\cong \pi_0\mathrm{Diff}(M,S)\to \pi_0\mathrm{Diff}(\widehat{M})\to 1.\]
Since $\widehat{M}$ is compact, $\pi_1\widehat{M}$ is finitely presented.
Furthermore, the kernel of $\alpha$ is central in $\pi_1(\widehat{M})$, so it is a finitely generated abelian group. The proof is completed by appealing to \Cref{lem:sesfiniteness} twice, first using the short exact sequence
\[
1\to \ker(\alpha)\to \pi_1\widehat{M}\to \pi_1\widehat{M}/\ker(\alpha) \to 1
\]
and then again using the short exact sequence
\[
1\to \pi_1\widehat{M}/\ker(\alpha) \to \pi_0\mathrm{Diff}(M,S)\to \pi_0\mathrm{Diff}(\widehat{M})\to 1.\qedhere
\]
\end{proof}

Recall that the group of isotopy classes of sphere twists (\cref{def:sphere twist}) is denoted by $\mathrm{Twist}(M)$.
The group $\Twist(M)$ fixes the homotopy class of any loop in $M$, so its action on $\pi_1(M)$ is trivial. In fact, so is its action on the sphere complex.

    \begin{proposition}
    \label{twisttriv}
      The group $\mathrm{Twist}(M)$ acts trivially on $S(M)$. Therefore, the action of $\pi_0\mathrm{Diff}(M)$ on $S(M)$ induces an action of $\pi_0\mathrm{Diff}(M)/\mathrm{Twist}(M)$ on $S(M)$.
    \end{proposition}

    \begin{proof}[Proof sketch]
       This follows from \cite[Lemma 3.1.1]{McCullough1985}. Let $S$ be an essential sphere in $M$. It suffices to show that $\mathrm{Twist}(M)$ acts trivially on $S$. Let $T_{S_0}\in \mathrm{Twist}(M)$ be a sphere twist about a sphere $S_0$. The intersection $S\cap S_0$ consists of finitely many circles. Performing surgery along an innermost circle in $S_0$ yields two spheres $S_1$ and $S_2$ such that $S\cap S_1=\emptyset$ and $S\cap S_2$ has one fewer intersection circle than $S\cap S_0$. Furthermore the sphere twist $T_{S_0}$ can be decomposed as a product of two sphere twists $T_{S_1}T_{S_2}$. Replicating this argument gives a decomposition of $T_{S_0}$ into a product of sphere twists around spheres all disjoint from $S$, and so $T_{S_0}$ acts trivially on $S$ as required.
    \end{proof}

    \begin{corollary}\label{cor: stabiliser FP iff stab/twist is}
    For each simplex $\sigma\in S(M)$ the stabiliser $\mathrm{Stab}_{\pi_0\mathrm{Diff}(M)}(\sigma)$ is finitely presented if and only if $\mathrm{Stab}_{\pi_0\mathrm{Diff}(M)/\mathrm{Twist}(M)}(\sigma)$ is finitely presented.
    \end{corollary}

    \begin{proof}
        Let $\sigma\in S(M)$ be a simplex.  By \Cref{twisttriv} we have a short exact sequence
        \[
        1\to \mathrm{Twist}(M)\to \mathrm{Stab}_{\pi_0\mathrm{Diff}(M)}(\sigma)\to \mathrm{Stab}_{\pi_0\mathrm{Diff}(M)/\mathrm{Twist}(M)}(\sigma)\to 1.
        \]
        Now note that $\Twist(M)\cong (\Z/2)^{d}$ is finitely presented.  Hence by \Cref{lem:sesfiniteness} the result follows.
    \end{proof}

\subsubsection{Finiteness of stabilisers}  We now establish (c), completing the proof of \Cref{thm-finite presentation}.

By \Cref{prop-MCG finite presentation no boundary 2-spheres} we may assume that $M$ has no 2-sphere boundary components. Let $\sigma\in S(M)$ represented by the sphere system $\Sigma$. Cutting $M$ along $\Sigma$ decomposes $M$ into a sequence of submanifolds $M_1,\dots,M_m$ (here by cutting along $\Sigma$ we mean the compact manifold whose interior is $M\setminus \Sigma$). Since we assumed that there are no boundary 2-spheres, each sphere in $\Sigma$ corresponds to two boundary spheres after cutting.

 Let $\pi_0\overline{\mathrm{Diff}_{\pi_0\partial}(M_j)}$ denote the subgroup generated by elements of \[\pi_0\mathrm{Diff}(M_j)/\mathrm{Twist}(M_j)\] that send components of $\partial M_j$ to themselves, and restrict to degree 1 maps on each 2-sphere boundary component. Define a homomorphism
\begin{equation}\label{eqn:iSigma}
    i_\Sigma\colon \prod_j \pi_0\overline{\mathrm{Diff}_{\pi_0\partial}(M_j)}\to \pi_0\mathrm{Diff}(M)/{\mathrm{Twist}(M)}
\end{equation}
by choosing representatives that restrict to the identity on each spherical boundary component and fitting them together to form a diffeomorphism of $M$. More precisely, for each $m$-tuple of elements represented by $(\phi_1,\dots,\phi_m)$, we can assume that $\phi_j$ restricts to the identity on each spherical component so they can be glued together along these sphere components using the identity map to obtain a diffeomorphism of $M$.

\begin{lemma}[{\cite[Lemma 3.4]{HatcherMcCullough1990}}]\label{lem:isigma_injective}
    The map $i_\Sigma$ is injective.
\end{lemma}

\begin{proof}[Proof sketch]
If $(\phi_1,\dots,\phi_m)\in \ker(i_\Sigma)$, let $\phi$ be the diffeomorphism of $M$ obtained by gluing together representatives of the mapping classes $\phi_j$ (chosen to restrict to the identity on each spherical boundary component). Then $\phi$ is isotopic to a product of sphere twists in $\Twist(M)$ which by \cref{twisttriv} can be assumed to be about spheres disjoint from $\Sigma$, i.e.~each sphere twist lies in $\Twist(M_j)$ for some $j$. Since $\phi_j \in \pi_0\overline{\mathrm{Diff}_{\pi_0\partial}(M_j)}$, we can change $(\phi_1,\dots,\phi_m)$ by these rotations and assume $\phi$ is isotopic to the identity. Hatcher--McCullough show, using {\cite[Thm.~2]{HendriksMcCullough1987}} and \cite{Laudenbach1973}, that there exists an isotopy that preserves $\Sigma$, so each $\phi_j$ is isotopic to the identity as required.
\end{proof}

\begin{proof}[Proof of \cref{thm-finite presentation}]

Since we have already shown (a) and (b) (\Cref{simplyconnectedsim} and \Cref{quotientfinite}), all that is left is to prove (c).  By \Cref{prop-MCG finite presentation no boundary 2-spheres} it suffices to prove this when $M$ has no spherical boundary components or, equivalently, that $M$ has no 3-ball prime summands.  Furthermore, we can assume that $M$ has at least two prime summands, since the result is vacuous for $M$ irreducible, and we have already shown in \cref{thm:S^1xS^2_MCG} that the mapping class group of $S^1\times S^2$ is finitely generated.  By \cref{cor: stabiliser FP iff stab/twist is}, it suffices to establish (c) for the action of $\pi_0\mathrm{Diff}(M)/\mathrm{Twist}(M)$ on $S(M)$.

Let $\sigma$ be represented by the sphere system $\Sigma$ and consider the map $i_\Sigma$ from \eqref{eqn:iSigma}. By construction the image of $i_\Sigma$ lies in $\mathrm{Stab}_{\pi_0\mathrm{Diff}(M)/\mathrm{Twist}(M)}(\sigma)$ and we claim it has finite index. To show this consider $h$ in the stabiliser. Then up to isotopy $h$ can be chosen such that $h(\Sigma)=\Sigma$ \cite[Lemma 3.1]{HatcherMcCullough1990}. Passing to a finite index subgroup of the stabiliser we can additionally require that $h$ preserves the orientation of each sphere in $\Sigma$ and does not reverse sides. It follows that $h$ preserves each $M_j$ and is isotopic to the identity on each boundary sphere, so $h$ is in the image of $i_\Sigma$. By \Cref{lem:isigma_injective} it is therefore enough to show that the domain of $i_\Sigma$ is finitely presented, i.e.\ for each $j$ we claim that $\pi_0\overline{\mathrm{Diff}_{\pi_0\partial}(M_j)}$ is finitely presented. This follows by induction on the number of prime factors of $M$, since the spheres of $\Sigma$ are essential, so $M_j$ has fewer prime factors than $M$. Hence by induction $\pi_0{\mathrm{Diff}(M_j)}$ is finitely presented and hence $\pi_0\overline{\mathrm{Diff}_{\pi_0\partial}(M_j)}$ is also finitely presented (since $\Twist(M)$ is an abelian normal subgroup). This establishes (c) and thus, by \Cref{thm-kbrown}, the proof of \Cref{thm-finite presentation} is complete.
\end{proof}

\section{Further topics
}
Other than direct computation of mapping class groups, there are many other interesting questions one can ask about diffeomorphism groups and their topology. In this section we showcase a sample of these and survey known results.

\subsection{Homological stability}\label{subsec-HS of MCGs}

We give an overview of a homological stability result by Hatcher and Wahl \cite{Hatcher_Wahl_MCG3Manifold}.

\begin{definition}
A sequence of groups and homomorphisms
$$\begin{tikzcd}
    G_0\ar{r}{\phi_0}&G_1\ar{r}{\phi_1}&G_2\ar{r}{\phi_2}& \ldots
\end{tikzcd}$$
is said to satisfy \textit{homological stability} if, for all \(k\), the induced maps \(H_k(G_n)\xrightarrow{\cong} H_k(G_{n+1})\) are isomorphisms for \(n\) sufficiently large with respect to \(k\).
\end{definition}

Knowing that a family of groups exhibits homological stability is an invaluable tool for examining the general behaviour of the entire family. It is most useful when used in conjunction with stable homology.

\begin{definition}
    Let \(\{G_n\}\) be a sequence of groups with associated inclusions. Then \(G_\infty \coloneqq \bigcup_i^\infty G_i\) is defined to be the limit of these groups, and we say that \(H_k(G_\infty)\) is the \textit{stable homology} of \(\{G_n\}\).
\end{definition}

Thus homological stability tells us that \(H_k(G_n) \cong H_k(G_\infty)\) in a range of degrees which increases as \(n\) increases.

Many sequences of groups and spaces are known to satisfy homological stability -- see \cite{Wahl2023} for a comprehensive introduction. Notably, Harer proved homological stability for mapping class groups of surfaces \cite{Harer1985}, and Madsen--Weiss computed the stable homology, showing that stably the rational cohomology ring is a polynomial algebra on Miller--Morita--Mumford classes, as conjectured by Mumford~\cite{MadsenWeiss2007}.

Let \(N\) be a compact, oriented \(3\)-manifold with nonempty boundary.
In \cite{Hatcher_Wahl_MCG3Manifold}, Hatcher and Wahl consider \(3\)-manifolds of the form \(N_n^P:= N\# P^{\# n}\), where \(P\) is compact and oriented.
Choose a component \(\partial_0 N \subseteq\partial N\), and fix a compact subsurface \(R\) of \(\partial N\) which contains \(\partial_0 N\). Now consider the mapping class group of diffeomorphisms of \(N_n^P\) that fix \(R\) pointwise, $\pi_0 \Diff_R(N_n^P)$ (denoted $\Gamma_n^P(N,R)$ in \cite{Hatcher_Wahl_MCG3Manifold}). Since $R$ is fixed pointwise, the mapping classes under consideration are orientation-preserving.
We can obtain \(N_{n+1}^P\) from \(N_n^P\) by taking a copy of \(P\), removing a disc, and identifying the resulting boundary sphere with a disc in \(\partial_0 N\). This gives an inclusion map \(N_n^P \hookrightarrow N_{n+1}^P\). Via this inclusion, we can extend diffeomorphisms on \(N_n^P\) to diffeomorphisms on \(N_{n+1}^P\) by taking them to be the identity away from the image of the inclusion. This induces a map on mapping class groups
\[\phi_n\colon\pi_0 \Diff_R(N_n^P) \to \pi_0 \Diff_R(N_{n+1}^P).
\]
The natural inclusion of the spaces \(N_n^P \hookrightarrow N_{n+1}^P\) imply that the maps \(\phi_n\) are also inclusions by \cite[Proposition~\(2.3\)]{Hatcher_Wahl_MCG3Manifold}.
Hatcher and Wahl proved that this sequence exhibits homological stability.

\begin{theorem}[{\cite[Theorem \(1.1\)]{Hatcher_Wahl_MCG3Manifold}}]\label{thm:3-Manifold-Stab}
    For any compact, oriented \(3\)-manifolds \(N\) and \(P\), and compact subsurface \(R\) of \(\partial N\) as above, the sequence \(\{\pi_0 \Diff_R(N_n^P), \phi_n\}\) satisfies homological stability, i.e. \[H_k(\pi_0 \Diff_R(N_n^P)) \xrightarrow{\cong} H_k(\pi_0 \Diff_R(N_{n+1}^P))\] for \(n>2k+2\).
\end{theorem}
The question of calculating the stable homology of this family of mapping class groups is still open. For the moduli space of the full diffeomorphism group, an analogous homological stability result appears in the thesis of Lam \cite{Lam2015}.

\subsection{Nielsen realisation}\label{subsec-Nielsen}
We introduce the Nielsen realisation problem in its various forms, and survey results for geometric, irreducible and reducible 3-manifolds.

\subsubsection{Formulating the question}

There are many different incarnations of Nielsen realisation and we restrict our attention to the version concerned with mapping class groups of closed, orientable 3-manifolds.
In this context, a finite subgroup $G$ of $\pi_0(\Diff(M))$ is called \emph{realisable} if there exists a lift of $G$ to a subgroup of $\Diff(M)$ as illustrated below.
\[
\xymatrix@R=3em@C=3em{
  & \Diff(M) \ar[d] \\
  G \ar@{_{(}->}[r] \ar@{-->}[ur]^{\exists ?} & \pi_0\Diff(M)
}
\]
The Nielsen realisation problem asks which finite subgroups $G \leq  \pi_0\Diff(M)$ are realisable.

It was answered for surfaces by Kerckhoff~\cite{Kerckhoff1983} in the 1980s, where he showed that every finite subgroup of the mapping class group of a closed hyperbolic surface admits a lift. It was this paper that first used the term `Nielsen realisation theorem', citing a problem posed by Nielsen in his 1942 paper~\cite{Nielsen1942}, which asked whether every finite cyclic subgroup of the mapping class group of a surface is realisable. It is worth noting that Nielsen worked on and posed many similar questions, explaining why the modern interpretation of `Nielsen realisation' takes many forms.

In comparison to the case for surfaces, in higher dimensions it is known that not all subgroups of the mapping class group are realisable. Indeed Raymond and Scott \cite{RaymondScott1997} construct torus bundles in dimension $2k-1$ ($k\geq 2$) for which order $k$ cyclic subgroups of the mapping class group are not realisable. We restrict ourselves to 3-manifolds.

Closer to the original formulation of Nielsen, we may also consider the map $\pi_0\Diff(M)\to \Out(\pi_1(M))$ from \eqref{eqn-MGC to out}. For finite subgroups in the image of this map, one may also ask if they are realisable by a subgroup of $\Diff(M)$. For the many families of 3-manifolds for which $\pi_0\Diff(M)\cong \Out(\pi_1(M))$, these formulations are equivalent.

The question can also be asked for $G$ a finite subgroup of the topological mapping class group $\pi_0\Homeo(M)$, where a subgroup is realisable if there is a lift to $\Homeo(M)$. By combining work of Pardon \cite{Pardon2021} and Edwards--Kirby \cite{EdwardsKirby1971} this is equivalent to the smooth realisation problem for orientable 3-manifolds (see \cite[Thm 1.2]{ChenTshishiku2025}).

\subsubsection{Nielsen realisation for irreducible 3-manifolds}
For irreducible 3-manifolds the literature often focuses on the realisation problem for finite subgroups of diffeomorphisms up to homotopy, e.g.~\cite[Section 2]{Pardon2021}. In the cases where homotopy implies isotopy (see \cref{ss:sym-M}) this is equivalent to considering finite subgroups of the mapping class group.

\emph{Haken Seifert fibred.} In this case, $\pi_0\Diff(M)\cong \Out(\pi_1(M))$ (\cref{thm:Waldhausen}) and for $G$ a finite subgroup of $\Out(\pi_1(M))$ there is a complete obstruction to realisation in $H^3(G, Z(\pi_1(M))$. This is because a lift of $G$ corresponds to the existence of an extension of $\pi_1(M)$ by $G$ corresponding to the abstract kernel of $\iota\colon G\hookrightarrow \Out(\pi_1(M))$. When $M$ is Haken and Seifert fibred, this obstruction is shown to be trivial by Heil-Tollefson~\cite{HeilTollefson1978} for $G\cong \Z / 2$ and Zieschang--Zimmerman~\cite{ZieschangZimmermann1979} in the case of general finite $G$. The proof of Zieschang--Zimmerman uses the theory of crystallographic groups on the universal cover of $M$.

\emph{Haken non-Seifert-fibred.} In this case Nielsen realisation holds for all finite subgroups of the mapping class group by work of Zimmerman \cite{Zimmermann1982}. His proof combines the JSJ decomposition with results of Gabai--Meyerhoff--Thurston for hyperbolic manifolds (\cref{sec-hyperbolic}), and is an extension of the techniques in his work with Zieschang, and Bass--Serre theory.

In the case where $M$ is irreducible but not Haken, it follows that the JSJ decomposition is trivial and thus $M$ is geometric.

\emph{Hyperbolic.} In this case the diffeomorphism group is discrete by Gabai's theorem that $\Diff_0(M)\simeq \{*\}$ \cite{Gabai1997}. It follows that $\Diff(M)\cong\pi_0\Diff(M)$ so every subgroup of the mapping class group is realisable.

\emph{Elliptic.} McCullough showed that the natural isomorphism from the isometry group $\Isom(M)\to \pi_0\Diff(M)$ has a section \cite[Thm 3.4]{McCullough2002}. Therefore in this case all subgroups of the mapping class group are realisable by isometries.\\

The irreducible manifolds not yet covered are non-Haken Seifert fibred spaces with infinite fundamental groups and the Nielsen realisation problem as stated above cannot be recovered from the literature for these manifolds. They are aspherical, so it is plausible that the obstruction from the Haken case will be a complete obstruction once again. The generalised Smale conjecture holds by McCullough--Soma \cite[Thm.~9.3]{McCulloghSoma2013}, 
so the result would also follow from realisability on the level of isometry groups. Moreover  by Meeks--Scott we have that any finite group action on $M$ leaves invariant some Riemannian metric coming from one of the eight geometries of geometrisation \cite[Thm.~2.1]{MeeksScott1986}. This can be viewed as a geometric variant of the Nielsen realisation problem.

\subsubsection{Nielsen realisation for reducible 3-manifolds}
More recently, there has been work on Nielsen realisation for reducible 3 manifolds. Recall that when $M$ is reducible we have the short exact sequence
\begin{equation*}
    1 \to \Twist(M) \to \pi_0\Diff(M) \to \Out(\pi_1(M))
\end{equation*}
where $\Twist(M)$ is generated by twists around embedded 2-spheres and is diffeomorphic to $(\mathbb{Z}/2)^d$ for some $d$ (\cref{rem: twist subgroup}).

Chen and Tshishiku~\cite{ChenTshishiku2025} proved the following realisation theorem for finite subgroups generated by sphere twists in $\Twist(M)$.

\begin{theorem}{\cite[Main Theorem]{ChenTshishiku2025}}
     Let $M$ be a closed, oriented 3-manifold M and $G$ a non-trivial subgroup of $\Twist(M)$. Then $G$ is realisable by diffeomorphisms if and only if $G$ is cyclic and $M$ is diffeomorphic to a connected sum of lens spaces.
\end{theorem}

Since $\Twist(M)\cong (\mathbb{Z} /2)^d$ it follows that a single $\mathbb{Z}/2$ summand may be realisable, depending on the prime components of $M$, but a finite subgroup generated by commuting twists can never be realised.

For reducible manifolds, the Nielsen realisation problem for finite subgroups of $\Out(\pi_1(M))$ that lie in the image of $\pi_0\Diff(M)$ is still open.

 \bibliographystyle{alpha}
 \bibliography{literature}

\end{document}